\documentclass[11pt]{article}
\usepackage{graphicx}
\usepackage{amsmath,amsthm,amssymb}
\usepackage{xypic}
\usepackage{mathabx}
\usepackage{bbm}
\usepackage{color}
\usepackage{enumerate}
\usepackage[hidelinks]{hyperref}
\def\smallint{\begingroup\textstyle \int\endgroup}

\makeatletter
\newtheorem*{rep@theorem}{\rep@title}
\newcommand{\newreptheorem}[2]{
\newenvironment{rep#1}[1]{
 \def\rep@title{#2 \ref{##1}}
 \begin{rep@theorem}}
 {\end{rep@theorem}}}
\makeatother

\newtheorem{theorem}{Theorem}[section]
\newtheorem{proposition}[theorem]{Proposition}
\newtheorem{definition}[theorem]{Definition}
\newtheorem{example}[theorem]{Example}
\newtheorem{conjecture}[theorem]{Conjecture}
\newtheorem{remark}[theorem]{Remark}
\newtheorem{corollary}[theorem]{Corollary}
\newtheorem{lemma}[theorem]{Lemma}
\newtheorem{warning}[theorem]{Warning}

\newreptheorem{corollary}{Corollary}
\newreptheorem{theorem}{Theorem}
\newreptheorem{lemma}{Lemma}
\newcommand{\oast}{\circledast}

\newcommand\blfootnote[1]{
  \begingroup
  \renewcommand\thefootnote{}\footnote{#1}
  \addtocounter{footnote}{-1}
  \endgroup
}

\begin{document}
\title{Enriched $\infty$-categories I: enriched presheaves}
\author{John D. Berman}
\maketitle

\begin{abstract}\blfootnote{This work was supported by a National Science Foundation Postdoctoral Fellowship under grant 1803089.}
This is the first of a series of papers on enriched $\infty$-categories, seeking to reduce enriched higher category theory to the higher algebra of presentable $\infty$-categories, which is better understood and can be approached via universal properties.

In this paper, we introduce enriched presheaves on an enriched $\infty$-category. We prove analogues of most familiar properties of presheaves. For example, we compute limits and colimits of presheaves, prove that all presheaves are colimits of representable presheaves, and prove a version of the Yoneda lemma.
\end{abstract}

\tableofcontents

\addtocontents{toc}{\protect\setcounter{tocdepth}{1}}
\section{Introduction}\label{S1}
\noindent There has been explosive progress in higher category theory in the last decade, largely made possible by Lurie's books Higher Topos Theory \cite{HTT} and Higher Algebra \cite{HA}. Although the foundations of the subject are combinatorial and notoriously technical, a toolbox of techniques involving presentable $\infty$-categories allows us to reduce many problems to universal properties which better agree with intuition from ordinary category theory.

If $\mathcal{C}$ is an $\infty$-category, the $\infty$-category of presheaves $\text{PSh}(\mathcal{C})=\text{Fun}(\mathcal{C}^\text{op},\text{Top})$ is presentable, in the sense that it admits all limits and colimits, and by the Yoneda lemma, $\mathcal{C}$ is the full subcategory of $\text{PSh}(\mathcal{C})$ spanned by representable presheaves.

Moreover, $\text{PSh}(\mathcal{C})$ is \emph{freely generated} by representables under colimits. For example, there are equivalences of $\infty$-categories $$\text{Fun}(\mathcal{C},\mathcal{D})\cong\text{Fun}^L_\star(\text{PSh}(\mathcal{C}),\text{PSh}(\mathcal{D})),$$ where $\text{Fun}^L_\star$ denotes colimit-preserving functors which send representables to representables. In other words, the $\infty$-category Cat of $\infty$-categories embeds as a full subcategory of $\text{Pr}^L_\star$ (presentable $\infty$-categories with some distinguished objects) via the assignment $\mathcal{C}\mapsto\text{PSh}(\mathcal{C})$.

In this way, any question about $\infty$-categories can be reduced to a question about presentable $\infty$-categories. There are many reasons one may wish to make this reduction. For example, presentable $\infty$-categories can often by modeled by model categories.

Another example comes from the higher algebra of presentable $\infty$-categories (introduced in \cite{HA} Section 4.8.1; also see \cite{GGN}). If $\mathcal{C}$ and $\mathcal{D}$ are presentable $\infty$-categories, there is a presentable $\infty$-category $\mathcal{C}\otimes\mathcal{D}$ satisfying the universal property: A colimit-preserving functor $\mathcal{C}\otimes\mathcal{D}\to\mathcal{E}$ classifies a functor $\mathcal{C}\times\mathcal{D}\to\mathcal{E}$ which preserves colimits in either variable independently.

In this way, $\text{Pr}^L$ is a symmetric monoidal $\infty$-category. To give some $\mathcal{C}\in\text{Pr}^L$ an algebra structure with respect to this external tensor product $\otimes$ is precisely to endow $\mathcal{C}$ with a closed monoidal operation. Hence, the higher algebra of $\text{Pr}^L$ is a powerful tool for constructing monoidal operations; for example, Lurie uses it to construct a symmetric monoidal smash product of spectra, which was considered a difficult problem in homotopy theory, even after it was first solved by Elmendorf-Kriz-Mandell-May \cite{EKMM}. \\

\noindent Our goal in this series of papers is study enriched higher category theory by reducing to the theory of presentable $\infty$-categories. Hopefully, this will allow enriched $\infty$-categories to be used more effectively in subjects that rely on them, such as secondary or iterated algebraic K-theory (see \cite{HSS}) or algebraic K-theory of analytic rings (see \cite{Scholze}).

Enriched higher category theory was developed by Gepner-Haugseng \cite{GH}, and a form of the enriched Yoneda lemma has been proven by Hinich \cite{Hinich}, but the theory so far is highly technical and has been difficult to use. Nonetheless, it is conjectured by experts that:

\begin{conjecture}\label{Conj}
If $\mathcal{V}$ is a presentable, closed monoidal $\infty$-category, the $\infty$-category $\text{Cat}^\mathcal{V}$ of $\mathcal{V}$-enriched categories embeds as a full subcategory of $\text{RMod}_\mathcal{V}(\text{Pr}^L)_\star$, the $\infty$-category of presentable right $\mathcal{V}$-modules with some distinguished objects.
\end{conjecture}

\noindent In other words, if $\mathcal{C}$ is a $\mathcal{V}$-enriched category, there should be an $\infty$-category $\text{PSh}^\mathcal{V}(\mathcal{C})$ of $\mathcal{V}$-enriched presheaves. Moreover, $\text{PSh}^\mathcal{V}(\mathcal{C})$ should be not only presentable but also right tensored over $\mathcal{V}$ via a functor $$\otimes:\text{PSh}^\mathcal{V}(\mathcal{C})\times\mathcal{V}\to\text{PSh}^\mathcal{V}(\mathcal{C}),$$ given roughly by the formula $(\mathcal{F}\otimes A)(X)=\mathcal{F}(X)\otimes A$.

The conjecture asserts that $\text{PSh}^\mathcal{V}(\mathcal{C})$ contains all of the information of the enriched category $\mathcal{C}$. This would completely reduce enriched higher category theory to the higher algebra of presentable $\infty$-categories, which has been well-studied over the past decade.

This is a series of papers aimed towards proving the conjecture. In this first paper, we introduce and study the $\infty$-categories $\text{PSh}^\mathcal{V}(\mathcal{C})$ of enriched presheaves. We think of such an enriched presheaf informally as a contravariant $\mathcal{V}$-enriched functor $\mathcal{C}^\text{op}\to\mathcal{V}$. We prove:

\begin{theorem}\label{PShPrL}
If $\mathcal{V}$ is presentable and closed monoidal, and $\mathcal{C}$ is a $\mathcal{V}$-enriched category, then $\text{PSh}^\mathcal{V}(\mathcal{C})$ is a presentable $\infty$-category.
\end{theorem}

\begin{theorem}\label{PropRMPSh}
If $\mathcal{F}\in\text{PSh}^\mathcal{V}(\mathcal{C})$ and $A\in\mathcal{V}$, let $\mathcal{F}\otimes A$ denote the presheaf roughly given by $(\mathcal{F}\otimes A)(X)=\mathcal{F}(X)\otimes A$. This $\otimes$ makes $\text{PSh}^\mathcal{V}(\mathcal{C})$ a presentable right $\mathcal{V}$-module.
\end{theorem}

\noindent In order to understand a presentable $\infty$-category, we generally ask two questions: How can we compute limits and colimits? What is a set of objects which generates the $\infty$-category under colimits? We completely answer both these questions:

\begin{corollary}\label{CorPFib1}
The limit and colimit of $F:I\to\text{PSh}^\mathcal{V}(\mathcal{C})$ are given by $(\text{lim}_{\text{PSh}^\mathcal{V}(\mathcal{C})}F)(X)=\text{lim}_\mathcal{V}(FX)$ and $(\text{colim}_{\text{PSh}^\mathcal{V}(\mathcal{C})}F)(X)=\text{colim}_\mathcal{V}(FX)$.
\end{corollary}

\begin{corollary}\label{CorFreeCo2}
As a presentable right $\mathcal{V}$-module, $\text{PSh}^\mathcal{V}(\mathcal{C})$ is generated by the representable presheaves. That is, every presheaf is a colimit of presheaves of the form $\text{rep}_X\otimes A$, where $X\in\mathcal{C}$ and $A\in\mathcal{V}$.
\end{corollary}

\noindent It is illustrative to compare to the ordinary theory of $\infty$-categories, which are categories enriched in the $\infty$-category Top of spaces. In this case, we have $\text{PSh}(\mathcal{C})=\text{Fun}(\mathcal{C}^\text{op},\text{Top})$. However, it is also known that if $\mathcal{M}$ is any presentable $\infty$-category, $\text{Fun}(\mathcal{C}^\text{op},\mathcal{M})\cong\text{PSh}(\mathcal{C})\otimes\mathcal{M}$.

We will prove an enriched version of this equivalence by constructing an $\infty$-category $\text{PSh}^\mathcal{V}(\mathcal{C};\mathcal{M})$ of presheaves \emph{with values in $\mathcal{M}$}. This construction makes sense provided that $\mathcal{M}$ is a left $\mathcal{V}$-module.

\begin{theorem}\label{ThmS5}
If $\mathcal{M}$ is a presentable left $\mathcal{V}$-module, then $$\text{PSh}^\mathcal{V}(\mathcal{C};\mathcal{M})\cong\text{PSh}^\mathcal{V}(\mathcal{C})\otimes_\mathcal{V}\mathcal{M}.$$
\end{theorem}

\noindent We have already announced that $\text{PSh}^\mathcal{V}(\mathcal{C})$ is generated by representable presheaves. In ordinary category theory, something much stronger is true: $\text{PSh}(\mathcal{C})$ is \emph{freely} generated by representable presheaves, in the sense that there are equivalences of $\infty$-categories $\text{Fun}^L(\text{PSh}(\mathcal{C}),\mathcal{D})\cong\text{Fun}(\mathcal{C},\mathcal{D})$.

Our last main theorem will be a generalization of this statement to enriched category theory. In order to make the statement, we introduce another variant on the enriched presheaf construction.

That is, we introduce an $\infty$-category $\text{coPSh}^\mathcal{V}(\mathcal{C})$ of $\mathcal{V}$-enriched \emph{copresheaves}, which are roughly covariant functors $\mathcal{C}\to\mathcal{V}$.

\begin{theorem}\label{ThmS6}
If $\mathcal{V}$ is presentable and closed monoidal, $\text{PSh}^\mathcal{V}(\mathcal{C})$ is dualizable as a right $\mathcal{V}$-module in $\text{Pr}^L$, and its dual is $\text{coPSh}^\mathcal{V}(\mathcal{C})$.
\end{theorem}

\noindent Combining with Theorem \ref{ThmS5}, this implies that $$\text{Fun}^L_{\text{RMod}_\mathcal{V}}(\text{PSh}^\mathcal{V}(\mathcal{C}),\mathcal{D})\cong\text{coPSh}^\mathcal{V}(\mathcal{C};\mathcal{D}),$$ provided $\mathcal{D}$ is a presentable right $\mathcal{V}$-module.

\begin{remark}
If $\mathcal{C}$ is a $\mathcal{V}$-enriched category, there are two natural candidates for the $\infty$-category of \emph{$\mathcal{V}$-enriched functors} $\mathcal{C}\to\mathcal{D}$. On one hand, $\mathcal{V}$-enriched categories form an $(\infty,2)$-category $\text{Cat}^\mathcal{V}$, so therefore we have an $\infty$-category $\text{Fun}^\mathcal{V}(\mathcal{C},\mathcal{D})$ of $\mathcal{V}$-enriched functors whenever $\mathcal{D}$ is a $\mathcal{V}$-enriched category.

On the other hand, we will define an $\infty$-category of $\mathcal{V}$-enriched copresheaves $\text{coPSh}^\mathcal{V}(\mathcal{C};\mathcal{D})$ any time $\mathcal{D}$ is an $\infty$-category right tensored over $\mathcal{V}$ (a right $\mathcal{V}$-module).

It is important to recognize that, although these two constructions convey essentially the same information, one is defined for $\mathcal{V}$-enriched categories $\mathcal{D}$ and the other for right $\mathcal{V}$-module $\infty$-categories $\mathcal{D}$. If $\mathcal{V}$ is presentable and closed monoidal, Gepner-Haugseng have shown that presentable right $\mathcal{V}$-modules $\mathcal{D}$ are canonically $\mathcal{V}$-enriched (\cite{GH} 7.4.13), and in this case we conjecture the two constructions are equivalent.
\end{remark}

\subsection{Methods}\label{S11}
\noindent Our results rest on Gepner-Haugseng's operadic model for enriched $\infty$-categories \cite{GH}. If $\mathcal{V}$ is a monoidal $\infty$-category, we expect a $\mathcal{V}$-enriched category to consist of a set $S$ of objects, along with:
\begin{itemize}
\item objects $\mathcal{C}(X,Y)\in\mathcal{V}$ for each $X,Y\in S$;
\item composition morphisms $\mathcal{C}(X,Y)\otimes\mathcal{C}(Y,Z)\to\mathcal{C}(X,Z)$;
\item identity morphisms $1\to\mathcal{C}(X,X)$;
\item and coherences relating these.
\end{itemize}
\noindent This data is encoded in an $\infty$-operad $\text{Assoc}_S$ (which depends on the set $S$), so that a $\mathcal{V}$-enriched category can be identified with an $\text{Assoc}_S$-algebra in $\mathcal{V}$. We define $\text{Assoc}_S$ in Definition \ref{DefAss} and Remark \ref{RmkAss}; note that the definition is fairly concrete, as $\text{Assoc}_S$ is the nerve of a 1-category.

A $\mathcal{V}$-enriched presheaf on $\mathcal{C}$ should consist of the following in addition to the above data:
\begin{itemize}
\item objects $\mathcal{F}(X)\in\mathcal{V}$ for each $X\in S$;
\item morphisms $\mathcal{C}(X,Y)\otimes\mathcal{F}(Y)\to\mathcal{F}(X)$;
\item and coherences relating these.
\end{itemize}
\noindent This data is also encoded in an $\infty$-operad, which we call $\text{LM}_S$ (Definition \ref{DefLMS}). That is, an $\text{LM}_S$-algebra in $\mathcal{V}$ should be regarded as a pair $(\mathcal{C};\mathcal{F})$, where $\mathcal{C}$ is a $\mathcal{V}$-enriched category, and $\mathcal{F}$ is a $\mathcal{V}$-enriched presheaf on $\mathcal{C}$.

We write $\text{Cat}^\mathcal{V}_S=\text{Alg}_{\text{Assoc}_S}(\mathcal{V})$ and $\text{PSh}^\mathcal{V}_S=\text{Alg}_{\text{LM}_S}(\mathcal{V})$. There is an inclusion $\text{Assoc}_S\subseteq\text{LM}_S$, inducing a forgetful functor $\theta:\text{PSh}^\mathcal{V}_S\to\text{Cat}^\mathcal{V}_S$. If $\mathcal{C}$ is a $\mathcal{V}$-enriched category, we may define $\text{PSh}^\mathcal{V}(\mathcal{C})$ to be the fiber $\theta^{-1}(\mathcal{C})$.

\begin{example}
We will always keep in mind the example $|S|=1$. In this case, $\text{Assoc}_S$ is the usual associative operad Assoc, and $\text{LM}_S$ is the left module operad LM. In other words, an enriched category $\mathcal{C}$ with one object $X$ can be identified with the endomorphism algebra $\mathcal{C}(X,X)$. A presheaf on this enriched category can be identified with a left module over $\mathcal{C}(X,X)$.

In Higher Algebra Chapter 4 \cite{HA}, Lurie studies $\infty$-categories of left modules, which are $\infty$-categories of presheaves on enriched categories with one object. This paper should be viewed as a generalization of that chapter to the case $|S|>1$.
\end{example}

\noindent Defining $\text{PSh}^\mathcal{V}(\mathcal{C})$ as the fiber $\theta^{-1}(\mathcal{C})$ agrees with our intuition, but it is not easy to prove anything about these fibers.

In order to prove our main results, we will need a second model for enriched category theory, which is more explicit although less formally well-behaved. To do this, we construct $\infty$-preoperads $\Delta_{/S}^\text{op}$ and $\Delta_{/S}^\text{op}\times\Delta^1$, where $\Delta^1$ is the diagram category $0\to 1$, and $\Delta_{/S}$ is the category of finite, nonempty, totally ordered sets equipped with a function to $S$. That is, $\Delta_{/S}=\Delta\times_\text{Set}\text{Set}_{/S}$.

\begin{remark}
To describe $\Delta_{/S}^\text{op}$ and $\Delta_{/S}^\text{op}\times\Delta^1$ as $\infty$-preoperads, we also need to equip them with maps to the commutative operad. See Sections \ref{S41}, respectively \ref{S42}.
\end{remark}

\begin{definition}
An $\mathbb{A}_\infty$-$\mathcal{V}$-enriched category (with set $S$ of objects) is a $\Delta_{/S}^\text{op}$-algebra in $\mathcal{V}$. An $\mathbb{A}_\infty$-$\mathcal{V}$-enriched presheaf is a $\Delta_{/S}^\text{op}\times\Delta^1$-algebra in $\mathcal{V}$.
\end{definition}

\noindent We will prove that this $\mathbb{A}_\infty$-model is equivalent to the Gepner-Haugseng model (Corollaries \ref{CorModelEq1} and \ref{CorModelEq2}) by utilizing Lurie's theory of approximations to $\infty$-operads (which is reviewed in Section \ref{S24}).

Because $\Delta_{/S}^\text{op}$ and $\Delta_{/S}^\text{op}\times\Delta^1$ are $\infty$-preoperads rather than $\infty$-operads, the $\mathbb{A}_\infty$-model is not as formally well-behaved as the operadic model. However, it has two benefits.

The first benefit is that $\Delta_{/S}^\text{op}$ is easier to construct than $\text{Assoc}_S$. The second, more significant, benefit is the very close relationship between $\Delta_{/S}^\text{op}$ and $\Delta_{/S}^\text{op}\times\Delta^1$.

That is, an $\mathbb{A}_\infty$-enriched presheaf can be regarded as a map $\mathcal{F}_0\to\mathcal{F}_1$ of functors $\Delta_{/S}^\text{op}\to\smallint\mathcal{V}$, where $\smallint\mathcal{V}$ is the $\infty$-operad associated to $\mathcal{V}$ (see Section \ref{S15} notation (4)). It turns out that $\mathcal{F}_1$ is just the underlying $\mathbb{A}_\infty$-enriched category. Therefore, an $\mathbb{A}_\infty$-enriched presheaf on $\mathcal{C}$ can be regarded as a single functor $\Delta_{/S}^\text{op}\to\smallint\mathcal{V}$ satisfying properties.

This provides us with a new model for $\text{PSh}^\mathcal{V}(\mathcal{C})$, which is more technical but also more concrete than its definition as a fiber of $\theta:\text{PSh}^\mathcal{V}_S\to\text{Cat}^\mathcal{V}_S$. This is Proposition \ref{PropPShModel}.

We will use this new model to conclude Theorem \ref{PShPrL} and Corollary \ref{CorPFib1}. We will also prove (Corollary \ref{Cor1S4}) that $\text{PSh}^\mathcal{V}(\mathcal{C})$ is functorial in $\mathcal{C}$, in the sense of a functor $$\text{PSh}^\mathcal{V}(-):\text{Cat}^\mathcal{V}_S\to\text{Pr}^L.$$

\begin{example}
To get a feeling for the $\mathbb{A}_\infty$-model, suppose that $\mathcal{V}$ is cartesian monoidal (monoidal under its product). Then $\mathcal{V}$-enriched categories $\mathcal{C}$ (with set $S$ of objects) can be identified with functors $\mathcal{C}_{\mathbb{A}_\infty}:\Delta_{/S}^\text{op}\to\mathcal{V}$, in that there is a full subcategory inclusion $\text{Cat}^\mathcal{V}_S\subseteq\text{Fun}(\Delta_{/S}^\text{op},\mathcal{V})$. For $X,Y\in S$, let $[X,Y]\in\Delta_{/S}^\text{op}$ denote the function $[1]\to S$ which sends $0\mapsto X$ and $1\mapsto Y$. Then $\mathcal{C}_{\mathbb{A}_\infty}[X,Y]$ is the mapping object $\mathcal{C}(X,Y)\in\mathcal{V}$.

When $|S|=1$, then the enriched category $\mathcal{C}$ carries the same information as a $\mathcal{V}$-algebra $A=\mathcal{C}(X,X)$, and $\mathcal{C}_{\mathbb{A}_\infty}$ is a simplicial object of $\mathcal{V}$ -- the simplicial object given by the Milnor construction (\cite{HA} 4.1.2.4): $$\xymatrix{
\mathcal{C}_{\mathbb{A}_\infty}= &\cdots A^{\times 2} \ar@<2ex>[r] \ar@<0ex>[r] \ar@<-2ex>[r] &  A \ar@<1ex>[r] \ar@<-1ex>[r] \ar@<1ex>[l] \ar@<-1ex>[l] & 1. \ar@<0ex>[l]
}$$ Given such a $\mathcal{C}_{\mathbb{A}_\infty}:\Delta^\text{op}_{/S}\to\mathcal{V}$, a presheaf on $\mathcal{C}$ can be identified with a second functor $\mathcal{F}_{\mathbb{A}_\infty}:\Delta^\text{op}_{/S}\to\mathcal{V}$ along with a natural transformation $\mathcal{F}_{\mathbb{A}_\infty}\to\mathcal{C}_{\mathbb{A}_\infty}$. We think of this natural transformation as a functor $\Delta^\text{op}_{/S}\times\Delta^1\to\mathcal{V}$.

In the case $|S|=1$, such a presheaf is a left $X$-module $M$, and it is completely specified by a map of simplicial objects: $$\xymatrix{
\quad\mathcal{F}_{\mathbb{A}_\infty}=\ar[d] &\cdots X^{\times 2}\times M \ar[d]\ar@<2ex>[r] \ar@<0ex>[r] \ar@<-2ex>[r] &  X\times M \ar[d]\ar@<1ex>[r] \ar@<-1ex>[r] \ar@<1ex>[l] \ar@<-1ex>[l] &  M \ar[d]\ar@<0ex>[l] \\
\quad\mathcal{C}_{\mathbb{A}_\infty}= &\cdots\quad\,X^{\times 2}\quad\,\, \ar@<2ex>[r] \ar@<0ex>[r] \ar@<-2ex>[r] &  \quad X\quad \ar@<1ex>[r] \ar@<-1ex>[r] \ar@<1ex>[l] \ar@<-1ex>[l] & 1. \ar@<0ex>[l]
}$$ The downward maps are just projection away from $M$; all the information about the algebra and module structures is carried in the horizontal maps.

In general, $\mathbb{A}_\infty$-enriched categories and presheaves can be thought of as Milnor constructions like this, but indexed by the thickened simplex category $\Delta_{/S}^\text{op}$ instead of $\Delta^\text{op}$. We could represent $\mathcal{C}_{\mathbb{A}_\infty}$ pictorially as a simplicial object of $\mathcal{V}$ where the $n$-simplices are given by an $S^2\times\cdots\times S^2$-hypercube of entries $$\xymatrix{
\cdots{\left(\begin{smallmatrix}\mathcal{C}(X_1,X_1)\otimes\mathcal{C}(X_1,X_1) &\cdots &\mathcal{C}(X_1,X_1)\otimes\mathcal{C}(X_n,X_n) \\ \vdots &&\vdots \\ \mathcal{C}(X_n,X_n)\otimes\mathcal{C}(X_1,X_1) &\cdots &\mathcal{C}(X_n,X_n)\otimes\mathcal{C}(X_n,X_n)\end{smallmatrix}\right)} \ar@<2ex>[r] \ar@<0ex>[r] \ar@<-2ex>[r] &  {\left(\begin{smallmatrix}\mathcal{C}(X_1,X_1) \\ \vdots \\ \mathcal{C}(X_n,X_n)\end{smallmatrix}\right)} \ar@<1ex>[r] \ar@<-1ex>[r] \ar@<1ex>[l] \ar@<-1ex>[l] & 1. \ar@<0ex>[l]
}$$ Obviously such a diagram is cumbersome; it is included as a mnemonic, but it won't appear in this paper again.
\end{example}

\noindent Our discussion so far concerns the presheaves on a $\mathcal{V}$-enriched category $\mathcal{C}$ with values in $\mathcal{V}$ (roughly, $\mathcal{V}$-enriched functors $\mathcal{C}^\text{op}\to\mathcal{V}$). There are two variants on this construction.

First, an \emph{enriched copresheaf} on $\mathcal{C}$ is roughly a $\mathcal{V}$-enriched functor $\mathcal{C}\to\mathcal{V}$. The theory of enriched copresheaves is completely parallel to the theory of enriched presheaves; we introduce an $\infty$-operad $\text{RM}_S$, and a copresheaf is an $\text{RM}_S$-algebra.

Let $\mathcal{V}^\text{rev}$ denote the $\infty$-category $\mathcal{V}$ with its reverse monoidal operation described by $X\otimes^\text{rev}Y=Y\otimes X$. Then $\mathcal{C}^\text{op}$ is naturally a $\mathcal{V}^\text{rev}$-enriched category (Example \ref{ExOp}) and $$\text{coPSh}^\mathcal{V}(\mathcal{C})\cong\text{PSh}^{\mathcal{V}^\text{rev}}(\mathcal{C}^\text{op})$$ (Remark \ref{RmkOp}). Therefore, everything we prove about presheaves is also true about copresheaves.

\begin{example}
If $\mathcal{V}$ is symmetric monoidal, then $\mathcal{V}^\text{rev}\cong\mathcal{V}$ as monoidal $\infty$-categories, so $\text{coPSh}^\mathcal{V}(\mathcal{C})\cong\text{PSh}^\mathcal{V}(\mathcal{C}^\text{op})$. This is not true in general.
\end{example}

\noindent Second, a \emph{$\mathcal{V}$-enriched copresheaf} on $\mathcal{C}$ \emph{with values in $\mathcal{M}$} is roughly a $\mathcal{V}$-enriched functor $\mathcal{C}^\text{op}\to\mathcal{M}$. We write $\text{PSh}^\mathcal{V}(\mathcal{C};\mathcal{M})$ for the $\infty$-category thereof.

It turns out that this construction makes sense whenever $\mathcal{M}$ is left tensored over $\mathcal{V}$ (a left $\mathcal{V}$-module in Cat). In this case, we may regard the pair $(\mathcal{V};\mathcal{M})$ as an LM-algebra in Cat, or an LM-monoidal $\infty$-category, and we define a presheaf with values in $\mathcal{M}$ to be an $\text{LM}_S$-algebra in this LM-monoidal $\infty$-category.

In Sections \ref{S5} and \ref{S6}, we take advantage of these constructions to study the interplay of presheaves (left actions) with copresheaves (right actions). The key is a pairing $$\left\langle -,-\right\rangle:\text{LM}_S\times\text{RM}_T\to\text{Assoc}_{S\amalg T}$$ introduced in Section \ref{S52}. Careful analysis of this pairing allows us to:
\begin{itemize}
\item construct the right $\mathcal{V}$-action on $\text{PSh}^\mathcal{V}(\mathcal{C})$ in Section \ref{S52} (Theorem \ref{PropRMPSh}) and prove (Theorem \ref{ThmS5}) $$\text{PSh}^\mathcal{V}(\mathcal{C};\mathcal{M})\cong\text{PSh}^\mathcal{V}(\mathcal{C})\otimes_\mathcal{V}\mathcal{M};$$
\item construct the Yoneda embedding $\mathfrak{Y}\in\text{coPSh}^\mathcal{V}(\mathcal{C};\text{PSh}^\mathcal{V}(\mathcal{C}))$ in Section \ref{S6} and prove that it induces a duality between presheaves and copresheaves (Theorem \ref{ThmS6}).
\end{itemize}

\begin{remark}
Since a copresheaf is like an enriched functor, $\mathfrak{Y}$ can be regarded informally as an enriched functor $\mathcal{C}\to\text{PSh}^\mathcal{V}(\mathcal{C})$. In this sense, $\mathfrak{Y}$ really is the Yoneda embedding.

On the other hand, by Theorems \ref{ThmS5} and \ref{ThmS6} together, $$\text{coPSh}^\mathcal{V}(\mathcal{C};\text{PSh}^\mathcal{V}(\mathcal{C}))\cong\text{PSh}^\mathcal{V}(\mathcal{C})\otimes_\mathcal{V}\text{coPSh}^\mathcal{V}(\mathcal{C}).$$ In this sense, we may regard $\mathfrak{Y}$ as an element of $\text{PSh}^\mathcal{V}(\mathcal{C})\otimes_\mathcal{V}\text{coPSh}^\mathcal{V}(\mathcal{C})$. It is precisely the element which exhibits the duality between $\text{PSh}^\mathcal{V}(\mathcal{C})$ and $\text{coPSh}^\mathcal{V}(\mathcal{C})$.
\end{remark}

\subsection{What not to expect from this paper}\label{S13}
\noindent We study the $\infty$-category $\text{Cat}^\mathcal{V}_S$ whose
\begin{itemize}
\item objects are $\mathcal{V}$-enriched categories $\mathcal{C}$ with set $S$ of objects;
\item morphisms are $\mathcal{V}$-enriched functors $F:\mathcal{C}\to\mathcal{D}$ which act as the identity on the set $S$ of objects, $F(X)=X$.
\end{itemize}
\noindent This construction is functorial in the set $S$, $\text{Cat}^\mathcal{V}_{-}:\text{Set}^\text{op}\to\text{Cat}$, and it is possible to construct an $\infty$-category $\text{Cat}^\mathcal{V}$ of \emph{all} $\mathcal{V}$-enriched categories by taking a sort of oplax colimit over $\text{Cat}^\mathcal{V}_{-}$. This is due to Gepner-Haugseng \cite{GH}, although the idea is older (of first studying enriched categories with a fixed set of objects and then extrapolating). See \cite{Berman0} for the formulation as an oplax colimit.

In this paper, we make a systematic study of $\text{PSh}^\mathcal{V}(\mathcal{C})$ for any fixed $\mathcal{V}$-enriched category $\mathcal{C}$, and we will describe the functoriality as $\mathcal{C}$ varies within $\text{Cat}^\mathcal{V}_S$ (see Section \ref{S43}). However, we will not review the construction of $\text{Cat}^\mathcal{V}$, and we will not make any comparison of $\text{PSh}^\mathcal{V}(\mathcal{C})$ and $\text{PSh}^\mathcal{V}(\mathcal{D})$ when $\mathcal{C}$ and $\mathcal{D}$ have different sets of objects.

The reader should expect such results in a future paper of this series. They do not appear here because they use substantially different techniques (centered around oplax colimits) than the results of this paper (which center around $\mathbb{A}_\infty$-algebras, the pairing $\left\langle -,-\right\rangle:\text{LM}_S\times\text{RM}_T\to\text{Assoc}_{S\amalg T}$, and the Barr-Beck Theorem). Nonetheless, it is instructive to keep some of the expected results in mind. These include Conjecture \ref{Conj} as well as:
\begin{enumerate}
\item We expect a functor $\text{PSh}^\mathcal{V}(-):\text{Cat}^\mathcal{V}\to\text{RMod}_\mathcal{V}(\text{Pr}^L)$; that is, for any $\mathcal{V}$-enriched functor $F:\mathcal{C}\to\mathcal{D}$, we expect to have an adjunction $F_\ast:\text{PSh}^\mathcal{V}(\mathcal{C})\leftrightarrows\text{PSh}^\mathcal{V}(\mathcal{D}):F^\ast$ with $F_\ast$ a $\mathcal{V}$-module functor;
\item If $\mathcal{V}$ is symmetric monoidal, we expect the functor $\text{PSh}^\mathcal{V}(-)$ to be symmetric monoidal; that is, we expect$$\text{PSh}^\mathcal{V}(\mathcal{C}\otimes\mathcal{D})\cong\text{PSh}^\mathcal{V}(\mathcal{C})\otimes_\mathcal{V}\text{PSh}^\mathcal{V}(\mathcal{D});$$
\item If $\text{Fun}^\mathcal{V}$ denotes the $\infty$-category of $\mathcal{V}$-enriched functors, we expect statements of the form $$\text{PSh}^\mathcal{V}(\mathcal{C};\mathcal{M})\cong\text{Fun}^\mathcal{V}(\mathcal{C}^\text{op},\mathcal{M}),$$ $$\text{coPSh}^\mathcal{V}(\mathcal{C};\mathcal{M})\cong\text{Fun}^\mathcal{V}(\mathcal{C},\mathcal{M}).$$
\end{enumerate}

\begin{remark}
Gepner-Haugseng construct $\text{Cat}^\mathcal{V}$ in a slightly different, but equivalent way. That is, they construct an $\infty$-operad $\text{Assoc}_S$ for each space (or $\infty$-groupoid) $S$, interpreting an $\text{Assoc}_S$-algebra as an enriched category with space $S$ of objects. When $S$ is discrete, this recovers our $\text{Assoc}_S$.

Then $\text{Cat}^\mathcal{V}_{-}$ is even functorial $\text{Top}^\text{op}\to\text{Cat}$, and Gepner-Haugseng define $\text{Cat}^\mathcal{V}$ to be the oplax colimit over $\text{Top}^\text{op}$.

However, regardless of whether we take the oplax colimit over $\text{Set}^\text{op}$ or $\text{Top}^\text{op}$, we obtain equivalent $\infty$-categories $\text{Cat}^\mathcal{V}$ by \cite{GH} Theorem 5.3.17.\footnote{Thanks to Aaron Mazel-Gee for pointing out this distinction.}

To summarize, we might use the term `enriched category' for an enriched category with a set of objects, or `enriched $\infty$-category' for an enriched category with a space of objects. Then there is no difference between the theories of $\mathcal{V}$-enriched categories and $\mathcal{V}$-enriched $\infty$-categories.

The reason we choose $S$ to vary over sets, not spaces, is very simple: If $S$ is a set, $\text{Assoc}_S$ is (the nerve of) an ordinary category, which makes the entire theory pleasantly concrete.
\end{remark}

\subsection{Summary}\label{S14}
\noindent We begin in Section \ref{S2} with some review of marked $\infty$-categories and $\infty$-operads. (Our notation for $\infty$-operads differs in some respects from Lurie's, so this section should not be skipped.)

In Section \ref{S3}, we review enriched categories and introduce enriched presheaves. This section should be fairly accessible. We end this section by introducing the representable presheaves $\text{rep}_X$. We prove a key technical result about them (Theorem \ref{ThmFree}), which is best understood through the following corollary:

\begin{repcorollary}{CorYon}
If $\mathcal{F}\in\text{PSh}^\mathcal{V}(\mathcal{C};\mathcal{M})$, $\text{Map}(\text{rep}_X\otimes M,\mathcal{F})\cong\text{Map}(M,\mathcal{F}(X))$.
\end{repcorollary}

\noindent If we take $\mathcal{M}=\mathcal{V}$, this is a kind of enriched Yoneda lemma, in the sense that $\text{Map}(\text{rep}_X\otimes -,\mathcal{F})$, which is a priori a presheaf on $\mathcal{V}$, is just the representable presheaf described by $\mathcal{F}(X)$.

The last three sections are more technical. Section \ref{S4} introduces the new $\mathbb{A}_\infty$-model for enriched higher category theory, which we prove is equivalent to Gepner-Haugseng's operadic model. We prove Theorem \ref{PShPrL} and Corollary \ref{CorPFib1}.

In Section \ref{S5}, we construct the right $\mathcal{V}$-module structure on $\text{PSh}^\mathcal{V}(\mathcal{C})$ and prove Theorem \ref{PropRMPSh}, Corollary \ref{CorFreeCo2}, and Theorem \ref{ThmS5}. The hard part is the construction of the $\mathcal{V}$-action; the equivalence is an application of the Barr-Beck Theorem.

In Section \ref{S6}, we construct the Yoneda embedding as a copresheaf $\mathfrak{Y}$ in $\text{coPSh}^\mathcal{V}(\mathcal{C};\text{PSh}^\mathcal{V}(\mathcal{C}))$, and we prove Theorem \ref{ThmS6}. By this point, we will have done most of the work already, so this section is short.

The casual reader who wishes to skim the paper without understanding the proofs is recommended to read Sections \ref{S21}-\ref{S23}, \ref{S3}, \ref{S43}, and the section introductions. The more serious reader should follow the same prescription before returning to Section \ref{S24} and reading from there to fill in the gaps.

\subsection{Acknowledgments}\label{S15}
\noindent The author thanks the National Science Foundation for his time as a postdoctoral fellow, as well as his postdoctoral mentor Andrew Blumberg, who has taken an active interest in this project. He also thanks Rok Gregoric and Rune Haugseng for helpful conversations.

Enriched categories with one object are just associative algebras, and presheaves over them are just left modules. Chapter 4 of Lurie's book Higher Algebra contains all of our results already in the case of enriched categories with one object, and most of the proofs generalize directly. Therefore, the author also thanks Jacob Lurie for many proof techniques, and for doing most of the hard work already.

\subsection{Notation}\label{S16}
\noindent Because of the heavy reliance on Higher Algebra, we consistently cite it as HA rather than \cite{HA}, and cite Higher Topos Theory as HTT rather than \cite{HTT}. We will follow Lurie's notation in most cases, with the following exceptions:
\begin{enumerate}
\item We will write Cat (not $\text{Cat}_\infty$) for the $\infty$-category of $\infty$-categories, similarly dropping $\infty$ notationally anywhere it won't introduce confusion.
\item If $F:\mathcal{C}\to\text{Cat}$ is a functor, we use $\smallint F\to\mathcal{C}$ for the associated cocartesian fibration.
\item We use notation like $\mathcal{O}$ rather than Lurie's $\mathcal{O}^\otimes$ for $\infty$-operads.
\item We regard a monoidal $\infty$-category as a functor $\mathcal{C}:\text{Assoc}\to\text{Cat}$, where Assoc is the associative operad. We write the associated $\infty$-operad $\smallint\mathcal{C}$ rather than Lurie's $\mathcal{C}^\otimes$, because it is the associated cocartesian fibration $\smallint\mathcal{C}\to\text{Assoc}$.
\end{enumerate}

\noindent In other respects, we follow Lurie's notation. For example:
\begin{itemize}
\item We distinguish between Cat, the large $\infty$-category of small $\infty$-categories, and $\widehat{\text{Cat}}$, the very large $\infty$-category of large $\infty$-categories.
\item We denote by $\text{Pr}^L$ the subcategory of $\widehat{\text{Cat}}$ spanned by the presentable $\infty$-categories, along with functors that preserve small colimits (equivalently, admit right adjoints).

$\text{Pr}^L$ has a symmetric monoidal operation $\otimes$, and an algebra in $\text{Pr}^L$ carries the same data as a presentable, closed monoidal $\infty$-category.
\end{itemize}

\addtocontents{toc}{\protect\setcounter{tocdepth}{2}}
\section{Preliminaries}\label{S2}
\noindent In this section we review some preliminaries, beginning with marked $\infty$-categories in Section \ref{S21}. This material follows \cite{Berman0}.

The rest of the section is on $\infty$-operads, following HA (Higher Algebra \cite{HA}). Section \ref{S23} reviews the foundations of $\infty$-operads and $\infty$-preoperads. Then Section \ref{S24} is on operadic approximation, and \ref{S25} is on the $\mathbb{A}_\infty$-model for algebras and modules, an extended example of operadic approximation.

These last two sections are used crucially in Section \ref{S4} to describe limits and colimits in $\infty$-categories of presheaves. However, the casual reader would do better to read Section \ref{S3} first, and only then return to \ref{S24} and \ref{S25}.

\subsection{Marked categories}\label{S21}
\begin{definition}
A \emph{marked $\infty$-category} is an $\infty$-category $\mathcal{C}$ along with a specified collection of morphisms, such that:
\begin{itemize}
\item all equivalences are marked;
\item given two equivalent morphisms $f\cong g$, $f$ is marked if and only if $g$ is;
\item any composite of marked morphisms is marked.
\end{itemize}
\noindent If $\mathcal{C},\mathcal{D}$ are marked $\infty$-categories, a functor $F:\mathcal{C}\rightarrow\mathcal{D}$ is \emph{marked} if it sends marked morphisms to marked morphisms.
\end{definition}

\noindent There is an ($\infty$,2)-category of marked $\infty$-categories and marked functors, which we denote $\text{Cat}^\dag$.

\begin{remark}
There are many ways to construct $\text{Cat}^\dag$:
\begin{itemize}
\item Let $\text{Cat}_1^\dag$ denote the 2-category of marked 1-categories, marked functors, and (all) natural equivalences. A marked $\infty$-category can be specified by a marking on the homotopy category, so $\text{Cat}^\dag=\text{Cat}\times_{\text{Cat}_1}\text{Cat}_1^\dag$. See \cite{Berman0} 2.2 or \cite{HA} 4.1.7.1.
\item $\text{Cat}^\dag$ is equivalent to the full subcategory of $\text{Fun}(\Delta^1,\text{Cat})$ spanned by those functors $\mathcal{C}^\text{mk}\to\mathcal{C}$ which exhibit $\mathcal{C}^\text{mk}$ as a subcategory of marked morphisms. See \cite{BarK} 1.14.
\item $\text{Cat}^\dag$ can be described in terms of a model structure on marked simplicial sets.
\end{itemize}
\end{remark}

\noindent We will often be interested in multiple markings on the same $\infty$-category, so we will use notation like $\mathcal{C}^\dag,\mathcal{C}^\sharp,\mathcal{C}^\natural,\ldots$ to denote markings on $\mathcal{C}$.

The following reference table provides examples of marked $\infty$-categories and establishes notation that we will use throughout the paper.

\begin{tabular}{|c|c|c|}\hline
\textbf{Notation} &\textbf{Marked morphisms} &\textbf{Comments} \\ \hline
$\mathcal{C}^\sharp$ &all morphisms &the \emph{sharp} marking \\ \hline
$\mathcal{C}^\flat$ &equivalences &the \emph{flat} marking \\ \hline
$\mathcal{C}^\natural$ &$p$-(co)cartesian morphisms &for a given (co)cartesian \\
&&fibration $p:\mathcal{C}\to\mathcal{D}$ \\ \hline
$\mathcal{C}^\mathsection$ &inert morphisms &for $\infty$-operads and similar; \\
&&see Definition \ref{DefInert} \\ \hline
$\mathcal{C}^\ddagger$ &totally inert morphisms &technical (Warning \ref{TotInert}) \\ \hline
$\mathcal{C}^{!}$ &left inert morphisms &technical (Definition \ref{DefLeftInert}) \\ \hline
\end{tabular}

\begin{example}
If $p:\mathcal{C}\to\mathcal{D}$ is a cocartesian (respectively cartesian) fibration and $\mathcal{D}^\dag$ is marked, there is an \emph{induced marking} $\mathcal{C}^\dag$, in which $f$ is marked if and only if $f$ is $p$-cocartesian (respectively $p$-cartesian) and $p(f)$ is marked.

If $\mathcal{D}^\flat$ has the flat marking, the induced marking on $\mathcal{C}$ is the flat marking $\mathcal{C}^\flat$. If $\mathcal{D}^\sharp$ has the sharp marking, the induced marking on $\mathcal{C}$ is the \emph{natural} marking $\mathcal{C}^\natural$: the marked morphisms are precisely the $p$-cocartesian morphisms.
\end{example}

\noindent If $\mathcal{C}^\dagger,\mathcal{D}^\dagger\in\text{Cat}^\dagger$, we will write $\text{Fun}^\dagger(\mathcal{C}^\dagger,\mathcal{D}^\dagger)$ for the full subcategory of $\text{Fun}(\mathcal{C},\mathcal{D})$ spanned by marked functors\footnote{In principle, $\text{Fun}^\dag(\mathcal{C}^\dag,\mathcal{D}^\dag)$ is itself naturally a marked category; marked morphisms are natural transformations that send each object of $\mathcal{C}$ to a marked morphism of $\mathcal{D}$. However, we will never use this marking.}.

If $\mathcal{C}^\dag\in\text{Cat}^\dag$, there is a universal functor $\mathcal{C}\to|\mathcal{C}^\dag|$ which sends each marked morphism to an equivalence \cite{Berman0}. Roughly, $|\mathcal{C}^\dag|$ is obtained from $\mathcal{C}$ by localizing (adjoining formal inverses) at all the marked morphisms.

\begin{example}
If $\mathcal{C}$ is any $\infty$-category, then $|\mathcal{C}^\flat|\cong\mathcal{C}$, and $|\mathcal{C}^\sharp|$ is the \emph{geometric realization}, or the $\infty$-groupoid built by formally inverting all morphisms.
\end{example}

\noindent The $\infty$-category $\text{Cat}^\dag$ admits all small limits and colimits (\cite{Berman0} 2.5). The limit (respectively colimit) of a diagram of marked $\infty$-categories $F:I\to\text{Cat}^\dag$ is the limit (colimit) of the underlying diagram of $\infty$-categories, marked via:
\begin{itemize}
\item A morphism $\phi$ of $\text{colim}(F)$ is marked if there exists $i\in I$ such that $F(i)\to\text{colim}(F)$ sends some marked morphism of $F(i)$ to $\phi$;
\item A morphism $\phi$ of $\text{lim}(F)$ is marked if for all $i\in I$, $\text{lim}(F)\to F(i)$ sends $\phi$ to a marked morphism of $F(i)$.
\end{itemize}

\begin{example}
If $\mathcal{A}^\dag\to\mathcal{C}^\dag$ and $\mathcal{B}^\dag\to\mathcal{C}^\dag$ are marked functors, then the pullback $\mathcal{A}\times_\mathcal{C}\mathcal{B}$ inherits a natural marking: a morphism is marked if and only if the projections to $\mathcal{A}$ and $\mathcal{B}$ are each marked. With this marking, $\mathcal{A}\times_\mathcal{C}\mathcal{B}$ is the pullback in $\text{Cat}^\dag$.
\end{example}

\subsection{Operads}\label{S23}
\noindent If $S$ is a set, we denote by $S_{+}=S\amalg\{\ast\}$ the pointed set obtained by adjoining a new basepoint. For each integer $n$, we also denote 
\begin{itemize}
\item $\left\langle n\right\rangle^\circ=\{1,\ldots,n\}$, an object of the category Fin of finite sets;
\item $\left\langle n\right\rangle=\left\langle n\right\rangle^\circ_{+}$, an object of the category $\text{Fin}_\ast$ of finite pointed sets;
\item $[n]=\{0<1<\cdots<n\}$, an object of the simplex category $\Delta$ of finite, nonempty, totally ordered sets.
\end{itemize}

\begin{definition}
We say that a function of pointed sets $f:S_{+}\to T_{+}$ is \emph{inert} if $|f^{-1}(t)|=1$ for all $t\in T$. Notice there is no condition on $f^{-1}(\ast)$.

We denote by $\text{Comm}^\mathsection$ the category of finite pointed sets, marked by inert morphisms\footnote{We use $\mathsection$ by analogy to $\natural$, because the inert marking on an $\infty$-operad is typically closely related to the natural marking on a cocartesian fibration.}.
\end{definition}

\begin{definition}[HA 2.1.4.2]\label{DefInert}
An \emph{$\infty$-preoperad} is an $\infty$-category equipped with a functor to $\text{Comm}$.

If $p:\mathcal{O}\to\text{Comm}$ is an $\infty$-preoperad, we say a morphism of $\mathcal{O}$ is \emph{inert} if it is $p$-cocartesian and its image in Comm is inert. We will always regard $\infty$-preoperads as marked by their inert morphisms, writing $\mathcal{O}^\mathsection\in\text{Cat}^\dag$.

A morphism of $\infty$-preoperads is a functor over Comm which sends inert morphisms to inert morphisms. That is, the $\infty$-category POp is the full subcategory of $\text{Cat}^\dag_{/\text{Comm}^\mathsection}$ spanned by those $\infty$-categories over Comm which are marked by their inert morphisms.
\end{definition}

\begin{warning}
This definition is slightly stronger than Lurie's. He defines an $\infty$-preoperad to be any marked $\infty$-category over $\text{Comm}^\mathsection$, while we require that it carries the canonical inert marking.
\end{warning}

\noindent Before we continue, note that there are inert morphisms $\rho^i:\left\langle n\right\rangle\to\left\langle 1\right\rangle$ for each $1\leq i\leq n$, defined by $\rho^i(j)=1$ if $i=j$ and $\ast$ otherwise.

If $p:\mathcal{O}\to\text{Comm}$ is an $\infty$-preoperad, we write $\mathcal{O}_n$ for the fiber over $\left\langle n\right\rangle$, which is a pullback $\mathcal{O}\times_\text{Comm}\{\left\langle n\right\rangle\}$.

\begin{definition}[HA 2.1.1.10,14]\label{DefOp}
An $\infty$-preoperad $p:\mathcal{O}\to\text{Comm}$ is an \emph{$\infty$-operad} if:
\begin{enumerate}
\item For every $X\in\mathcal{O}$ and inert morphism $f:p(X)\to Y$ in Comm, there is an inert morphism (equivalently, a $p$-cocartesian morphism) $X\to\bar{Y}$ in $\mathcal{O}$ lifting $f$; this implies that an inert morphism $\rho:\left\langle n\right\rangle\to\left\langle m\right\rangle$ induces a functor $\rho_\ast:\mathcal{O}_n\to\mathcal{O}_m$.
\item The functor $\mathcal{O}_n\to\mathcal{O}_1^{\times n}$ induced by the inerts $\rho^i:\left\langle n\right\rangle\to\left\langle 1\right\rangle$ is an equivalence of $\infty$-categories.
\item For every $X,Y\in\mathcal{O}$ and $f:p(X)\to p(Y)$, let $\text{Map}_\mathcal{O}^f(X,Y)$ be the union of connected components of $\text{Map}_\mathcal{O}(X,Y)$ lying over $f$. Choose an inert $Y\to Y_i$ lying over each inert $\rho^i:p(Y)\to\left\langle 1\right\rangle$. Then $$\text{Map}_\mathcal{O}^f(X,Y)\to\prod_i\text{Map}^{\rho^i f}_\mathcal{O}(X,Y_i)$$ is an equivalence.
\end{enumerate}
The $\infty$-operads form a full subcategory $\text{Op}\subseteq\text{POp}$.
\end{definition}

\noindent The terminal object of Op is $\text{Comm}^\mathsection$ itself, which we regard as the commutative $\infty$-operad (hence the notation Comm).

We also have the equally important associative $\infty$-operad Assoc: An object of Assoc is a finite pointed set. A morphism is a basepoint-preserving function $f:S_{+}\to T_{+}$, equipped with the data of total orderings on $f^{-1}(t)$ for each $t\in T$. (Note there is no extra data on $f^{-1}(\ast)$.)

\begin{definition}[HA 2.1.2.13]
If $\mathcal{O}$ is an $\infty$-operad, an \emph{$\mathcal{O}$-monoidal $\infty$-category} is a functor $\mathcal{V}:\mathcal{O}\to\text{Cat}$ such that $\smallint\mathcal{V}^\mathsection\xrightarrow{p}\mathcal{O}^\mathsection\to\text{Comm}^\mathsection$ exhibits the cocartesian fibration $\smallint\mathcal{V}^\mathsection$ as an $\infty$-operad.

In this case, we call $p$ a \emph{cocartesian fibration of $\infty$-operads}.

A \emph{lax $\mathcal{O}$-monoidal functor} $\mathcal{V}\to\mathcal{V}^\prime$ is a map of $\infty$-operads $\smallint\mathcal{V}\to\smallint\mathcal{V}^\prime$ over $\mathcal{O}$, and an \emph{$\mathcal{O}$-monoidal functor} is a lax $\mathcal{O}$-monoidal functor which also sends $p$-cocartesian morphisms to $p^\prime$-cocartesian morphisms.
\end{definition}

\begin{definition}
Suppose that $\mathcal{O}^\mathsection$ is an $\infty$-operad, $\mathcal{V}$ is an $\mathcal{O}$-monoidal $\infty$-category, and $\mathcal{O}^{\prime\mathsection}$ is an $\infty$-preoperad over $\mathcal{O}^\mathsection$. An \emph{$\mathcal{O}^\prime$-algebra in $\mathcal{V}$} is a map of $\infty$-preoperads over $\mathcal{O}^\mathsection$: $$\xymatrix{
\mathcal{O}^{\prime\mathsection}\ar[rd]\ar@{-->}[rr] &&\smallint\mathcal{V}^\mathsection\ar[ld] \\
&\mathcal{O}^\mathsection. &
}$$ There is an $\infty$-category of algebras $\text{Alg}_{\mathcal{O}^\prime/\mathcal{O}}(\mathcal{V})=\text{Fun}^\dag_{/\mathcal{O}^\mathsection}(\mathcal{O}^{\prime\mathsection},\smallint\mathcal{V}^\mathsection)$.
\end{definition}

\noindent We record a small lemma for later use:

\begin{lemma}\label{LemSpcf}
For any $\infty$-operad $\mathcal{O}$, $|\mathcal{O}^\mathsection|$ is contractible.
\end{lemma}

\begin{proof}
The claim is equivalent to: The functor $\text{Fun}^\dag(\mathcal{O}^\mathsection,\mathcal{C}^\flat)\to\mathcal{C}$, given by evaluation at the terminal object $\emptyset$, is an equivalence.

Suppose $F:\mathcal{O}\to\mathcal{C}$ is a functor which sends inert morphisms to equivalences. For any morphism $f:X\to Y$ in $\mathcal{O}$, $F$ sends the inert morphisms $Y\to\emptyset$ and $X\xrightarrow{f}Y\to\emptyset$ to equivalences, so $F(f)$ is an equivalence. Therefore, $\text{Fun}^\dag(\mathcal{O}^\mathsection,\mathcal{C}^\flat)=\text{Fun}^\dag(\mathcal{O}^\sharp,\mathcal{C}^\flat)\cong\text{Fun}(|\mathcal{O}|,\mathcal{C})$. Since $\mathcal{O}$ has a terminal object, the geometric realization is contractible, and this completes the proof.
\end{proof}

\noindent We will also be crucially interested in the left module $\infty$-operad LM (HA 4.2.1). It has the defining property that an LM-algebra in a monoidal $\infty$-category $\mathcal{V}$ is a pair $(A,M)$, where $A$ is an algebra in $\mathcal{V}$, and $M$ is a left $A$-module. We will give a formal definition in Section 4, as a special case of a more general construction $\text{LM}_S$.

\begin{example}\label{ExPair}
An LM-monoidal $\infty$-category may be regarded as a pair $(\mathcal{V};\mathcal{M})$, where $\mathcal{V}$ is a monoidal $\infty$-category and $\mathcal{M}$ is an $\infty$-category left tensored over $\mathcal{V}$.
\end{example}

\noindent Suppose $(\mathcal{V};\mathcal{M})$ is such an LM-monoidal $\infty$-category. An LM-algebra in $(\mathcal{V};\mathcal{M})$ consists of a pair $(A,M)$, where $A$ is an algebra in $\mathcal{V}$, and $M$ is a left $A$-module in $\mathcal{M}$. We write $$\text{LMod}(\mathcal{V};\mathcal{M})=\text{Alg}_{\text{LM}/\text{LM}}(\mathcal{V};\mathcal{M}).$$ There is a canonical forgetful functor $\theta:\text{LMod}(\mathcal{V};\mathcal{M})\to\text{Alg}(\mathcal{V})$, given by $\theta(A,M)=A$. Therefore, we may define:

\begin{definition}
If $A\in\text{Alg}(\mathcal{V})$, the $\infty$-category of left $A$-modules in $\mathcal{M}$ is $$\text{LMod}_A(\mathcal{M})=\text{LMod}(\mathcal{V};\mathcal{M})\times_{\text{Alg}(\mathcal{V})}\{A\}.$$
\end{definition}

\noindent Our goal is to generalize many of the nice properties satisfied by $\text{LMod}_A(\mathcal{M})$ to $\infty$-categories of presheaves. However, there is a major obstacle: The definition of $\text{LMod}_A(\mathcal{M})$ is fairly unnatural. For example, if $\mathcal{V}$ and $\mathcal{M}$ are both presentable satisfying mild conditions, then $\text{LMod}_A(\mathcal{M})$ is also presentable, but this is not at all clear from the definition.

Lurie solves this problem in \cite{HA} Chapter 4 by giving an alternative $\mathbb{A}_\infty$-model for algebras and left modules, which is equivalent to the operadic model discussed above. Then he identifies $\text{LMod}_A(M)$ with an $\infty$-category of marked functors.

The $\mathbb{A}_\infty$-model relies heavily on the theory of operadic approximations, which we review first.

\subsection{Operadic approximation}\label{S24}
\noindent Suppose $\mathcal{O}$ is an $\infty$-operad. We will often want to construct a simpler $\infty$-preoperad $I$ which has the same algebras as $\mathcal{O}$; that is, such that there is a functor $f:I\to\mathcal{O}$ which induces an equivalence $\text{Alg}_\mathcal{O}(\mathcal{V})\to\text{Alg}_I(\mathcal{V})$ for each monoidal $\infty$-category $\mathcal{V}$. We do this using Lurie's theory of approximations to $\infty$-operads (HA 2.3.3).

\begin{definition}
A morphism $f:S_{+}\to T_{+}$ in Comm is called \emph{active} if $f^{-1}(\ast)=\{\ast\}$ (HA 2.1.2.1-3). If $p:\mathcal{O}\to\text{Comm}$ is an $\infty$-operad, a morphism in $\mathcal{O}$ is called \emph{active} if its image in Comm is active.
\end{definition}

\begin{definition}[HA 2.3.3.6]\label{DefApprox}
Suppose $f:I\to\mathcal{O}$ is a map of $\infty$-preoperads, and $\mathcal{O}$ is an $\infty$-operad. We say $f$ is an \emph{approximation to $\mathcal{O}$} if:
\begin{enumerate}
\item For every $X\in I$ and inert morphism $\phi:p(X)\to\left\langle 1\right\rangle$ in Comm, there is an inert morphism $X\to Y$ in $I$ lifting $\phi$.
\item For every $Y\in I$ and active morphism $\phi:X\to f(Y)$ in $\mathcal{O}$, there is an $f$-cartesian morphism $\bar{X}\to Y$ lifting $\phi$.
\end{enumerate}
We say $f$ is a \emph{strong approximation} to $\mathcal{O}$ if it is an approximation to $\mathcal{O}$ and $f_1:I_1\to\mathcal{O}_1$ is an equivalence between the fibers over $\left\langle 1\right\rangle\in\text{Comm}$.
\end{definition}

\begin{theorem}[HA 2.3.3.23]\label{ThmApprox}
Let $\mathcal{O}$ be an $\infty$-operad, $\mathcal{V}$ a monoidal $\infty$-category, and $f:I\to\mathcal{O}$ a strong approximation to $\mathcal{O}$. Then the map $f^\ast:\text{Alg}_\mathcal{O}(\mathcal{V})\to\text{Alg}_I(\mathcal{V})$ induced by composition with $f$ is an equivalence of $\infty$-categories.
\end{theorem}

\noindent We are motivated by two examples of operadic approximations.

\begin{example}[HA 4.1.2.11]\label{ExCut}
There is a strong approximation to Assoc, $\text{Cut}:\Delta^\text{op}\to\text{Assoc}$. A morphism $f^\ast:[n]\to[m]$ in $\Delta^\text{op}$ is inert if and only if the associated morphism $f_\ast:[m]\to[n]$ in $\Delta$ embeds $[m]$ as a convex subset $\{i<i+1<\cdots<i+m\}\subseteq[n]$.
\end{example}

\begin{example}[HA 4.2.2.8]\label{ExCut2}
There is a strong approximation to LM, \\$\text{LCut}:\Delta^\text{op}\times\Delta^1\to\text{LM}$, which fits into a commutative square $$\xymatrix{
\Delta^\text{op}\times\{1\}\ar[r]^-{\text{Cut}}\ar[d] &\text{Assoc}\ar[d] \\
\Delta^\text{op}\times\Delta^1\ar[r]_-{\text{LCut}} &\text{LM}.
}$$ A morphism $([n]\xrightarrow{f^\ast}[m],i\to j)$ in $\Delta^\text{op}\times\Delta^1$ is inert if and only if $f^\ast$ is inert in $\Delta^\text{op}$ and either:
\begin{itemize}
\item $f_\ast(m)=n$;
\item or $j=1$.
\end{itemize}
\end{example}

\noindent We will delay explicit constructions of Cut and LCut until Section 5, as special cases of more general constructions $\text{Cut}_S$ and $\text{LCut}_S$.

\begin{warning}\label{TotInert}
From the inert marking on $\Delta^\text{op}\times\Delta^1$, notice that the two embeddings $\Delta^\text{op}\subseteq\Delta^\text{op}\times\Delta^1$ induce two different markings on $\Delta^\text{op}$:
\begin{itemize}
\item The embedding $\{1\}\subseteq\Delta^1$ induces the \emph{inert} marking $\Delta^{\text{op}\mathsection}$;
\item The embedding $\{0\}\subseteq\Delta^1$ induces the \emph{totally inert} marking $\Delta^{\text{op}\ddagger}$, where a morphism $f^\ast:[n]\to[m]$ of $\Delta^\text{op}$ is \emph{totally inert} if it is inert and $f_\ast(m)=n$.
\end{itemize}
\end{warning}

\subsection{The $\mathbb{A}_\infty$-model for higher algebra}\label{S25}
\noindent We will end this section by reviewing the $\mathbb{A}_\infty$-model for algebras and left modules.

For a monoidal $\infty$-category $\mathcal{V}$, the composite $\Delta^\text{op}\xrightarrow{\text{Cut}}\text{Assoc}\xrightarrow{\mathcal{V}}\text{Cat}$ is the simplicial $\infty$-category $\mathcal{BV}$ given by the Milnor construction (HA 4.1.2.4): $$\xymatrix{
\cdots\mathcal{V}^{\times 2} \ar@<2ex>[r] \ar@<0ex>[r] \ar@<-2ex>[r] &  \mathcal{V} \ar@<1ex>[r] \ar@<-1ex>[r] \ar@<1ex>[l] \ar@<-1ex>[l] &  \ast. \ar@<0ex>[l]
}$$

\begin{definition}
An \emph{$\mathbb{A}_\infty$-algebra} of $\mathcal{V}$ is a section of $\smallint\mathcal{BV}\xrightarrow{p}\Delta^\text{op}$ which sends inert morphisms to $p$-cocartesian morphisms, and we write $$\mathbb{A}_\infty\text{Alg}(\mathcal{V})=\text{Fun}^\dag_{/\Delta^{\text{op}\mathsection}}(\Delta^{\text{op}\mathsection},\smallint\mathcal{BV}^\mathsection).$$
\end{definition}

\noindent By Example \ref{ExCut} and Theorem \ref{ThmApprox}, composition with $\text{Cut}:\Delta^\text{op}\to\text{Assoc}$ induces an equivalence $\text{Alg}(\mathcal{V})\to\mathbb{A}_\infty\text{Alg}(\mathcal{V})$.

\begin{remark}\label{RmkHA413}
The preceding material follows HA 4.1.3, where Lurie refers to $\smallint\mathcal{BV}\to\Delta^\text{op}$ as the \emph{planar $\infty$-operad} (which he writes $\mathcal{V}^\oast\to\Delta^\text{op}$) associated to the $\infty$-operad $\smallint\mathcal{V}\to\text{Assoc}$ (which he writes $\mathcal{V}^\otimes\to\text{Assoc}$).
\end{remark}

\noindent Now suppose we also have a left $\mathcal{V}$-module $\infty$-category $\mathcal{M}$; that is, the pair $(\mathcal{V};\mathcal{M})$ is an LM-monoidal $\infty$-category as in Example \ref{ExPair}. Denote by $\mathcal{B}(\mathcal{V};\mathcal{M})$ the composite $\Delta^\text{op}\times\Delta^1\xrightarrow{\text{LCut}}\text{LM}\xrightarrow{(\mathcal{V};\mathcal{M})}\text{Cat}$. For concreteness, $\mathcal{B}(\mathcal{V};\mathcal{M})$ may be regarded as the morphism of simplicial $\infty$-categories: $$\xymatrix{
\mathcal{BV}\ltimes\mathcal{M}=\ar[d]_q &\cdots\mathcal{V}^{\times 2}\times\mathcal{M} \ar[d]\ar@<2ex>[r] \ar@<0ex>[r] \ar@<-2ex>[r] &  \mathcal{V}\times\mathcal{M} \ar[d]\ar@<1ex>[r] \ar@<-1ex>[r] \ar@<1ex>[l] \ar@<-1ex>[l] &  \mathcal{M} \ar[d]\ar@<0ex>[l] \\
\mathcal{BV}= &\cdots\quad\,\mathcal{V}^{\times 2}\quad\,\, \ar@<2ex>[r] \ar@<0ex>[r] \ar@<-2ex>[r] &  \quad\mathcal{V}\quad \ar@<1ex>[r] \ar@<-1ex>[r] \ar@<1ex>[l] \ar@<-1ex>[l] &  \ast. \ar@<0ex>[l]
}$$ The downward maps are simply projection away from $\mathcal{M}$; the module structure is recorded by the horizontal maps in the top row. The bottom simplicial $\infty$-category is just $\mathcal{BV}$ (which is in particular independent of $\mathcal{M}$) because the square of Example \ref{ExCut2} commutes.

\begin{definition}
A \emph{left $\mathbb{A}_\infty$-module} is a section of $\smallint\mathcal{B}(\mathcal{V};\mathcal{M})\xrightarrow{p}\Delta^\text{op}\times\Delta^1$ which sends inert morphisms to $p$-cocartesian morphisms, and we write $$\mathbb{A}_\infty\text{LMod}(\mathcal{V};\mathcal{M})=\text{Fun}^\dag_{/\Delta^\text{op}\times\Delta^{1\mathsection}}(\Delta^\text{op}\times\Delta^{1\mathsection},\smallint\mathcal{B}(\mathcal{V};\mathcal{M})^\mathsection).$$
\end{definition}

\noindent By Example \ref{ExCut2}, composition with $\text{LCut}:\Delta^\text{op}\times\Delta^1\to\text{LM}$ induces an equivalence $\text{LMod}(\mathcal{V};\mathcal{M})\to\mathbb{A}_\infty\text{LMod}(\mathcal{V};\mathcal{M})$.

If $A\in\mathbb{A}_\infty\text{Alg}(\mathcal{V})$, then we may define $\mathbb{A}_\infty\text{LMod}_A(\mathcal{M})$ just as in we did in the operadic model:

\begin{definition}
If $(\mathcal{V};\mathcal{M})$ is an LM-monoidal $\infty$-category and $A\in\mathbb{A}_\infty\text{Alg}(\mathcal{V})$, then $\mathbb{A}_\infty\text{LMod}_A(\mathcal{M})=\mathbb{A}_\infty\text{LMod}(\mathcal{V};\mathcal{M})\times_{\mathbb{A}_\infty\text{Alg}(\mathcal{V})}\{A\}$.
\end{definition}

\noindent The benefit of the $\mathbb{A}_\infty$-model is that we can also give a much more explicit model for $\mathbb{A}_\infty\text{LMod}_A(\mathcal{M})$ than this last definition. In order to do so, we need a bit more notation.

Applying the Grothendieck construction to the map of simplicial $\infty$-categories $\smallint\mathcal{BV}\ltimes\mathcal{M}\to\smallint\mathcal{BV}$, we have a functor $q$ over $\Delta^\text{op}$, which sends $p_0$-cocartesian morphisms to $p_1$-cocartesian morphisms: $$\xymatrix{
\smallint\mathcal{BV}\ltimes\mathcal{M}\ar[rr]^-q\ar[rd]_-{p_0} &&\smallint\mathcal{BV}\ar[ld]^-{p_1} \\
&\Delta^\text{op} &
}$$

\begin{remark}\label{RmkPrec}
This discussion follows HA 4.2.2, where Lurie uses the notation $\mathcal{M}^\oast=\smallint\mathcal{BV}\ltimes\mathcal{M}$. Therefore, he writes $\mathcal{M}^\oast\to\mathcal{V}^\oast$ where we write $\smallint\mathcal{BV}\ltimes\mathcal{M}\to\smallint\mathcal{BV}$.
\end{remark}

\begin{remark}\label{RmkLCF}
Our discussion so far is model-independent, but everything is fairly concrete if we work with quasicategories. Suppose we are given a cocartesian fibration of quasicategories $\smallint(\mathcal{V};\mathcal{M})\to\text{LM}$ which realizes $\mathcal{M}$ as a left $\mathcal{V}$-module. Pullback along the two inclusions $\Delta^\text{op}\subseteq\Delta^\text{op}\times\Delta^1\xrightarrow{\text{LCut}}\text{LM}$ yields quasicategory models for $\smallint\mathcal{BV}\ltimes\mathcal{M}$ and $\smallint\mathcal{BV}$.

In this case, the functor $q:\smallint\mathcal{BV}\ltimes\mathcal{M}\to\smallint\mathcal{BV}$ is a categorical fibration (HA 4.2.2.19) and a locally cocartesian fibration (HA 4.2.2.20).
\end{remark}

\begin{remark}\label{RmkTotInert}
We will be interested in the following markings (with notation as in the triangle above):
\begin{itemize}
\item The marking $\smallint\mathcal{BV}^\mathsection$ by \emph{inert} morphisms, or morphisms which are $p_1$-cocartesian and lie over inert morphisms in $\Delta^\text{op}$;
\item The marking $\smallint\mathcal{BV}\ltimes\mathcal{M}^{\natural !}$ by locally $q$-cocartesian morphisms (we use the notation $\natural!$ to emphasize this is \emph{not} the expected marking by $p_0$-cocartesian morphisms);
\item The marking $\smallint\mathcal{BV}\ltimes\mathcal{M}^\ddagger$ by \emph{totally inert} morphisms, or morphisms which are $p_0$-cocartesian and lie over totally inert morphisms in $\Delta^\text{op}$ (see Warning \ref{TotInert}).
\end{itemize}
\end{remark}

\noindent We end with three important propositions describing left module $\infty$-categories.

\begin{proposition}[HA 4.2.2.19]\label{PropA1}
If $\mathcal{V}$ is a monoidal $\infty$-category, $\mathcal{M}$ is a left $\mathcal{V}$-module $\infty$-category, and $A$ is an $\mathbb{A}_\infty$-algebra of $\mathcal{V}$ (that is, a marked section of $\smallint\mathcal{BV}^\mathsection\to\Delta^{\text{op}\mathsection}$), then $$\mathbb{A}_\infty\text{LMod}_A(\mathcal{M})\cong\text{Fun}^\dag_{/\smallint\mathcal{BV}^\mathsection}(\Delta^{\text{op}\ddagger},\smallint\mathcal{BV}\ltimes\mathcal{M}^\ddagger).$$
\end{proposition}

\begin{proof}
See Proposition \ref{PropPShModel} for the proof of a more general result (or just see HA 4.2.2.19 for this one).
\end{proof}

\begin{proposition}[HA 4.8.4.12]\label{PropALMod}
If $\mathcal{V}$ is a monoidal $\infty$-category, $\mathcal{M},\mathcal{N}$ are left $\mathcal{V}$-module $\infty$-categories, and $\text{Fun}_{\text{LMod}_\mathcal{V}}(\mathcal{M},\mathcal{N})$ is the $\infty$-category of $\mathcal{V}$-linear functors (HA 4.6.2.7), then $$\text{Fun}_{\text{LMod}_\mathcal{V}}(\mathcal{M},\mathcal{N})\cong\text{Fun}^\dag_{/\smallint\mathcal{BV}^\sharp}(\smallint\mathcal{BV}\ltimes\mathcal{M}^{\natural !},\smallint\mathcal{BV}\ltimes\mathcal{N}^{\natural !}).$$
\end{proposition}

\begin{corollary}\label{CorALMod}
$\mathbb{A}_\infty\text{LMod}_A(\text{Cat})$ is equivalent as an $(\infty,2)$-category to the full subcategory of $\text{Cat}^\dag_{/\smallint\mathcal{BV}^\sharp}$ spanned by $\smallint\mathcal{BV}\ltimes\mathcal{M}^{\natural !}$ as $\mathcal{M}$ varies over left $\mathcal{V}$-modules.
\end{corollary}

\noindent Actually, we haven't defined $\text{LMod}_A(\text{Cat})$ as an $(\infty,2)$-category, but Proposition \ref{PropALMod} implies that there is an equivalence of $\infty$-categories, and that it would promote to an equivalence of $(\infty,2)$-categories given any sensible definition of $\text{LMod}_A(\text{Cat})$ as an $(\infty,2)$-category. On the other hand, we might simply take Corollary \ref{CorALMod} as our definition of $\text{LMod}_A(\text{Cat})$ as an $(\infty,2)$-category.

For the next proposition, note that we can define right modules in parallel with left modules, using an $\infty$-operad RM in place of LM. We still have a strong approximation $\Delta^\text{op}\times\Delta^1\to\text{RM}$, and an RM-monoidal $\infty$-category is a pair $(\mathcal{V};\mathcal{N})$ exhibiting $\mathcal{N}$ as a right $\mathcal{V}$-module $\infty$-category. In this case, the associated functor $\mathcal{B}(\mathcal{V};\mathcal{N}):\Delta^\text{op}\times\Delta^1\to\text{Cat}$ may be identified with a morphism of simplicial $\infty$-categories $$\xymatrix{
\mathcal{N}\rtimes\mathcal{BV}=\ar[d]_q &\cdots\mathcal{N}\times\mathcal{V}^{\times 2} \ar[d]\ar@<2ex>[r] \ar@<0ex>[r] \ar@<-2ex>[r] & \mathcal{N}\times\mathcal{V} \ar[d]\ar@<1ex>[r] \ar@<-1ex>[r] \ar@<1ex>[l] \ar@<-1ex>[l] &  \mathcal{N} \ar[d]\ar@<0ex>[l] \\
\mathcal{BV}= &\cdots\quad\,\mathcal{V}^{\times 2}\quad\,\, \ar@<2ex>[r] \ar@<0ex>[r] \ar@<-2ex>[r] &  \quad\mathcal{V}\quad \ar@<1ex>[r] \ar@<-1ex>[r] \ar@<1ex>[l] \ar@<-1ex>[l] &  \ast. \ar@<0ex>[l]
}$$

\begin{proposition}[HA 4.8.4.3]\label{PropLR}
If $\mathcal{V}$ is a monoidal $\infty$-category, $\mathcal{N},\mathcal{M}$ are right (respectively left) $\mathcal{V}$-module $\infty$-categories, and $\mathcal{N}\otimes_\mathcal{V}\mathcal{M}$ is the tensor product of $\mathcal{V}$-modules (HA 4.4), let $r$ be the canonical cocartesian fibration $(\smallint\mathcal{N}\rtimes\mathcal{BV})\times_{\smallint\mathcal{BV}}(\smallint\mathcal{BV}\ltimes\mathcal{M})\to\Delta^\text{op}$. Then $$\mathcal{N}\otimes_\mathcal{V}\mathcal{M}\cong|(\smallint\mathcal{N}\rtimes\mathcal{BV})\times_{\smallint\mathcal{BV}}(\smallint\mathcal{BV}\ltimes\mathcal{M})^\natural|,$$ where the marking $\natural$ is by $r$-cocartesian edges.
\end{proposition}

\begin{remark}\label{PropLRRmk}
According to Proposition \ref{PropLR}, a functor $F:\mathcal{N}\otimes_\mathcal{V}\mathcal{M}\to\mathcal{Z}$ is classified by a functor $\bar{F}:(\smallint\mathcal{N}\rtimes\mathcal{BV}\times_{\smallint\mathcal{BV}}(\smallint\mathcal{BV}\ltimes\mathcal{M})\to\mathcal{Z}$. HA 4.8.4.3 actually asserts more.

Note that the fiber of $(\smallint\mathcal{N}\rtimes\mathcal{BV}\times_{\smallint\mathcal{BV}}(\smallint\mathcal{BV}\ltimes\mathcal{M})\to\Delta^\text{op}$ over $[n]$ is $\mathcal{N}\times\mathcal{V}^{\times n}\times\mathcal{M}$. If the restriction of $\bar{F}$ to $\mathcal{N}\times\mathcal{V}^{\times n}\times\mathcal{M}\to\mathcal{Z}$ preserves small colimits for each $n$, then $F$ specializes to a colimit-preserving functor $\mathcal{N}\otimes^L_\mathcal{V}\mathcal{M}\to\mathcal{Z}$, where $\otimes^L_\mathcal{V}$ is a relative tensor product taken in $\text{Pr}^L$.
\end{remark}

\section{The operadic model for enriched categories}\label{S3}
\noindent Fix a set $S$. Gepner-Haugseng \cite{GH} construct an $\infty$-operad $\text{Assoc}_S$ (they call it $\mathcal{O}_S$) with the universal property: If $\mathcal{V}$ is a monoidal $\infty$-category, an $\text{Assoc}_S$-algebra in $\mathcal{V}$ is a $\mathcal{V}$-enriched category with set $S$ of objects. We write $$\text{Cat}^\mathcal{V}_S=\text{Alg}_{\text{Assoc}_S/\text{Assoc}}(\mathcal{V}).$$ A $\mathcal{V}$-enriched presheaf on $\mathcal{C}\in\text{Cat}^\mathcal{V}_S$ is then something like a contravariant, $\mathcal{V}$-enriched functor from $\mathcal{C}$ to $\mathcal{V}$. However, this is not a suitable definition at this point, because of two obstacles:
\begin{itemize}
\item $\mathcal{V}$ does not belong to the same $\infty$-category $\text{Cat}^\mathcal{V}_S$, as it doesn't have set $S$ of objects;
\item $\mathcal{V}$ is not itself $\mathcal{V}$-enriched unless it is closed monoidal; even then, it is not obvious how to construct the self-enrichment on $\mathcal{V}$.
\end{itemize}
\noindent The first obstacle is not serious; Gepner-Haugseng construct an $\infty$-category $\text{Cat}^\mathcal{V}$ of \emph{all} $\mathcal{V}$-enriched categories. However, we will delay this construction until a future paper. In any case, it is not easy to understand enriched functor $\infty$-categories in $\text{Cat}^\mathcal{V}$, so we prefer to avoid this approach.

The second obstacle is more serious. While it is possible to construct the self-enrichment on $\mathcal{V}$ (see \cite{GH} Section 7), the construction is rather obscure and not easy to use.

Instead, we will construct an $\infty$-operad $\text{LM}_S$ with a natural embedding $\text{Assoc}_S\subseteq\text{LM}_S$. If $\mathcal{C}$ is a $\mathcal{V}$-enriched category, which is an $\text{Assoc}_S$-algebra in $\mathcal{V}$, then an enriched presheaf on $\mathcal{C}$ is a lift to an $\text{LM}_S$-algebra. In other words, if we write $$\text{PSh}^\mathcal{V}_S=\text{Alg}_{\text{LM}_S/\text{Assoc}}(\mathcal{V}),$$ we can regard an object of $\text{PSh}^\mathcal{V}_S$ as a pair $(\mathcal{C},\mathcal{F})$, where $\mathcal{C}\in\text{Cat}^\mathcal{V}_S$ and $\mathcal{F}$ is an enriched presheaf on $\mathcal{C}$. The inclusion $\text{Assoc}_S\subseteq\text{LM}_S$ induces a forgetful functor $\theta:\text{PSh}^\mathcal{V}_S\to\text{Cat}^\mathcal{V}_S$, and we define $$\text{PSh}^\mathcal{V}(\mathcal{C})=\theta^{-1}(\mathcal{C}),$$ the $\infty$-category of enriched presheaves on $\mathcal{C}$.

In fact, we always work in a more general situation. If $\mathcal{M}$ is an $\infty$-category left tensored over $\mathcal{V}$, the pair $(\mathcal{V},\mathcal{M})$ is itself an LM-algebra in Cat, or an LM-monoidal $\infty$-category. We will define $\text{PSh}^{\mathcal{V};\mathcal{M}}_S=\text{Alg}_{\text{LM}_S/\text{LM}}(\mathcal{V};\mathcal{M})$, which we interpret as the $\infty$-category of pairs $(\mathcal{C},\mathcal{F})$, where $\mathcal{C}$ is a $\mathcal{V}$-enriched category, and $\mathcal{F}$ is an enriched presheaf \emph{with values in $\mathcal{M}$}. This should be thought of informally as an enriched, contravariant functor from $\mathcal{C}$ to $\mathcal{M}$.

In Section \ref{S31}, we review the construction of enriched categories, due to Gepner-Haugseng \cite{GH}. Then in Section \ref{S32}, we introduce enriched presheaves. Finally, in Section \ref{S33}, we introduce two closely related constructions:

If $X\in\mathcal{C}$, a presheaf on $\mathcal{C}$ can be evaluated at $X$, and evaluation at $X$ is functorial $\text{ev}_X:\text{PSh}^\mathcal{V}(\mathcal{C};\mathcal{M})\to\mathcal{M}$. Also for each $X$, there is a representable presheaf $\text{rep}_X\in\text{PSh}^\mathcal{V}(\mathcal{C})$ given informally by $\text{rep}_X(Y)=\mathcal{C}(Y,X)$. In our main result of this section (Theorem \ref{ThmFree}), we prove that $\text{rep}_X$ is a free presheaf. In other words:

\begin{repcorollary}{CorLAdj}
The functor $\text{ev}_X:\text{PSh}^\mathcal{V}(\mathcal{C};\mathcal{M})\to\mathcal{M}$ has a left adjoint described by $\text{rep}_X\otimes -:\mathcal{M}\to\text{PSh}^\mathcal{V}(\mathcal{C};\mathcal{M})$. The presheaf $\text{rep}_X\otimes M$ is given informally by $(\text{rep}_X\otimes M)(Y)=\mathcal{C}(Y,X)\otimes M$.
\end{repcorollary}

\noindent This can be regarded as a form of the enriched Yoneda lemma:

\begin{repcorollary}{CorYon}
If $\mathcal{F}\in\text{PSh}^\mathcal{V}(\mathcal{C})$ and $A\in\mathcal{V}$, then $$\text{Map}(\text{rep}_X\otimes A,\mathcal{F})\cong\text{Map}(A,\mathcal{F}(X)).$$
\end{repcorollary}

\subsection{Enriched categories}\label{S31}
\begin{definition}\label{DefAss}
Fix a set $S$. An object of $\text{Assoc}_S$ is a finite pointed set $E_{+}$ along with two functions $s,t:E\to S$. A morphism is a basepoint-preserving function $f:E_{+}\to E^\prime_{+}$ along with a total ordering of $f^{-1}(e)$ for each $e\in E^\prime$, such that:
\begin{enumerate}
\item If $f^{-1}(e)$ is empty, then $s(e)=t(e)$;
\item If $f^{-1}(e)=\{e_0<e_1<\ldots<e_n\}$ nonempty, then $s(e_0)=s(e)$, $t(e_n)=t(e)$, and $t(e_i)=s(e_{i+1})$ for each $0\leq i<n$.
\end{enumerate}
We say $f$ is \emph{inert} if $|f^{-1}(e)|=1$ for all $e\in E^\prime$, which makes $\text{Assoc}_S^\mathsection$ marked.
\end{definition}

\begin{remark}\label{RmkAss}
We make sense of the definition as so: We regard an object of $\text{Assoc}_S$ as a \emph{directed graph} $\Gamma$ with set $S$ of vertices and set $E$ of edges. Each edge $e\in E$ has source vertex $s(e)$ and target vertex $t(e)$.

A morphism $f:\Gamma\to\Gamma^\prime$ is a way of transforming $\Gamma$ into $\Gamma^\prime$ by means of the following three operations:
\begin{enumerate}
\item deleting some edges -- these are the edges for which $f(e)=\ast$;
\item adding some loops (edges from a vertex to itself) -- this corresponds to condition (1) above;
\item deleting edges $e_0,\ldots,e_n$ which form a path from $s=s(e_0)$ to $t=t(e_n)$, and replacing them by a single edge from $s$ to $t$ -- this corresponds to condition (2) above.
\end{enumerate}
\noindent This description should also make clear how to compose morphisms.

An inert morphism is one that involves only the operation (1), and an active morphism is one that involves only the operations (2)-(3).
\end{remark}

\begin{remark}
In \cite{BermanTHH}, we call the active morphisms in $\text{Assoc}_S$ \emph{bypass operations}, and we study the symmetric monoidal envelope of $\text{Assoc}_S$, which we call $\text{Bypass}_S$.
\end{remark}

\begin{remark}\label{RmkAss3}
There is an evident forgetful functor $\text{Assoc}_S\to\text{Assoc}$. It sends inert morphisms to inert morphisms, and it is an isomorphism when $|S|=1$ (in this case the functions $s,t:E\to S$ carry no information, and the two conditions of Definition \ref{DefAss} are satisfied for free).
\end{remark}

\noindent Suppose there are two graphs $\Gamma,\Gamma^\prime\in\text{Assoc}_S$. Then there is a new graph $\Gamma\otimes\Gamma^\prime$ given by disjoint union of the sets of edges in $\Gamma$ and $\Gamma^\prime$. This $\otimes$ endows $\text{Assoc}_S$ with a monoidal structure.

The unit of $\otimes$ is $\emptyset$, the graph with no edges. Also, given $X,Y\in S$, let $(X,Y)\in\text{Assoc}_S$ denote the graph with a single edge from $X$ to $Y$.

Every object of $\text{Assoc}_S$ can be written uniquely (up to permutation) in the form $(X_1,Y_1)\otimes\cdots\otimes(X_n,Y_n)$. There are also morphisms corresponding to the three operations of Remark \ref{RmkAss}:
\begin{enumerate}
\item $(X,Y)\to\emptyset$ for each $X,Y\in S$;
\item $\emptyset\to(X,X)$ for each $X\in S$;
\item $(X,Y)\otimes(Y,Z)\to(X,Z)$ for each $X,Y,Z\in S$.
\end{enumerate}
Morphisms of type (1) are inert. Morphisms of type (2) and (3) will describe identity morphisms and composition in enriched categories.

\begin{proposition}\label{PropAssOp}
The composite $p:\text{Assoc}_S^\mathsection\to\text{Assoc}^\mathsection\to\text{Comm}^\mathsection$ is an $\infty$-operad, and $\text{Assoc}_S^\mathsection\to\text{Assoc}^\mathsection$ is a map of $\infty$-operads.
\end{proposition}

\begin{proof}
First we check that the inert morphisms in $\text{Assoc}_S$ are $p$-cocartesian, so that the inert marking we have described is the correct inert marking for an $\infty$-operad. Let $f:\Gamma_0\to\Gamma_1$ be the inert morphism in question, so that $f^{-1}$ induces an inclusion of graphs $\Gamma_1\subseteq\Gamma_0$ (in the sense of an inclusion of edge sets, not a morphism in $\text{Assoc}_S$). Assume we have a map $h:\Gamma_0\to\Gamma^\prime$, and a commutative triangle $$\xymatrix{
p(\Gamma_0)\ar[d]_f\ar[r]^h &p(\Gamma^\prime). \\
p(\Gamma_1) \ar@{-->}[ru]_u&
}$$ Then the restriction of $h$ along $\Gamma_1\subseteq\Gamma_0$ describes a map $\bar{u}:\Gamma_1\to\Gamma^\prime$ (the only possible lift of $u$ to $\text{Assoc}_S$), precisely because the triangle commutes. This is what it means for $f$ to be $p$-cocartesian. Hence $\text{Assoc}_S^\mathsection\to\text{Comm}^\mathsection$ is an $\infty$-preoperad.

Now we need to check conditions (1)-(3) of Definition \ref{DefOp}.

(1) Given $\Gamma\in\text{Assoc}_S$ and an inert morphism $f:p(\Gamma)\to T_{+}$, $f^{-1}$ exhibits $T$ as a subset of the set of edges of $\Gamma$. Let $\Gamma^\prime$ be the subgraph spanned by those edges. Then $f$ lifts to an inert morphism $\bar{f}:\Gamma\to\Gamma^\prime$.

(2) If $\mathcal{O}=\text{Assoc}_S$, each $\mathcal{O}_n$ is just the (discrete) set $(S^2)^n$, where an element $(X,Y)\in S^2$ is interpreted as an edge from $X$ to $Y$. It follows that $\mathcal{O}_n\to\mathcal{O}_1^{\times n}$ is an equivalence.

(3) A morphism $\Gamma\to\Gamma^\prime$ in $\text{Assoc}_S$ is determined only by the underlying morphism $f:E_{+}\to E^\prime_{+}$ and data and conditions on each fiber $f^{-1}(e)$.

Therefore $\text{Assoc}_S^\mathsection$ is an $\infty$-operad. The forgetful functor $\text{Assoc}_S\to\text{Assoc}$ is compatible with the forgetful functors to Comm. It sends inerts to inerts, because in each case a morphism is inert if and only if it lies over an inert morphism in Comm, so it is a map of $\infty$-operads.
\end{proof}

\begin{definition}
If $\mathcal{V}$ is a monoidal $\infty$-category and $S$ a set, a \emph{$\mathcal{V}$-enriched category with set $S$ of objects} is an $\text{Assoc}_S$-algebra in $\mathcal{V}$, and we write $$\text{Cat}^\mathcal{V}_S=\text{Alg}_{\text{Assoc}_S/\text{Assoc}}(\mathcal{V}).$$
\end{definition}

\begin{example}
Suppose $|S|=1$. Then $\text{Assoc}_S\cong\text{Assoc}$ as in Remark \ref{RmkAss3}. Therefore, $\mathcal{V}$-enriched categories with one object can be identified with associative algebras in $\mathcal{V}$.
\end{example}

\noindent Unpacking the definition, a $\mathcal{V}$-enriched category with set $S$ of objects is a map of $\infty$-operads $\mathcal{C}:\text{Assoc}_S^\mathsection\to\smallint\mathcal{V}^\mathsection$. Such a functor $\mathcal{C}$ comes with the following data, plus coherences:
\begin{itemize}
\item `hom' objects $\mathcal{C}(X,Y)\in\smallint\mathcal{V}_1=\mathcal{V}$;
\item `identity' morphisms $1\to\mathcal{C}(X,X)$;
\item `composition' morphisms $\mathcal{C}(X,Y)\otimes\mathcal{C}(Y,Z)\to\mathcal{C}(X,Z)$.
\end{itemize}

\noindent This is exactly the classical structure of an enriched category.

A morphism $\mathcal{C}\to\mathcal{D}$ in $\text{Cat}^\mathcal{V}_S$ consists of maps $\mathcal{C}(X,Y)\to\mathcal{D}(X,Y)$. In other words, it is an \emph{enriched functor} $F:\mathcal{C}\to\mathcal{D}$ which acts as the identity on objects: $F(X)=X$ for all $X\in S$.

\begin{example}\label{ExTriv1}
Let $0$ denote the contractible $\infty$-category $\ast$ with its unique monoidal structure. The corresponding functor $\text{Assoc}\to\text{Cat}$ is the constant functor with value $\ast$, so the associated cocartesian fibration is the identity $\text{Assoc}\to\text{Assoc}$. Therefore, $\text{Cat}^0_S=\text{Fun}^\dag_{/\text{Assoc}}(\text{Assoc}_S,\text{Assoc})\cong\ast$ is contractible for each $S$.

In summary, there is a unique (up to equivalence) $0$-enriched category $1_S$ with set $S$ of objects. We may regard it as having $1_S(X,Y)=1$ for each $X,Y\in S$, where $1\in\mathcal{V}$ is the monoidal unit.
\end{example}

\begin{example}\label{ExTriv1b}
For any monoidal $\infty$-category $\mathcal{V}$, there is a unique monoidal functor $0\to\mathcal{V}$. By pushing forward the $\text{Assoc}_S$-algebra $1_S$ of the last example, we obtain a $\mathcal{V}$-enriched category $1_S$. It has the property $1_S(X,Y)=1$ for all $X,Y\in S$. This is the \emph{trivial $\mathcal{V}$-enriched category with set $S$ of objects}.
\end{example}

\begin{remark}
If $\mathcal{V}$ is presentable and closed monoidal, then $\text{Cat}^\mathcal{V}_S=\text{Alg}_{\text{Assoc}_S}(\mathcal{V})$ is also presentable by HA 3.2.3.5.
\end{remark}

\subsection{Enriched presheaves}\label{S32}
\begin{definition}\label{DefLMS}
Suppose $S_{+}$ is a pointed set and $\Gamma\in\text{Assoc}_{S_{+}}$. Call $\Gamma$ \emph{left-modular} if $s(e)\neq\ast$ for all edges $e$ in $\Gamma$.

If $S$ is a set, $\text{LM}_S$ is the full subcategory of $\text{Assoc}_{S_{+}}$ spanned by left modular graphs.

A morphism of $\text{LM}_S$ is \emph{inert} if it is inert in $\text{Assoc}_{S_{+}}$.
\end{definition}

\noindent In other words, we may think of the objects of $\text{LM}_S$ as graphs on vertex-set $S_{+}$, such that no edges have source $\ast$.

Hence there are two kinds of edges of $\text{LM}_S$: those of the form $(X,Y)$, and those of the form $(Z,\ast)$. The edges that do not involve $\ast$ span the full subcategory $\text{Assoc}_S$, so that there are inclusions of marked categories $$\text{Assoc}_S^\mathsection\subseteq\text{LM}_S^\mathsection\subseteq\text{Assoc}_{S_{+}}^\mathsection,$$ and $\text{LM}_S$ inherits from $\text{Assoc}_{S_{+}}$ the monoidal operation $\otimes$.

\begin{example}
If $|S|=1$, then $\text{LM}_S\cong\text{LM}$, the left module operad.
\end{example}

\noindent For any function $S\to T$, the induced functor $\text{Assoc}_{S_{+}}\to\text{Assoc}_{T_{+}}$ restricts to $\text{LM}_S\to\text{LM}_T$. In particular, there are canonical functors $\text{LM}_S\to\text{LM}$ for each $S$ which forget the labelings of the vertices of a graph (except for the distinguished vertex $\ast$).

\begin{proposition}
The composite $\text{LM}_S^\mathsection\to\text{LM}^\mathsection\to\text{Comm}^\mathsection$ is an $\infty$-operad.
\end{proposition}

\noindent The proof is exactly like Proposition \ref{PropAssOp}.

Notice that the $\infty$-operads $\text{Assoc}_S$ and $\text{LM}_S$ have the property: A morphism is inert if and only if it lies over an inert morphism in Comm. Therefore, all of the functors in the following diagram are $\infty$-operad maps (where $S\to T$ is a function): $$\xymatrix{
\text{Assoc}_S\ar@{^{(}->}[r]\ar[d] &\text{LM}_S\ar@{^{(}->}[r]\ar[d] &\text{Assoc}_{S_{+}}\ar[d] \\
\text{Assoc}_T\ar@{^{(}->}[r] &\text{LM}_T\ar@{^{(}->}[r] &\text{Assoc}_{T_{+}}.
}$$

\begin{definition}
Suppose $(\mathcal{V};\mathcal{M})$ is an LM-monoidal $\infty$-category; that is, $\mathcal{V}$ is a monoidal $\infty$-category and $\mathcal{M}$ is a left $\mathcal{V}$-module $\infty$-category. A \emph{$\mathcal{V}$-enriched presheaf} with underlying set $S$ and values in $\mathcal{M}$ is an $\text{LM}_S$-algebra in $(\mathcal{V};\mathcal{M})$, and we write $$\text{PSh}^{\mathcal{V};\mathcal{M}}_S=\text{Alg}_{\text{LM}_S/\text{LM}}(\mathcal{V};\mathcal{M}).$$
\end{definition}

\noindent In order to compare to $\text{Cat}^\mathcal{V}$, we will need a small lemma:

\begin{lemma}\label{LemCatPair}
With notation as before, $\text{Cat}^\mathcal{V}\cong\text{Alg}_{\text{Assoc}_S/\text{LM}}(\mathcal{V};\mathcal{M})$.
\end{lemma}

\noindent We will henceforth abuse notation by defining $\text{Cat}^\mathcal{V}=\text{Alg}_{\text{Assoc}_S/\text{LM}}(\mathcal{V};\mathcal{M})$ any time we are working with a pair $(\mathcal{V};\mathcal{M})$.

\begin{proof}
There is a pullback $$\xymatrix{
\smallint\mathcal{V}\ar[r]\ar[d] &\smallint(\mathcal{V};\mathcal{M})\ar[d] \\
\text{Assoc}\ar@{^{(}->}[r] &\text{LM}.
}$$ which is compatible with inert morphisms, in the sense that a morphism in $\smallint\mathcal{V}$ is inert if and only if its images in $\smallint(\mathcal{V};\mathcal{M})$ and Assoc are inert. Therefore, an $\infty$-operad map $\text{Assoc}_S\to\smallint\mathcal{V}$ over Assoc records the same data as an $\infty$-operad map $\text{Assoc}_S\to\smallint(\mathcal{V};\mathcal{M})$ over LM.
\end{proof}

\noindent Now let's unpack the definition. By the lemma, the inclusion $\text{Assoc}_S\subseteq\text{LM}_S$ induces a forgetful functor $$\theta:\text{PSh}^{\mathcal{V};\mathcal{M}}_S\to\text{Cat}^\mathcal{V}_S.$$ In total, the data of an enriched presheaf consists of the restriction to $\text{Assoc}_S$, which is an enriched category $\mathcal{C}$, with the following data and coherences:
\begin{itemize}
\item objects $\mathcal{F}(X,\ast)\in\mathcal{M}$ for each $X\in S$;
\item morphisms $\mathcal{C}(X,Y)\otimes\mathcal{F}(Y,\ast)\to\mathcal{F}(X,\ast)$ for each $X,Y\in S$.
\end{itemize}

\noindent We generally write just $\mathcal{F}(X)$ instead of $\mathcal{F}(X,\ast)$. We read (2) as a method for turning hypothetical `morphisms' $X\to Y$ in $\mathcal{C}$ into maps $\mathcal{F}(Y)\to\mathcal{F}(X)$ in $\mathcal{M}$. This data is something like a `functor' $\mathcal{C}^\text{op}\to\mathcal{M}$, and it is for this reason we call an $\text{LM}_S$-algebra a \emph{presheaf}.

\begin{example}[Presheaves with values in $\mathcal{V}$]
Suppose $\mathcal{V}$ is a monoidal $\infty$-category. LM is an $\infty$-operad over Assoc via $\text{LM}\subseteq\text{Assoc}_{\left\langle 1\right\rangle}\to\text{Assoc}$. The composite $\text{LM}\to\text{Assoc}\xrightarrow{\mathcal{V}}\text{Cat}$ describes an LM-monoidal $\infty$-category $(\mathcal{V};\mathcal{V})$; that is, $\mathcal{V}$ acts on itself by tensoring on the left. We write $$\text{PSh}^\mathcal{V}_S=\text{PSh}^{\mathcal{V};\mathcal{V}}_S,$$ understanding by convention that we are taking presheaves with values in $\mathcal{V}$.

By construction, $\smallint(\mathcal{V};\mathcal{V})\cong\smallint\mathcal{V}\times_\text{Assoc}\text{LM}$, and this pullback is compatible with inert morphisms. Just as in the proof of the last lemma, we conclude $$\text{PSh}^\mathcal{V}_S\cong\text{Alg}_{\text{LM}_S/\text{Assoc}}(\mathcal{V}).$$
\end{example}

\noindent We have constructed the $\infty$-category $\text{PSh}^{\mathcal{V};\mathcal{M}}_S$ of pairs $(\mathcal{C};\mathcal{F})$ where $\mathcal{C}$ is an enriched category and $\mathcal{F}$ is a presheaf. However, we will typically be more interested in fixing $\mathcal{C}$ and studying the $\infty$-category $\text{PSh}^\mathcal{V}(\mathcal{C};\mathcal{M})$ of presheaves on $\mathcal{C}$:

\begin{definition}
If $\mathcal{C}$ is a $\mathcal{V}$-enriched category with set $S$ of objects, then the $\infty$-category of enriched presheaves on $\mathcal{C}$ with values in $\mathcal{M}$ is $$\text{PSh}^\mathcal{V}(\mathcal{C};\mathcal{M})=\text{PSh}^{\mathcal{V};\mathcal{M}}_S\times_{\text{Cat}^\mathcal{V}_S}\{\mathcal{C}\}.$$ When $\mathcal{M}=\mathcal{V}$ with its canonical left action on itself, we will also write $$\text{PSh}^\mathcal{V}(\mathcal{C})=\text{PSh}^\mathcal{V}(\mathcal{C};\mathcal{V}).$$
\end{definition}

\noindent Notice that $\text{PSh}^\mathcal{V}(\mathcal{C};\mathcal{M})$ depends on the set $S$ although it is not specified notationally. (In fact, $S$ is part of the data of $\mathcal{C}$, so in this sense it is implicit in the notation.)

\begin{example}
If the enriched category $\mathcal{C}$ has a single object, let $A$ be the associated algebra (the endomorphism algebra of the single object). In this case, $\text{PSh}^\mathcal{V}(\mathcal{C};\mathcal{M})=\text{LMod}_A(\mathcal{M})$ by definition (HA 4.2.1.13)
\end{example}

\begin{example}\label{ExTriv2}
We will prove later:
\begin{itemize}
\item (Proposition \ref{PropTriv}) If $1_S$ is the trivial $\mathcal{V}$-enriched category with set $S$ of objects (Example \ref{ExTriv1b}) then $\text{PSh}^\mathcal{V}(1_S;\mathcal{M})\cong\mathcal{M}$;
\item (Corollary \ref{CorTriv}) If $0$ is the trivial monoidal $\infty$-category, $\text{PSh}^{0;\mathcal{M}}_S\cong\mathcal{M}$.
\end{itemize}
\end{example}

\begin{remark}\label{RmkFiber}
All of our definitions so far are model-independent, but before we move on, we should say a word about the quasicategory model. Suppose we are given a cocartesian fibration $\smallint(\mathcal{V};\mathcal{M})\to\text{LM}$ of quasicategories exhibiting $\mathcal{M}$ as a left $\mathcal{V}$-module $\infty$-category. Then the constructions $$\text{Cat}^\mathcal{V}_S=\text{Fun}^\dag_{/\text{LM}^\mathsection}(\text{Assoc}_S^\mathsection,\smallint(\mathcal{V};\mathcal{M})^\mathsection)$$ $$\text{PSh}^{\mathcal{V};\mathcal{M}}_S=\text{Fun}^\dag_{/\text{LM}^\mathsection}(\text{LM}_S^\mathsection,\smallint(\mathcal{V};\mathcal{M})^\mathsection)$$ model $\text{Cat}^\mathcal{V}_S$ and $\text{PSh}^{\mathcal{V};\mathcal{M}}_S$ concretely as quasicategories.

Since the inclusion $\text{Assoc}_S\to\text{LM}_S$ is a categorical fibration, the forgetful functor $\theta:\text{PSh}^{\mathcal{V};\mathcal{M}}_S\to\text{Cat}^\mathcal{V}_S$ is also a categorical fibration, so $\text{PSh}^\mathcal{V}(\mathcal{C};\mathcal{M})$ may be modeled as the literal fiber $\theta^{-1}(\mathcal{C})$, or the quasicategory of lifts $$\xymatrix{
\text{Assoc}_S^\mathsection\ar[d]\ar[r]^-{\mathcal{C}} &\smallint\mathcal{V}^\mathsection\ar[d] \\
\text{LM}_S^\mathsection\ar@{-->}[r] &\smallint(\mathcal{V};\mathcal{M})^\mathsection.
}$$
\end{remark}

\subsection{Representable presheaves}\label{S33}
\noindent Suppose $\mathcal{F}$ is a $\mathcal{V}$-enriched presheaf with values in $\mathcal{M}$. That is, $\mathcal{F}$ is a map of $\infty$-operads $\mathcal{F}:\text{LM}_S\to\smallint(\mathcal{V};\mathcal{M})$ for some set $S$. If $X\in S$, then $\mathcal{F}(X,\ast)$ is an object of $\smallint(\mathcal{V};\mathcal{M})$ lying over $(0,\ast)\in\text{LM}_{\{0\}}$. Since the fiber over $(0,\ast)$ is just $\mathcal{M}$, $\mathcal{F}(X,\ast)$ is an object of $\mathcal{M}$.

\begin{definition}
As above, evaluation at $(X,\ast)\in\text{LM}_S$ induces $$\text{ev}_X:\text{PSh}^{\mathcal{V};\mathcal{M}}_S\to\mathcal{M}.$$ We call this functor \emph{evaluation at $X$} and also write $\mathcal{F}(X)=\text{ev}_X(\mathcal{F})$.

We may also restrict to the fibers $\text{PSh}^\mathcal{V}(\mathcal{C};\mathcal{M})\subseteq\text{PSh}^{\mathcal{V};\mathcal{M}}_S$, yielding functors $\text{ev}_X:\text{PSh}^\mathcal{V}(\mathcal{C};\mathcal{M})\to\mathcal{M}$ for each enriched category $\mathcal{C}$.
\end{definition}

\noindent We will now construct the representable presheaves. For any $X\in S$, there is a function $\pi^X:S_{+}\to S$ given by $\pi^X(\ast)=X$ and $\pi^X(Y)=Y$. This induces an operad map $\pi^X_\ast:\text{LM}_S\subseteq\text{Assoc}_{S+}\to\text{Assoc}_S$, and composition with $\pi^X_\ast$ induces $\text{rep}_X:\text{Cat}^\mathcal{V}_S\to\text{PSh}^\mathcal{V}_S$.

The inclusion $\text{Assoc}_S\to\text{LM}_S$ is a section of $\pi^X_\ast$, so $\text{rep}_X(\mathcal{C})$ is a presheaf on $\mathcal{C}$.

\begin{definition}
If $\mathcal{C}$ is a $\mathcal{V}$-enriched category with set $S$ of objects and $X\in S$, then $\text{rep}_X(\mathcal{C})$ is the \emph{representable presheaf} at $X$. When $\mathcal{C}$ is clear from context, we will just write $\text{rep}_X$.
\end{definition}

\noindent Notice that $\text{rep}_X(Y)=\mathcal{C}(Y,X)$ by construction. The representable presheaves are \emph{free}, in the following sense.

\begin{definition}\label{DefFree}
Consider $(\mathcal{C},\mathcal{F})\in\text{PSh}^{\mathcal{V};\mathcal{M}}_S$, so that $\mathcal{C}$ is an enriched category and $\mathcal{F}$ a presheaf on $\mathcal{C}$. If $X\in S$ and $M\in\mathcal{M}$, we say that a morphism $\lambda:M\to\mathcal{F}(X)$ exhibits $\mathcal{F}$ as \emph{freely generated by $M$ at $X$} if for all $Y\in S$, the map $\mathcal{C}(Y,X)\otimes M\to\mathcal{C}(Y,X)\otimes\mathcal{F}(X)\to\mathcal{F}(Y)$ is an equivalence.
\end{definition}

\begin{example}
Let $1$ be the unit of the monoidal structure on $\mathcal{V}$. Then $\text{rep}_X\in\text{PSh}^\mathcal{V}(\mathcal{C};\mathcal{V})$ is freely generated by $1$ at $X$.
\end{example}

\begin{theorem}\label{ThmFree}
For any $\mathcal{C}\in\text{Cat}^\mathcal{V}_S$, $X\in\mathcal{C}$, and $M\in\mathcal{M}$, there exists a presheaf $\mathcal{F}\in\text{PSh}^\mathcal{V}(\mathcal{C};\mathcal{M})$ which is freely generated by $M$ at $X$. Moreover, for any $(\mathcal{D},\mathcal{G})\in\text{PSh}^{\mathcal{V};\mathcal{M}}_S$, composition with $\lambda$ induces an equivalence $$\text{Map}_{\text{PSh}^{\mathcal{V};\mathcal{M}}_S}((\mathcal{C},\mathcal{F}),(\mathcal{D},\mathcal{G}))\to\text{Map}_{\text{Cat}^\mathcal{V}_S}(\mathcal{C},\mathcal{D})\times\text{Map}_\mathcal{M}(M,\mathcal{G}(X)).$$
\end{theorem}

\noindent This is the main result of this section, but it is not the theorem itself that interests us so much as its corollaries. We record those corollaries before proving the theorem.

\begin{corollary}
If $\mathcal{C}\in\text{Cat}^\mathcal{V}_S$, $X\in S$, and $M\in\mathcal{M}$, then there exists a presheaf in $\text{PSh}^\mathcal{V}(\mathcal{C};\mathcal{M})$ which is freely generated by $M$ at $X$, and it is essentially unique in the following sense:

If $\mathcal{F}^0,\mathcal{F}^1$ are two such presheaves, there is an equivalence $\mathcal{F}^0\to\mathcal{F}^1$ which is compatible with the maps $\lambda^i:M\to\mathcal{F}^i(X)$ and is unique up to homotopy.
\end{corollary}

\noindent We refer to the presheaf freely generated by $M$ at $X$ as $\text{rep}_X\otimes M$, notation justified by the formula $(\text{rep}_X\otimes M)(Y)\cong\mathcal{C}(Y,X)\otimes M\cong\text{rep}_X(Y)\otimes M$.

\begin{corollary}[Enriched Yoneda lemma, weak form]\label{CorYon}
If $(\mathcal{C},\mathcal{F})\in\text{PSh}^{\mathcal{V};\mathcal{M}}_S$, $X\in S$, and $M\in\mathcal{M}$, then $\lambda:M\to(\text{rep}_X\otimes M)(X)$ induces an equivalence $$\text{Map}_{\text{PSh}^\mathcal{V}(\mathcal{C};\mathcal{M})}(\text{rep}_X\otimes M,\mathcal{F})\to\text{Map}_\mathcal{M}(M,\mathcal{F}(X)).$$
\end{corollary}

\begin{example}
If $\mathcal{M}=\mathcal{V}$ and $M=1$ is the monoidal unit, then the last corollary asserts $\text{Map}(\text{rep}_X,\mathcal{F})\cong\text{Map}(1,\mathcal{F}(X))$.
\end{example}

\begin{corollary}\label{CorLAdj}
The functor $\text{ev}_X:\text{PSh}^\mathcal{V}(\mathcal{C};\mathcal{M})\to\mathcal{M}$ has a left adjoint described by $\text{rep}_X\otimes -:\mathcal{M}\to\text{PSh}^\mathcal{V}(\mathcal{C};\mathcal{M})$.
\end{corollary}

\noindent We can use this last corollary to prove the claims of Example \ref{ExTriv2}:

\begin{proposition}\label{PropTriv}
Let $1_S$ be the trivial $\mathcal{V}$-enriched category of Example \ref{ExTriv1b}. For any $X\in S$, the functor $\text{ev}_X:\text{PSh}^\mathcal{V}(\ast_X;\mathcal{M})\to\mathcal{M}$ is an equivalence.
\end{proposition}

\begin{proof}
Suppose that $\mathcal{F}$ is a presheaf on $\ast_X$. For each $X,Y\in S$, we have structure maps $c_{X,Y}:\mathcal{F}(Y)=1_S(X,Y)\otimes\mathcal{F}(Y)\to\mathcal{F}(X)$. We claim $c_{X,Y}$ is an equivalence.

Since $\mathcal{F}(X)\xrightarrow{\text{id}}1_S(X,X)\otimes\mathcal{F}(X)\xrightarrow{c_{X,X}}\mathcal{F}(X)$ is the identity, $c_{X,X}$ is the identity on $\mathcal{F}(X)$. Consider the commutative square $$\xymatrix{
1_S(X,Y)\otimes1_S(Y,X)\otimes\mathcal{F}(X)\ar[r]^-{c_{Y,X}}\ar@{=}[d] &1_S(X,Y)\otimes\mathcal{F}(Y)\ar[d]^-{c_{X,Y}} \\
1_S(X,X)\otimes\mathcal{F}(X)\ar[r]_-{c_{X,X}} &\mathcal{F}(X),
}$$ it follows that $c_{X,Y}$ and $c_{Y,X}$ are inverse to each other, hence are equivalences.

In other words, the map $1_S(Y,X)\otimes\mathcal{F}(X)\to\mathcal{F}(Y)$ is an equivalence for all $Y$, which means that the identity $\mathcal{F}(X)\to\mathcal{F}(X)$ exhibits $\mathcal{F}$ as freely generated by $\mathcal{F}(X)$ at $X$, so $\mathcal{F}\cong\text{rep}_X\otimes\mathcal{F}(X)$. In other words, the unit of the adjunction $\text{ev}_X:\text{PSh}^\mathcal{V}(\ast_X;\mathcal{M})\leftrightarrows\mathcal{M}:\text{rep}_X\otimes -$ is an equivalence. On the other hand, the counit $(\text{rep}_X\otimes M)(X)\to M$ is also an equivalence because $\text{rep}_X(X)\cong 1$.

Hence $\text{ev}_X$ and $\text{rep}_X\otimes -$ is a pair of inverse functors, and $\text{ev}_X$ is an equivalence.
\end{proof}

\begin{corollary}\label{CorTriv}
Let $0$ be the trivial monoidal $\infty$-category as in Example \ref{ExTriv1}. Then any $\infty$-category $\mathcal{M}$ has a (unique) trivial action of $0$. By Example \ref{ExTriv1}, $\text{Cat}^0_S$ is contractible. By Proposition \ref{PropTriv}, $\text{ev}_X:\text{PSh}^0(1_S;\mathcal{M})\to\mathcal{M}$ is an equivalence. Therefore $\theta:\text{PSh}^{0;\mathcal{M}}_S\to\text{Cat}^0_S$ is equivalent to the functor $\mathcal{M}\to\ast$; that is, we have an equivalence $$\text{ev}_X:\text{PSh}^{0;\mathcal{M}}_S\to\mathcal{M}.$$
\end{corollary}

\noindent The rest of this section constitutes a proof of Theorem \ref{ThmFree}, following HA 4.2.4.2. The proof is a straightforward application of HA 3.1.3 (free algebras over $\infty$-operads), but in order to apply it, we will need to introduce some notation. None of the notation will reappear in this paper. Define $\text{LM}_{S,X}$ to be the subcategory of $\text{LM}_S$ spanned by:
\begin{itemize}
\item graphs $\Gamma\in\text{LM}_S$ such that every edge which terminates at $\ast$ originates at $X$;
\item morphisms $f:\Gamma\to\Gamma^\prime$ such that for any edge $e\in\Gamma$, $f(e)$ terminates at $\ast$ if and only if $e$ terminates at $\ast$.
\end{itemize}
\noindent Let Triv denote the trivial $\infty$-operad of HA 2.1.1.20 (the subcategory of Comm spanned by all objects and inert morphisms), and $\boxtimes$ the coproduct of $\infty$-operads. These satisfy the universal properties $\text{Alg}_\text{Triv}(\mathcal{V})\cong\mathcal{V}$ and $\text{Alg}_{\mathcal{O}\boxtimes\mathcal{O}^\prime}(\mathcal{V})\cong\text{Alg}_\mathcal{O}(\mathcal{V})\times\text{Alg}_{\mathcal{O}^\prime}(\mathcal{V})$ by HA 2.1.3.5, respectively 2.2.3.6.

Consider the full subcategories of $\text{LM}_{S,X}$ spanned by edges of the form $(X,\ast)$, respectively edges which do not terminate at $\ast$. These subcategories are isomorphic to $\text{Triv}$, respectively $\text{Assoc}_S$, and the induced map of $\infty$-operads $\text{Assoc}_S\boxtimes\text{Triv}\to\text{LM}_{S,X}$ is an isomorphism by construction of $\boxtimes$ (HA 2.2.3.3). Therefore, the inclusion $\text{LM}_{S,X}\subseteq\text{LM}_S$ induces a functor $$\xi:\text{PSh}^{\mathcal{V};\mathcal{M}}_S=\text{Alg}_{\text{LM}_S/\text{LM}}(\mathcal{V};\mathcal{M})\to\text{Alg}_{\text{LM}_{S,X}/\text{LM}}(\mathcal{V};\mathcal{M})\cong\text{Cat}^\mathcal{V}_S\times\mathcal{M}$$ defined by $\xi(\mathcal{C},\mathcal{F})=(\mathcal{C},\mathcal{F}(X))$. Now we can formulate:

\begin{lemma}
Given algebras $(\mathcal{C},M)\in\text{Cat}^\mathcal{V}_S\times\mathcal{M}\cong\text{Alg}_{\text{LM}_{S,X}/\text{LM}}(\mathcal{V};\mathcal{M})$ and $(\bar{\mathcal{C}},\mathcal{F})\in\text{PSh}^{\mathcal{V};\mathcal{M}}_S=\text{Alg}_{\text{LM}_S/\text{LM}}(\mathcal{V};\mathcal{M})$, as well as a morphism \\$\lambda:(\mathcal{C},M)\to(\bar{\mathcal{C}},\mathcal{F}(X))=\xi(\bar{\mathcal{C}},\mathcal{F})$, the following are equivalent:
\begin{enumerate}
\item $\lambda$ exhibits $(\bar{\mathcal{C}},\mathcal{F})$ as the free $\text{LM}_S$-algebra generated by $(\mathcal{C},M)$ in the sense of HA 3.1.3.1;
\item $\lambda_0:\mathcal{C}\to\bar{\mathcal{C}}$ is an equivalence and $\lambda_1:M\to\mathcal{F}(X)$ exhibits $\mathcal{F}$ as freely generated by $M$ at $X$ in the sense of Definition \ref{DefFree}.
\end{enumerate}
\end{lemma}

\begin{proof}
We recall HA Definition 3.1.3.1 for reference (the case $\mathcal{O}=\text{LM}$, $\mathcal{A}=\text{LM}_{S,X}$, and $\mathcal{B}=\text{LM}_S$).

If $\Gamma\in\text{LM}_S$, define $\mathcal{A}^\text{act}_{/\Gamma}=\text{LM}_{S,X}\times_{\text{LM}_S}(\text{LM}_S^\text{act})_{/\Gamma}$, where $\text{LM}_S^\text{act}$ is the subcategory of $\text{LM}_S$ spanned by all objects and active morphisms.

Then (1) asserts that the induced map $\alpha_\Gamma:(\mathcal{A}^\text{act}_{/\Gamma})^\triangleright\to\smallint(\mathcal{V};\mathcal{M})$ is an operadic colimit diagram for all graphs $\Gamma$ with a single edge. To unpack this, we split into two cases:
\begin{enumerate}[(a)]
\item If $\Gamma=(A,B)$ for any $A,B\in S$, then $\Gamma\in\mathcal{A}^\text{act}_{/\Gamma}$ is terminal. Therefore, $\alpha_\Gamma$ is an operadic colimit diagram if and only if $\lambda:\mathcal{C}(A,B)\to\bar{\mathcal{C}}(A,B)$ is an equivalence.
\item If $\Gamma=(A,\ast)$ for any $A\in S$, then $(A,X)\otimes(X,\ast)\in\mathcal{A}^\text{act}_{/\Gamma}$ is terminal, so $\alpha_\Gamma$ is an operadic colimit diagram if and only if $\lambda:\mathcal{C}(A,X)\otimes M\to\mathcal{F}(X)$ is an equivalence.
\end{enumerate}
\noindent Hence, the statements (1) and (2) unpack to the same conditions.
\end{proof}

\begin{proof}[Proof of Theorem \ref{ThmFree}]
In light of the lemma, HA 3.1.3.3 asserts that free presheaves exist (noting, as in the proof of the lemma, that $\mathcal{A}^\text{act}_{/\Gamma}$ has a terminal object for each $\Gamma\in\text{LM}_S$ with a single edge). The equivalence in the theorem statement is a restatement of HA 3.1.3.2.
\end{proof}

\section{The $\mathbb{A}_\infty$-model for enriched categories}\label{S4}
\noindent Recall that a functor $p:\mathcal{A}\to\mathcal{B}$ is called a \emph{presentable fibration} if either of the following equivalent conditions hold (HTT 5.5.3.3):
\begin{itemize}
\item $p$ is cartesian and the functor $\mathcal{B}^\text{op}\to\widehat{\text{Cat}}$ factors through $\text{Pr}^R\subseteq\widehat{\text{Cat}}$;
\item $p$ is cocartesian and the functor $\mathcal{B}\to\widehat{\text{Cat}}$ factors through $\text{Pr}^L\subseteq\widehat{\text{Cat}}$.
\end{itemize}

\begin{definition}
If $\mathcal{V}$ is a monoidal $\infty$-category and $\mathcal{M}$ is a left $\mathcal{V}$-module, we say that $(\mathcal{V};\mathcal{M})$ is a \emph{presentable pair} if the equivalent conditions hold:
\begin{itemize}
\item $\mathcal{V}$ and $\mathcal{M}$ are both presentable, $\mathcal{V}$ is closed monoidal ($\mathcal{V}\otimes\mathcal{V}\to\mathcal{V}$ preserves colimits independently in each variable), and $\mathcal{M}$ is a presentable left $\mathcal{V}$-module ($\mathcal{V}\times\mathcal{M}\to\mathcal{M}$ preserves colimits independently in each variable);
\item The pair $(\mathcal{V};\mathcal{M})$, which is a priori an LM-algebra in Cat, restricts to an LM-algebra in $\text{Pr}^L$.
\end{itemize}
\end{definition}

\noindent Our primary goal in this section is to prove:

\begin{reptheorem}{PShPrL}
If $(\mathcal{V};\mathcal{M})$ is a presentable pair and $\mathcal{C}$ is $\mathcal{V}$-enriched, then $\text{PSh}^\mathcal{V}(\mathcal{C};\mathcal{M})$ is presentable.
\end{reptheorem}

\noindent We will also describe how to compute limits and colimits in $\text{PSh}^\mathcal{V}(\mathcal{C};\mathcal{M})$:

\begin{repcorollary}{CorPFib1b}
If $(\mathcal{V};\mathcal{M})$ is a presentable pair and $\mathcal{C}$ is $\mathcal{V}$-enriched:
\begin{enumerate}
\item A functor $p:K^\triangleleft\to\text{PSh}^\mathcal{V}(\mathcal{C};\mathcal{M})$ is a limit diagram if and only if $\text{ev}_Xp:K^\triangleleft\to\mathcal{M}$ is a limit diagram for all $X\in\mathcal{C}$;
\item A functor $p:K^\triangleright\to\text{PSh}^\mathcal{V}(\mathcal{C};\mathcal{M})$ is a colimit diagram if and only if $\text{ev}_Xp:K^\triangleright\to\mathcal{M}$ is a colimit diagram for all $X\in\mathcal{C}$;
\end{enumerate}
\end{repcorollary}

\noindent These and other important structural results appear in Section \ref{S43}. Their proofs rely on the $\mathbb{A}_\infty$-model for enriched categories and presheaves, technical tools which we first introduce in Sections \ref{S41} and \ref{S42}.

We recommend that the reader begin by skimming the results in Section \ref{S43}, and then return to Sections \ref{S41} and \ref{S42} before attempting to understand the proofs.

\subsection{Enriched categories}\label{S41}
\noindent Define $\Delta_{/S}=\Delta\times_\text{Set}\text{Set}_{/S}$. In this way, an object of $\Delta_{/S}$ is simply a function $\ell:[n]\to S$. For ease of exposition, we will refer to this as the object $X=\{X_0<\cdots<X_n\}$ when $\ell(i)=X_i$. We think of $X$ as the ordered set $[n]=\{0<\cdots<n\}$ along with a labeling of each element in $S$. A morphism $X\to Y$ is an order-preserving function $f:[n]\to[m]$ which preserves the labeling, $Y_{f(i)}=X_i$.

We begin by introducing a strong approximation $\text{Cut}_S:\Delta_{/S}^\text{op}\to\text{Assoc}_S$ generalizing Example \ref{ExCut} (which corresponds to the case $|S|=1$).

We may imagine the totally ordered set $[n]=\{0<\cdots<n\}$ as a directed graph on the set $\{0,\ldots,n\}$: $$\xymatrix{
0\ar[r] &1\ar[r] &\cdots\ar[r] &n,
}$$ and therefore as an object of $\text{Assoc}_{[n]}$. If $\ell:[n]\to S$ is an object of $\Delta_{/S}$, push forward along $\ell_\ast:\text{Assoc}_{[n]}\to\text{Assoc}_S$ to a graph $\text{Cut}_S(\ell)\in\text{Assoc}_S$.

Concretely, if $X=\{X_0<\cdots<X_n\}\in\Delta_{/S}$, then $$\text{Cut}_S(X)=(X_0,X_1)\otimes(X_1,X_2)\otimes\cdots\otimes(X_{n-1},X_n).$$ If $Y=\{Y_0<\cdots<Y_m\}$ and $f:X\to Y$ is a map in $\Delta_{/S}$, then there is an induced map $f^\ast:\text{Cut}_S(Y)\to\text{Cut}_S(X)$ as follows: $f^\ast$ sends the edge $(Y_{i-1},Y_i)$ to the edge $(X_{j-1},X_j)$ if $f(j-1)<i\leq f(j)$. If there is no such $j$, $f^\ast$ sends $(Y_{i-1},Y_i)$ to the basepoint. The total ordering on $[n]$ induces total orderings on the fibers of $f^\ast$, and $f^\ast$ satisfies the two properties of Definition \ref{DefAss}, so that we have a functor $$\text{Cut}_S:\Delta_{/S}^\text{op}\to\text{Assoc}_S.$$ The composite $\Delta_{/S}^\text{op}\xrightarrow{\text{Cut}_S}\text{Assoc}_S\to\text{Comm}$ makes $\Delta_{/S}^\text{op}$ an $\infty$-preoperad.

\begin{remark}
Recall (Example \ref{ExCut}) that a morphism $[n]\to[m]$ of $\Delta^\text{op}$ is inert if it embeds $[m]$ as a convex subset $\{i<\cdots<i+m\}\subseteq[n]$. Then a morphism of $\Delta_{/S}^\text{op}$ is inert if and only if the underlying morphism in $\Delta^\text{op}$ is inert.

In particular, we can verify any such morphism is $p$-cocartesian, just as in the proof of Proposition \ref{PropAssOp}.
\end{remark}

\begin{proposition}\label{PropCutS}
The functor $\text{Cut}_S:\Delta_{/S}^\text{op}\to\text{Assoc}_S$ is a strong approximation to the $\infty$-operad $\text{Assoc}_S$.
\end{proposition}

\begin{proof}
Both $\Delta_{/S}^\text{op}$ and $\text{Assoc}_S$ have the property: A morphism is inert if and only if it lies over an inert morphism in $\text{Comm}$. Therefore, $\text{Cut}_S$ sends inert morphisms to inert morphisms, so it is a morphism of $\infty$-preoperads.

We need to check the two conditions of Definition \ref{DefApprox} to prove that $\text{Cut}_S$ is an approximation. Since the fibers of $\Delta_{/S}^\text{op}$ and $\text{Assoc}_S$ over $\left\langle 1\right\rangle\in\text{Comm}$ are each equivalent to the (discrete) set $S$, it is clear that $\text{Cut}_S$ is a strong approximation if and only if it is an approximation.

(1) Note that $p:\Delta_{/S}^\text{op}\to\text{Comm}$ is given by $p(X_0<\cdots<X_n)=\left\langle n\right\rangle$, where $i\in\left\langle n\right\rangle$ corresponds to the edge $(X_{i-1},X_i)$ in $\text{Cut}_S(X_0<\cdots<X_n)$. Given an inert morphism $\phi:\{0,\ldots,n-1\}\to\left\langle 1\right\rangle$, let $i=\phi^{-1}(1)$, and let $\{X_i<X_{i+1}\}\to\{X_0<\cdots<X_n\}$ be the natural inclusion. This describes an inert morphism $\{X_0<\cdots<X_n\}\to\{X_i<X_{i+1}\}$ in $\Delta_{/S}^\text{op}$ lifting $\phi$, so $\text{Cut}_S$ satisfies Definition \ref{DefApprox}(1).

(2) Fix an active morphism $\phi:\Gamma\to(Y_0,Y_1)\otimes\cdots\otimes(Y_{m-1},Y_m)$ in $\text{Assoc}_S$. That is, $\phi$ transforms the graph $\Gamma$ into $(Y_0,Y_1)\otimes\cdots\otimes(Y_{m-1},Y_m)$ via repeated application of the two moves:
\begin{itemize}
\item $(A,B)\otimes(B,C)\to(A,C)$;
\item $\emptyset\to(A,A)$.
\end{itemize}
\noindent Recall that the morphism $\phi$ includes the data of total orderings on the fibers $\phi^{-1}(Y_i,Y_{i+1})$. Taken together, these induce a total ordering on the edges of $\Gamma$ such that they form a path from $Y_0$ to $Y_m$. Hence we may write $\Gamma=(X_0,X_1)\otimes\cdots\otimes(X_{n-1},X_n)$ where $X_0=Y_0$ and $X_n=Y_m$.

Define $0=k_0\leq\cdots\leq k_m=n$ such that $\phi(X_j,X_{j+1})=(Y_i,Y_{i+1})$ whenever $k_i\leq j<k_{i+1}$. That is, $(X_{k_i},X_{k_i+1})\otimes\cdots\otimes(X_{k_{i+1}-1},X_{k_{i+1}})$ is the fiber $\phi^{-1}(X_i,X_{i+1})$, or the fiber is empty when $k_i=k_{i+1}$. Since $\phi^{-1}(X_i,X_{i+1})$ must form a path from $X_i$ to $X_{i+1}$, we conclude $X_{k_i}=Y_i$ for all $i$. Therefore, the indices $k_0,\ldots,k_m$ describe a morphism $$\bar{\phi}:\{X_0<\cdots<X_n\}\to\{Y_0<\cdots<Y_m\}$$ of $\Delta_{/S}^\text{op}$ lifting $\phi$. By construction, $\phi$ is universal among such lifts of $\phi$, which is to say it is $f$-cartesian.

Thus $\text{Cut}_S$ satisfies Definition \ref{DefApprox}(2), which completes the proof.
\end{proof}

\begin{definition}
If $\mathcal{V}$ is a monoidal $\infty$-category, an \emph{$\mathbb{A}_\infty$-$\mathcal{V}$-enriched category} with set $S$ of objects is a $\Delta_{/S}^\text{op}$-algebra in $\mathcal{V}$, and they form an $\infty$-category $$\mathbb{A}_\infty\text{Cat}^\mathcal{V}_S=\text{Alg}_{\Delta_{/S}^\text{op}/\text{Assoc}}(\mathcal{V}).$$
\end{definition}

\noindent Applying Theorem \ref{ThmApprox}, we find:

\begin{corollary}\label{CorModelEq1}
$\text{Cut}_S$ induces an equivalence $\text{Cat}^\mathcal{V}_S\to\mathbb{A}_\infty\text{Cat}^\mathcal{V}_S$.
\end{corollary}

\begin{remark}
If $\mathcal{V}$ is a monoidal $\infty$-category, recall from Section \ref{S25} that we can regard $\mathcal{V}$ either as a functor $\mathcal{V}:\text{Assoc}\to\text{Cat}$, with associated $\infty$-operad $\smallint\mathcal{V}\to\text{Assoc}$, or as a functor $\mathcal{BV}:\Delta^\text{op}\to\text{Cat}$, with associated planar $\infty$-operad $\smallint\mathcal{BV}\to\Delta^\text{op}$. By construction we have a pullback $$\xymatrix{
\smallint\mathcal{BV}\ar[r]\ar[d] &\smallint\mathcal{V}\ar[d] \\
\Delta^\text{op}\ar[r]_-{\text{Cut}} &\text{Assoc},
}$$ so $\mathbb{A}_\infty\text{Cat}^\mathcal{V}_S\cong\text{Fun}^\dag_{/\Delta^\text{op}}(\Delta_{/S}^{\text{op}\mathsection},\smallint\mathcal{BV}^\mathsection)$.
\end{remark}

\subsection{Enriched presheaves}\label{S42}
\noindent We have just shown that $\Delta_{/S}^\text{op}$ is a strong approximation to the $\infty$-operad $\text{Assoc}_S$, which means that we can identify $\mathcal{V}$-enriched categories with $\Delta_{/S}^\text{op}$-algebras in $\mathcal{V}$.

Now we will show that $\Delta_{/S}^\text{op}\times\Delta^1$ is a strong approximation to the $\infty$-operad $\text{LM}_S$. First, we construct the functor $\text{LCut}_S:\Delta_{/S}^\text{op}\times\Delta^1\to\text{LM}_S$. This will closely parallel the construction of $\text{Cut}_S$ in Section \ref{S41}.

Given $[n]\in\Delta$, consider the following graphs on $\{0,\ldots,n\}_{+}$: $$[n]_0=\left[0\to 1\to\cdots\to n\to\ast\right]$$ $$[n]_1=\left[0\to 1\to\cdots\to n\right],$$ which are objects of $\text{LM}_{[n]}$. Suppose $[n]\xrightarrow{\ell}S$ is an object of $\Delta_{/S}$, inducing $\ell_\ast:\text{LM}_{[n]}\to\text{LM}_S$. We write $\text{LCut}_S^0(\ell)=\ell_\ast[n]_0$ and $\text{LCut}_S^1(\ell)=\ell_\ast[n]_1$, which are left-modular graphs on the set $S_{+}$. Exactly as in Section \ref{S41}, we have functors $$\text{LCut}_S^0,\text{LCut}_S^1:\Delta_{/S}^\text{op}\to\text{LM}_S.$$

\begin{remark}
The functor $\text{LCut}_S^1$ is essentially the same as $\text{Cut}_S$; that is, it factors $\Delta_{/S}^\text{op}\xrightarrow{\text{Cut}_S}\text{Assoc}_S\subseteq\text{LM}_S$.
\end{remark}

\noindent There are inert morphisms $[n]_0\to[n]_1$ which send the edge $n\to\ast$ to the basepoint, and act as the identity function on the other edges. These assemble into a natural transformation $\text{LCut}_S^0\to\text{LCut}_S^1$; that is, a functor $\Delta^1\to\text{Fun}(\Delta_{/S}^\text{op},\text{LM}_S)$. There is a corresponding functor $$\text{LCut}_S:\Delta_{/S}^\text{op}\times\Delta^1\to\text{LM}_S.$$ The composite with $\text{LM}_S\to\text{Comm}$ makes $\Delta_{/S}^\text{op}\times\Delta^1$ an $\infty$-preoperad.

\begin{remark}
Recall (Example \ref{ExCut2}) that a morphism $([n]\xrightarrow{f}[m],i\to j)$ of $\Delta^\text{op}\times\Delta^1$ is inert if $f$ is inert in $\Delta^\text{op}$ and either:
\begin{itemize}
\item $f(m)=n$;
\item or $j=1$.
\end{itemize}
A morphism of $\Delta_{/S}^\text{op}\times\Delta^1$ is inert if and only if the underlying morphism in $\Delta^\text{op}\times\Delta^1$ is inert.
\end{remark}

\begin{proposition}\label{PropLCut}
The functor $\text{LCut}_S:\Delta_{/S}^\text{op}\times\Delta^1\to\text{LM}_S$ is a strong approximation to the $\infty$-operad $\text{LM}_S$.
\end{proposition}

\begin{proof}
The proof is just like Proposition \ref{PropCutS}; see also HA 4.2.2.8.
\end{proof}

\begin{definition}
Suppose $(\mathcal{V};\mathcal{M})$ is an LM-monoidal $\infty$-category; that is, $\mathcal{V}$ is a monoidal $\infty$-category, and $\mathcal{M}$ is a left $\mathcal{V}$-module $\infty$-category.

An $\mathbb{A}_\infty$-$\mathcal{V}$-enriched presheaf with set $S$ of objects is a $\Delta_{/S}^\text{op}\times\Delta^1$-algebra in $(\mathcal{V};\mathcal{M})$, and they form an $\infty$-category $$\mathbb{A}_\infty\text{PSh}_S^{\mathcal{V};\mathcal{M}}=\text{Alg}_{\Delta^\text{op}_{/S}\times\Delta^1/\text{LM}}(\mathcal{V};\mathcal{M}).$$
\end{definition}

\noindent As in Lemma \ref{LemCatPair}, $\mathbb{A}_\infty\text{Cat}^\mathcal{V}_S\cong\text{Alg}_{\Delta_{/S}^\text{op}/\text{LM}}(\mathcal{V};\mathcal{M})$, so composition with the inclusion $\Delta_{/S}^\text{op}\times\{1\}\to\Delta_{/S}^\text{op}\times\Delta^1$ induces a forgetful functor $$\theta:\mathbb{A}_\infty\text{PSh}_S^{\mathcal{V};\mathcal{M}}\to\mathbb{A}_\infty\text{Cat}^\mathcal{V}_S.$$

\begin{definition}
If $\mathcal{C}\in\mathbb{A}_\infty\text{Cat}^\mathcal{V}_S$, then the $\infty$-category of $\mathbb{A}_\infty$-$\mathcal{V}$-enriched presheaves on $\mathcal{C}$ with values in $\mathcal{M}$ is the fiber $$\mathbb{A}_\infty\text{PSh}^\mathcal{V}(\mathcal{C};\mathcal{M})=\mathbb{A}_\infty\text{PSh}^{\mathcal{V};\mathcal{M}}_S\times_{\mathbb{A}_\infty\text{Cat}^\mathcal{V}_S}\{\mathcal{C}\}.$$
\end{definition}

\begin{corollary}\label{CorModelEq2}
Let $\mathcal{C}$ be a $\mathcal{V}$-enriched category with set $S$ of objects, and $\bar{\mathcal{C}}$ the corresponding $\mathbb{A}_\infty$-$\mathcal{V}$-enriched category. Then composition with $\text{LCut}_S$ induces equivalences $$\text{PSh}^{\mathcal{V};\mathcal{M}}_S\to\mathbb{A}_\infty\text{PSh}^{\mathcal{V};\mathcal{M}}_S,$$ $$\text{PSh}^\mathcal{V}(\mathcal{C};\mathcal{M})\to\mathbb{A}_\infty\text{PSh}^\mathcal{V}(\bar{\mathcal{C}};\mathcal{M}).$$
\end{corollary}

\begin{proof}
The first equivalence follows from Proposition \ref{PropLCut} by Theorem \ref{ThmApprox}. In fact, we have a commutative square $$\xymatrix{
\text{PSh}^{\mathcal{V};\mathcal{M}}_S\ar[r]\ar[d] &\mathbb{A}_\infty\text{PSh}^{\mathcal{V};\mathcal{M}}_S\ar[d] \\
\text{Cat}^\mathcal{V}_S\ar[r] &\mathbb{A}_\infty\text{Cat}^\mathcal{V}_S,
}$$ where the horizontal maps are equivalences, and so the second equivalence follows by taking fibers.
\end{proof}

\noindent Suppose $\mathcal{V}$ is a monoidal $\infty$-category and $\mathcal{M}$ is a left $\mathcal{V}$-module. Recall from Section \ref{S25} that we can regard the pair $(\mathcal{V};\mathcal{M})$ as either a functor $(\mathcal{V};\mathcal{M}):\text{LM}\to\text{Cat}$ with corresponding $\infty$-operad $\smallint(\mathcal{V};\mathcal{M})\to\text{LM}$, or as a functor $\mathcal{B}(\mathcal{V};\mathcal{M}):\Delta^\text{op}\times\Delta^1\to\text{Cat}$ with corresponding planar $\infty$-operad $\smallint\mathcal{B}(\mathcal{V};\mathcal{M})\to\Delta^\text{op}\times\Delta^1$. By construction we have a pullback $$\xymatrix{
\smallint\mathcal{B}(\mathcal{V};\mathcal{M})\ar[r]\ar[d] &\smallint(\mathcal{V};\mathcal{M})\ar[d] \\
\Delta^\text{op}\times\Delta^1\ar[r]_-{\text{Cut}} &\text{LM},
}$$ so $\mathbb{A}_\infty\text{PSh}^{\mathcal{V};\mathcal{M}}_S\cong\text{Fun}^\dag_{/\Delta^\text{op}\times\Delta^1}(\Delta_{/S}^\text{op}\times\Delta^{1\mathsection},\smallint\mathcal{B}(\mathcal{V};\mathcal{M})^\mathsection)$.

\begin{remark}\label{RmkFiber2}
Although everything we have said in this section is model-independent, now we will say a word about quasicategories, paralleling Remark \ref{RmkFiber} for the operadic model. Suppose we are given a cocartesian fibration of quasicategories $\smallint\mathcal{B}(\mathcal{V};\mathcal{M})\to\Delta^\text{op}\times\Delta^1$ exhibiting $\mathcal{M}$ as a left $\mathcal{V}$-module $\infty$-category in the $\mathbb{A}_\infty$-sense. Then we have the following explicit quasicategory models: $$\mathbb{A}_\infty\text{Cat}^\mathcal{V}_S=\text{Fun}^\dag_{/\Delta^\text{op}\times\Delta^1}(\Delta^{\text{op}\mathsection}_{/S},\smallint\mathcal{B}(\mathcal{V};\mathcal{M})^\mathsection);$$ $$\mathbb{A}_\infty\text{PSh}^{\mathcal{V};\mathcal{M}}_S=\text{Fun}^\dag_{/\Delta^\text{op}\times\Delta^1}(\Delta^\text{op}_{/S}\times\Delta^{1\mathsection},\smallint\mathcal{B}(\mathcal{V};\mathcal{M})^\mathsection).$$ $\Delta_{/S}^\text{op}$ lies over $\Delta^\text{op}\times\Delta^1$, as always, by $\Delta_{/S}^\text{op}\to\Delta^\text{op}\times\{1\}\subseteq\Delta^\text{op}\times\Delta^1$. Since the inclusion $\Delta^\text{op}_{/S}\to\Delta^\text{op}_{/S}\times\Delta^1$ is a categorical fibration, the forgetful functor $\theta:\mathbb{A}_\infty\text{PSh}_S^{\mathcal{V};\mathcal{M}}\to\mathbb{A}_\infty\text{Cat}^\mathcal{V}_S$ is also a categorical fibration, so $\mathbb{A}_\infty\text{PSh}^\mathcal{V}(\mathcal{C};\mathcal{M})$ may be modeled as the literal fiber $\theta^{-1}(\mathcal{C})$, which is to say the quasicategory of lifts $$\xymatrix{
\Delta_{/S}^{\text{op}\mathsection}\ar[r]^-{\mathcal{C}}\ar[d] &\smallint\mathcal{BV}^\mathsection\ar[d] \\
\Delta_{/S}^\text{op}\times\Delta^{1\mathsection}\ar@{-->}[r]_-{\mathcal{F}} &\smallint\mathcal{B}(\mathcal{V};\mathcal{M})^\mathsection.
}$$
\end{remark}

\noindent Unlike in the operadic model, we can unpack this last diagram further. Recall (following Remark \ref{RmkHA413}) that $\mathcal{B}(\mathcal{V};\mathcal{M}):\Delta^\text{op}\times\Delta^1\to\text{Cat}$ can be identified with a morphism $\mathcal{BV}\ltimes\mathcal{M}\to\mathcal{BV}$ of simplicial $\infty$-categories. Applying the Grothendieck construction, we have a diagram $$\xymatrix{
\smallint\mathcal{BV}\ltimes\mathcal{M}\ar[rr]^-q\ar[rd]_-{p_0} &&\smallint\mathcal{BV}\ar[ld]^-{p_1} \\
&\Delta^\text{op}, &
}$$ where $q$ sends $p_0$-cocartesian morphisms to $p_1$-cocartesian morphisms. Also recall the marking $\smallint\mathcal{BV}\ltimes\mathcal{M}^\ddagger$ by totally inert morphisms (Warning \ref{TotInert} and Remark \ref{RmkTotInert}).

\begin{definition}
Say that a morphism of $\Delta_{/S}^\text{op}$ is \emph{totally inert} if the underlying morphism of $\Delta^\text{op}$ is totally inert, and write $\Delta_{/S}^{\text{op}\ddagger}$ for this marking. Equivalently, a morphism is totally inert if its image under $\Delta_{/S}^\text{op}\times\{0\}\subseteq\Delta_{/S}^\text{op}\times\Delta^1$ is inert.
\end{definition}

\begin{proposition}\label{PropPShModel}
For any pair $(\mathcal{V};\mathcal{M})$ and set $S$, $\mathbb{A}_\infty\text{PSh}^{\mathcal{V};\mathcal{M}}_S$ is equivalent to the full subcategory of $\text{Fun}^\dag(\Delta_{/S}^{\text{op}\ddagger},\smallint\mathcal{BV}\ltimes\mathcal{M}^\ddagger)$ spanned by functors such that $\Delta_{/S}^\text{op}\to\smallint\mathcal{BV}\ltimes\mathcal{M}\to\smallint\mathcal{BV}$ is an $\mathbb{A}_\infty$-algebra.

If $\mathcal{C}\in\mathbb{A}_\infty\text{Cat}^\mathcal{V}_S$, then $$\mathbb{A}_\infty\text{PSh}^\mathcal{V}(\mathcal{C};\mathcal{M})\cong\text{Fun}^\dag_{/\smallint\mathcal{BV}^\mathsection}(\Delta_{/S}^{\text{op}\ddagger},\smallint\mathcal{BV}\ltimes\mathcal{M}^\ddagger).$$
\end{proposition}

\begin{proof}
The proof follows HA 4.2.2.19 (stated earlier as Proposition \ref{PropA1}), which is the case $|S|=1$.

If $\smallint\mathcal{B}(\mathcal{V};\mathcal{M})_i$ is the fiber of $\smallint\mathcal{B}(\mathcal{V};\mathcal{M})\to\Delta^\text{op}\times\Delta^1$ over $\Delta^\text{op}\times\{i\}$, then:
\begin{itemize}
\item $\smallint\mathcal{B}(\mathcal{V};\mathcal{M})_0^\mathsection=\smallint\mathcal{BV}\ltimes\mathcal{M}^\ddagger$;
\item $\smallint\mathcal{B}(\mathcal{V};\mathcal{M})_1^\mathsection=\smallint\mathcal{BV}^\mathsection$;
\item If $X_i\in\smallint\mathcal{B}(\mathcal{V};\mathcal{M})_i$ are two objects both lying over $Y\in\Delta^\text{op}_{/S}$, and $f:X_0\to X_1$ is a morphism lying over $(Y,0)\to(Y,1)\in\Delta^\text{op}_S\times\Delta^1$ (the identity map on $Y$), then $f$ is inert if and only if the induced functor $q(X_0)\to X_1$ is an equivalence.
\end{itemize}
\noindent Therefore, giving a lift $\mathcal{F}$ as in the square $$\xymatrix{
\Delta_{/S}^{\text{op}\mathsection}\ar[r]^-{\mathcal{C}}\ar[d] &\smallint\mathcal{BV}^\mathsection\ar[d] \\
\Delta_{/S}^\text{op}\times\Delta^{1\mathsection}\ar@{-->}[r]_-{\mathcal{F}} &\smallint\mathcal{B}(\mathcal{V};\mathcal{M})^\mathsection.
}$$ is the same as giving a marked functor $\mathcal{F}_0:\Delta_{/S}^{\text{op}\ddagger}\to\smallint\mathcal{BV}\ltimes\mathcal{M}^\ddagger$, as well as a natural transformation $\mathcal{F}_0\to\mathcal{C}$ sends each $Y\in\Delta_{/S}^\text{op}$ to an inert morphism of $\mathcal{B}(\mathcal{V};\mathcal{M})$ lying over $(Y,0)\to(Y,1)$; in other words, $\mathcal{C}$ should factor $\mathcal{C}=q\mathcal{F}_0$. The proposition follows from this description.
\end{proof}

\subsection{Limits and colimits of presheaves}\label{S43}
\noindent We are ready to prove that $\theta:\text{PSh}^{\mathcal{V};\mathcal{M}}_S\to\text{Cat}^\mathcal{V}_S$ is a cartesian fibration, and a presentable fibration if $(\mathcal{V};\mathcal{M})$ is a presentable pair. We will begin by proving that $\theta$ is a cartesian fibration.

\begin{proposition}\label{PropCartFib}
If $\mathcal{V}$ is a monoidal $\infty$-category, $\mathcal{M}$ is a left $\mathcal{V}$-module, and $S$ is a set, then the forgetful functor $\theta:\text{PSh}^{\mathcal{V};\mathcal{M}}_S\to\text{Cat}^\mathcal{V}_S$ is a cartesian fibration, and a morphism $f$ in $\text{PSh}^{\mathcal{V};\mathcal{M}}_S$ is $\theta$-cartesian if and only if its image under $\text{ev}_X:\text{PSh}^{\mathcal{V};\mathcal{M}}_S\to\mathcal{M}$ is an equivalence for each $X\in S$.
\end{proposition}

\begin{remark}\label{RmkFuncEv}
Since $\theta$ is a cartesian fibration, a map $F:\mathcal{C}\to\mathcal{D}$ in $\text{Cat}^\mathcal{V}_S$ induces a functor $F^\ast:\text{PSh}^\mathcal{V}(\mathcal{D};\mathcal{M})\to\text{PSh}^\mathcal{V}(\mathcal{C};\mathcal{M})$. By the identification of $\theta$-cartesian morphisms, $F^\ast(\mathcal{F})$ can be evaluated at objects by the formula $F^\ast(\mathcal{F})(X)=\mathcal{F}(X)$. In other words, each triangle commutes: $$\xymatrix{
\text{PSh}^\mathcal{V}(\mathcal{D};\mathcal{M})\ar[rr]^-{F^\ast}\ar[rd]_-{\text{ev}_X} &&\text{PSh}^\mathcal{V}(\mathcal{C};\mathcal{M})\ar[ld]^-{\text{ev}_X} \\
&\mathcal{M}. &
}$$
\end{remark}

\begin{lemma}\label{LemCartFib}
Suppose the pair $(\mathcal{V};\mathcal{M})$ is modeled by a cocartesian fibration of quasicategories $\smallint(\mathcal{V};\mathcal{M})\to\text{LM}$, and $K$ is a simplicial set such that $\mathcal{M}$ admits $K$-indexed limits. Then:
\begin{enumerate}
\item For every commutative square $$\xymatrix{
K\ar[r]\ar@{_{(}->}[d] &\mathbb{A}_\infty\text{PSh}^{\mathcal{V};\mathcal{M}}_S\ar[d]^\theta \\
K^\triangleleft\ar[r]\ar@{-->}[ru] &\mathbb{A}_\infty\text{Cat}^\mathcal{V}_S,
}$$ there exists a dotted arrow as indicated, which is a $\theta$-limit diagram.
\item An arbitrary functor $K^\triangleleft\to\mathbb{A}_\infty\text{PSh}^{\mathcal{V};\mathcal{M}}_S$ is a $\theta$-limit diagram if and only if the composite $K^\triangleleft\to\mathbb{A}_\infty\text{PSh}^{\mathcal{V};\mathcal{M}}_S\xrightarrow{\text{ev}_X}\mathcal{M}$ is a limit diagram for each $X\in S$.
\end{enumerate}
\end{lemma}

\begin{proof}
We follow HA 4.2.3.1, which is the $|S|=1$ case. We set notation as in the triangle preceding Remark \ref{RmkPrec}: $$\xymatrix{
\smallint\mathcal{BV}\ltimes\mathcal{M}\ar[rr]^-q\ar[rd]_-{p_0} &&\smallint\mathcal{BV}\ar[ld]^-{p_1} \\
&\Delta^\text{op} &
}$$ The fibers over $[n]\in\Delta^\text{op}$ are $\smallint\mathcal{BV}\ltimes\mathcal{M}_n=\mathcal{V}^{\times n}\times\mathcal{M}$ and $\smallint\mathcal{BV}_n=\mathcal{V}^{\times n}$. Lurie concludes in HA 4.2.3.1 $(1_n^{\prime\prime}-2_n^{\prime\prime})$:
\begin{enumerate}[1'.]
\item For every commutative square $$\xymatrix{
K\ar[r]\ar@{_{(}->}[d] &\smallint\mathcal{BV}\ltimes\mathcal{M}_n\ar[d]^{q_n} \\
K^\triangleleft\ar[r]\ar@{-->}[ru] &\smallint\mathcal{BV}_n,
}$$ there exists a dotted arrow as indicated, which is a $q$-limit diagram.
\item An arbitrary map $K^\triangleleft\to\smallint\mathcal{BV}\ltimes\mathcal{M}_n$ is a $q$-limit diagram if and only if the projection $K^\triangleleft\to\smallint\mathcal{BV}\ltimes\mathcal{M}_n=\mathcal{V}^{\times n}\times\mathcal{M}\to\mathcal{M}$ is a limit diagram.
\end{enumerate}
\noindent Applying HA 3.2.2.9,
\begin{enumerate}[1''.]
\item For every commutative square $$\xymatrix{
K\ar[r]\ar@{_{(}->}[d] &\text{Fun}(\Delta_{/S}^\text{op},\smallint\mathcal{BV}\ltimes\mathcal{M})\ar[d]^{\theta} \\
K^\triangleleft\ar[r]\ar@{-->}[ru] &\text{Fun}(\Delta_{/S}^\text{op},\smallint\mathcal{BV}),
}$$ there exists a dotted arrow as indicated, which is a $\theta$-limit diagram.
\item An arbitrary map $K^\triangleleft\to\text{Fun}(\Delta_{/S}^\text{op},\smallint\mathcal{BV}\ltimes\mathcal{M})$ is a $\theta$-limit diagram if and only if for each $\{X_0<\cdots<X_n\}\in\Delta_{/S}$, the projection $K^\triangleleft\xrightarrow{\text{ev}_X}\smallint\mathcal{BV}\ltimes\mathcal{M}_n\to\mathcal{M}$ is a limit diagram.
\end{enumerate}

\noindent In 1''-2'', we can just as well restrict to $\mathbb{A}_\infty\text{Cat}^\mathcal{V}_S$, which is a full subcategory of $\text{Fun}(\Delta_{/S}^\text{op},\smallint\mathcal{BV})$ by Remark \ref{RmkFiber2} (that is, the full subcategory of functors with send inert morphisms to inert morphisms and are compatible with the functors down to $\Delta^\text{op}$).

Let $\text{pre}\mathbb{A}_\infty\text{PSh}^{\mathcal{V};\mathcal{M}}_S$ be the full subcategory of $\text{Fun}(\Delta_{/S}^\text{op},\smallint\mathcal{BV}\ltimes\mathcal{M})$ spanned by those functors $F$ for which $\theta(F)$ is an $\mathbb{A}_\infty$-algebra. Then the restriction $\theta:\text{pre}\mathbb{A}_\infty\text{PSh}^{\mathcal{V};\mathcal{M}}_S\to\mathbb{A}_\infty\text{Cat}^\mathcal{V}_S$ enjoys the same properties 1''-2''.

By Proposition \ref{PropPShModel}, $\mathbb{A}_\infty\text{PSh}^{\mathcal{V};\mathcal{M}}$ is the full subcategory of $\text{pre}\mathbb{A}_\infty\text{PSh}^{\mathcal{V};\mathcal{M}}$ spanned by those functors $\Delta_{/S}^\text{op}\to\smallint\mathcal{BV}\ltimes\mathcal{M}$ which send totally inert morphisms to $p_0$-cocartesian morphisms.

Now we will prove part (1) of the lemma. It suffices to show that if $\bar{g}$ is a $\theta$-limit diagram $K^\triangleleft\to\text{pre}\mathbb{A}_\infty\text{PSh}^{\mathcal{V};\mathcal{M}}_S$ and the restriction $g$ to $K$ factors through $\mathbb{A}_\infty\text{PSh}^{\mathcal{V};\mathcal{M}}_S$, then $\bar{g}$ factors through $\mathbb{A}_\infty\text{PSh}^{\mathcal{V};\mathcal{M}}_S$. Any totally inert morphism $f:\{X_0<\cdots<X_n\}\to\{Y_0<\cdots<Y_m\}$ in $\Delta_{/S}^\text{op}$ induces a natural transformation $\bar{g}_X\to\bar{g}_Y$ of functors $K^\triangleleft\to\smallint\mathcal{BV}\ltimes\mathcal{M}$, and we want to show that each object of $K^\triangleleft$ is sent to a $p_0$-cocartesian morphism in $\mathcal{BV}\ltimes\mathcal{M}$.

Since $f$ is totally inert, $f(m)=n$, and we have a commutative triangle, where the downward maps are just projection $$\xymatrix{
\smallint\mathcal{BV}\ltimes\mathcal{M}_n\ar[rr]\ar[rd]_-\alpha &&\smallint\mathcal{BV}\ltimes\mathcal{M}_m\ar[ld]^-\beta \\
&\mathcal{M}, &
}$$ and it suffices to show that $\bar{t}:\alpha\bar{g}_X\to\beta\bar{g}_Y$ is an equivalence. By hypothesis, the restriction $t:\alpha g_X\to\beta g_Y$ is an equivalence, and then by (2''), $\bar{t}$ is an induced natural transformation between limit diagrams. Therefore, $\bar{t}$ is also an equivalence, completing the proof of (1).

As for (2), we need to prove that the following are equivalent for a functor $\bar{g}:K^\triangleleft\to\mathbb{A}_\infty\text{PSh}^{\mathcal{V};\mathcal{M}}_S$:
\begin{itemize}
\item[2''.] Each $K^\triangleleft\xrightarrow{\bar{g}}\mathbb{A}_\infty\text{PSh}^{\mathcal{V};\mathcal{M}}_S\xrightarrow{\text{ev}_X}\mathcal{M}$ is a limit diagram, where $\text{ev}_X$ is evaluation of $\Delta_{/S}^\text{op}\to\smallint\mathcal{BV}\ltimes\mathcal{M}$ at any $X=\{X_0<\cdots<X_n\}\in\Delta_{/S}^\text{op}$;
\item[2.] Each $\text{ev}_X\bar{g}$ is a limit diagram if $n=0$; that is, $X=\{X_0\}\in\Delta_{/S}^\text{op}$.
\end{itemize}

\noindent Obviously 2'' implies 2. Conversely, assume 2. If $X=\{X_0<\cdots<X_n\}$, then there is a totally inert morphism $X\to\{X_n\}$ which induces a commutative square $$\xymatrix{
\mathbb{A}_\infty\text{PSh}^{\mathcal{V};\mathcal{M}}(\mathcal{C})\ar[r]^-{\text{ev}_X}\ar[d]_-{\text{ev}_{\{X_n\}}} &\smallint\mathcal{BV}\ltimes\mathcal{M}_n\ar[d] \\
\smallint\mathcal{BV}\ltimes\mathcal{M}_0\ar[r] &\mathcal{M}.
}$$ Then $\text{ev}_X\bar{g}=\text{ev}_{\{X_n\}}\bar{g}$ is a limit diagram by 2, completing the proof.
\end{proof}

\begin{proof}[Proof of Proposition \ref{PropCartFib}]
Suppose the pair $(\mathcal{V};\mathcal{M})$ is modeled by a cocartesian fibration of quasicategories $\smallint(\mathcal{V};\mathcal{M})\to\text{LM}$. We will prove that $\theta:\mathbb{A}_\infty\text{PSh}^{\mathcal{V};\mathcal{M}}_S\to\mathbb{A}_\infty\text{Cat}^\mathcal{V}_S$ is a cocartesian fibration of quasicategories. We already know $\theta$ is a categorical fibration (Remark \ref{RmkFiber2}), hence also an inner fibration.

Recall that a functor $\Delta^1\to\mathbb{A}_\infty\text{PSh}^{\mathcal{V};\mathcal{M}}_S$ is a $\theta$-limit diagram if and only if it is a $\theta$-cartesian edge (HTT 4.3.1.4). Hence when $K=\ast$, part (1) of the lemma asserts $\theta$ is a cartesian fibration, and (2) asserts that $f$ is $\theta$-cartesian if and only if $\text{ev}_X(f)$ is an equivalence for each $X\in S$.
\end{proof}

\noindent For the rest of this section, we will assume that $(\mathcal{V};\mathcal{M})$ is a presentable pair.

\begin{theorem}\label{PShPrL2}
If $(\mathcal{V};\mathcal{M})$ is a presentable pair and $\mathcal{C}$ is $\mathcal{V}$-enriched:
\begin{enumerate}
\item $\text{PSh}^\mathcal{V}(\mathcal{C};\mathcal{M})$ is presentable;
\item for each $X\in\mathcal{C}$, $\text{ev}_X:\text{PSh}^\mathcal{V}(\mathcal{C};\mathcal{M})\to\mathcal{M}$ preserves limits and colimits;
\item If $F:\mathcal{C}\to\mathcal{D}$ is a map in $\text{Cat}^\mathcal{V}_S$, then $F^\ast:\text{PSh}^\mathcal{V}_S(\mathcal{D};\mathcal{M})\to\text{PSh}^\mathcal{V}(\mathcal{C};\mathcal{M})$ preserves limits and colimits.
\item if $\mathcal{N}$ is presentable, a functor $F:\mathcal{N}\to\text{PSh}^\mathcal{V}(\mathcal{C};\mathcal{M})$ preserves colimits (respectively limits) if and only if the composite $\text{ev}_XF:\mathcal{N}\to\mathcal{M}$ preserves colimits (limits) for each $X\in\mathcal{C}$.
\end{enumerate}
\end{theorem}

\noindent In particular, (1) is Theorem \ref{PShPrL} of the introduction.

\begin{proof}
We follow the proof of HA 4.2.3.4. Pick a cocartesian fibration of quasicategories $\smallint(\mathcal{V};\mathcal{M})\to\text{LM}$ which realizes $\mathcal{M}$ as a left $\mathcal{V}$-module. Define $\mathcal{X}$ to be the pullback $$\xymatrix{
\mathcal{X}\ar[r]\ar[d]_-{q^\prime} &\smallint\mathcal{BV}\ltimes\mathcal{M}\ar[d]^-q \\
\Delta_{/S}^\text{op}\ar[r]_-{\mathcal{C}} &\smallint\mathcal{BV},
}$$ and say that a morphism of $\mathcal{X}$ is \emph{totally inert} if its images in $\Delta_{/S}^\text{op}$ and $\smallint\mathcal{BV}\ltimes\mathcal{M}$ are each totally inert. Recall (Remark \ref{RmkLCF}) that $q$ is a categorical fibration and a locally cocartesian fibration. Therefore, $q^\prime$ is as well.

By Proposition \ref{PropPShModel}, $\text{PSh}^\mathcal{V}(\mathcal{C};\mathcal{M})$ is equivalent to the quasicategory of sections of $q^\prime$ which send totally inert morphisms to totally inert morphisms. Applying HTT 5.4.7.11 using the subcategory $\text{Pr}^L\subseteq\widehat{\text{Cat}}$, we conclude:
\begin{itemize}
\item $\text{PSh}^\mathcal{V}(\mathcal{C};\mathcal{M})$ is presentable;
\item $F:\mathcal{N}\to\text{PSh}^\mathcal{V}(\mathcal{C};\mathcal{M})$ preserves colimits if and only if the composite $\text{ev}_XF:\mathcal{N}\to\smallint\mathcal{BV}\ltimes\mathcal{M}_n$ does for each $X=\{X_0<\cdots<X_n\}\in S$.
\end{itemize}
\noindent At the end of the proof of Lemma \ref{LemCartFib} (where we proved conditions 2 and 2'' are equivalent), we proved that this second condition is equivalent to the colimit formulation of (4).

The limit formulation of (4) is Lemma \ref{LemCartFib} when $K^\triangleleft\to\mathbb{A}_\infty\text{Cat}^\mathcal{V}_S$ is the constant functor with value $\mathcal{C}$.

So we have (1) and (4). Applying (4) to the identity functor establishes (2). Applying (4) to the functor $\text{PSh}^\mathcal{V}_S(\mathcal{D};\mathcal{M})\to\text{PSh}^\mathcal{V}(\mathcal{C};\mathcal{M})$ along with Remark \ref{RmkFuncEv} establishes (3).
\end{proof}

\noindent We deduce two important corollaries. The first is a formulation of Corollary \ref{CorPFib1} from the introduction.

\begin{corollary}\label{CorPFib1b}
If $(\mathcal{V};\mathcal{M})$ is a presentable pair and $\mathcal{C}$ is $\mathcal{V}$-enriched:
\begin{enumerate}
\item A functor $p:K^\triangleleft\to\text{PSh}^\mathcal{V}(\mathcal{C};\mathcal{M})$ is a limit diagram if and only if $\text{ev}_Xp:K^\triangleleft\to\mathcal{M}$ is a limit diagram for all $X\in\mathcal{C}$;
\item A functor $p:K^\triangleright\to\text{PSh}^\mathcal{V}(\mathcal{C};\mathcal{M})$ is a colimit diagram if and only if $\text{ev}_Xp:K^\triangleright\to\mathcal{M}$ is a colimit diagram for all $X\in\mathcal{C}$;
\end{enumerate}
\end{corollary}

\begin{corollary}\label{Cor1S4}
If $(\mathcal{V},\mathcal{M})$ is a presentable pair, $\theta:\text{PSh}^{\mathcal{V};\mathcal{M}}_S\to\text{Cat}^\mathcal{V}_S$ is a presentable fibration.
\end{corollary}

\begin{proof}
By Proposition \ref{PropCartFib}, $\theta$ is a cartesian fibration. Now assume $(\mathcal{V};\mathcal{M})$ is a presentable pair. Let $F:(\text{Cat}^\mathcal{V}_S)^\text{op}\to\widehat{\text{Cat}}$ be the associated functor, so that $F(\mathcal{C})=\text{PSh}^\mathcal{V}(\mathcal{C};\mathcal{M})$. By Theorem \ref{PShPrL2}, $F(\mathcal{C})$ is presentable for each $\mathcal{C}$, and $F(\mathcal{D})\to F(\mathcal{C})$ is a right adjoint functor (as it preserves limits and colimits) for each enriched functor $\mathcal{C}\to\mathcal{D}$. Therefore, $F$ factors through $\text{Pr}^R\subseteq\widehat{\text{Cat}}$, so $\theta$ is a presentable fibration.
\end{proof}

\begin{remark}\label{RmkFuncRep}
If $F:\mathcal{C}\to\mathcal{D}$ is a map in $\text{Cat}^\mathcal{V}_S$, then there is always a functor $F^\ast:\text{PSh}^\mathcal{V}(\mathcal{D};\mathcal{M})\to\text{PSh}^\mathcal{V}(\mathcal{C};\mathcal{M})$. Corollary \ref{Cor1S4} asserts that $F^\ast$ has a left adjoint $F_\ast$ if $(\mathcal{V};\mathcal{M})$ is a presentable pair. Recall that $F^\ast$ fits in a commutative triangle for each $X\in S$ (Remark \ref{RmkFuncEv}), $$\xymatrix{
\text{PSh}^\mathcal{V}(\mathcal{D};\mathcal{M})\ar[rr]^-{F^\ast}\ar[rd]_-{\text{ev}_X} &&\text{PSh}^\mathcal{V}(\mathcal{C};\mathcal{M})\ar[ld]^-{\text{ev}_{X}} \\
&\mathcal{M}. &
}$$ Taking left adjoints, we have a commutative triangle $$\xymatrix{
&\mathcal{M}\ar[ld]_-{\text{rep}_X\otimes -}\ar[rd]^-{\text{rep}_X\otimes -} &\\
\text{PSh}^\mathcal{V}(\mathcal{C};\mathcal{M})\ar[rr]_-{F_\ast} &&\text{PSh}^\mathcal{V}(\mathcal{D};\mathcal{M}),
}$$ so $F_\ast(\text{rep}_X(\mathcal{C})\otimes M)\cong\text{rep}_X(\mathcal{D})\otimes M$.
\end{remark}

\noindent Recall that a functor $F:\mathcal{X}\to\mathcal{Y}$ is called \emph{monadic} if either of the following equivalent conditions hold (HA 4.7.3.5):
\begin{itemize}
\item $F$ has a left adjoint, is conservative, and preserves certain colimits\footnote{We won't care which colimits are preserved, because all monadic functors we study will preserve \emph{all} small colimits.};
\item There is a monoidal $\infty$-category $\mathcal{E}$, an algebra $E$ in $\mathcal{E}$, and a left $\mathcal{E}$-module structure on $\mathcal{Y}$, such that $F$ is equivalent to the forgetful functor $\text{LMod}_E(\mathcal{Y})\to\mathcal{Y}$.
\end{itemize}

\noindent $\mathcal{E}$ can be taken to be the endomorphism $\infty$-category $\text{Fun}(\mathcal{Y},\mathcal{Y})$, and $E$ the monad associated to $F$ (the composite $F\circ L$, where $L$ is the left adjoint).

\begin{corollary}
If $(\mathcal{V};\mathcal{M})$ is a presentable pair and $\mathcal{C}$ is $\mathcal{V}$-enriched with set $S$ of objects, let $\text{ev}:\text{PSh}^\mathcal{V}(\mathcal{C};\mathcal{M})\to\mathcal{M}^{\times S}$ be the product of all the functors $\text{ev}_X$, as $X\in S$ varies. The functor $\text{ev}$ is monadic.
\end{corollary}

\begin{proof}
By Theorem \ref{PShPrL2}(2), ev is a functor of presentable $\infty$-categories which preserves small limits and colimits. Therefore, it has a left adjoint. By Corollary \ref{CorPFib1b} in the case $K=\ast$, ev is conservative, hence monadic.
\end{proof}

\begin{corollary}\label{CorFreeCo}
If $(\mathcal{V};\mathcal{M})$ is a presentable pair and $\mathcal{C}$ is $\mathcal{V}$-enriched, then $\text{PSh}^\mathcal{V}(\mathcal{C};\mathcal{M})$ is generated under colimits by the free presheaves $\text{rep}_X\otimes M$, where $X\in\mathcal{C}$ and $M\in\mathcal{M}$.
\end{corollary}

\begin{proof}
Since $\text{ev}:\text{PSh}^\mathcal{V}(\mathcal{C};\mathcal{M})\to\mathcal{M}^S$ is monadic, we can without loss of generality replace ev with a functor $\text{LMod}_E(\mathcal{M}^S)\to\mathcal{M}^S$. Every left $E$-module $M$ is a colimit of free left $E$-modules by the bar construction for $E\otimes_E M$ (HA 4.4.2). Therefore, $\text{PSh}^\mathcal{V}(\mathcal{C};\mathcal{M})$ is generated under colimits by the image of the free functor $\mathcal{M}^S\to\text{PSh}^\mathcal{V}(\mathcal{C};\mathcal{M})$. By Corollary \ref{CorLAdj}, this free functor sends $\{M_X\}_{X\in S}$ to the coproduct $\coprod_{X\in S}\text{rep}_X\otimes M_X$, so $\text{PSh}^\mathcal{V}(\mathcal{C};\mathcal{M})$ is generated under colimits by free presheaves.
\end{proof}

\begin{remark}
It should be possible to relax the presentability conditions on $\mathcal{V}$ and $\mathcal{M}$ as follows:
\begin{itemize}
\item By Lemma \ref{LemCartFib}, $\text{PSh}^\mathcal{V}(\mathcal{C};\mathcal{M})$ admits limits indexed by $K$ as long as $\mathcal{M}$ admits limits indexed by $K$ (with no condition on $\mathcal{V}$).
\item $\text{PSh}^\mathcal{V}(\mathcal{C};\mathcal{M})$ admits colimits indexed by $K$ if $\mathcal{M}$ admits colimits indexed by $K$, which are compatible with the left $\mathcal{V}$-module structure in a certain sense. If details are needed, consult HA 4.2.3.4-5.
\end{itemize}
In each case, limits and colimits are detected by $\text{ev}_X:\text{PSh}^\mathcal{V}(\mathcal{C};\mathcal{M})\to\mathcal{M}$.
\end{remark}

\section{Tensor products of presheaves}\label{S5}
\noindent Suppose that $\mathcal{F}\in\text{PSh}^\mathcal{V}(\mathcal{C})$ is a presheaf, $\mathcal{M}$ is a left $\mathcal{V}$-module, and $M\in\mathcal{M}$. Then there is a tensor presheaf $\mathcal{F}\otimes M\in\text{PSh}^\mathcal{V}(\mathcal{C};\mathcal{M})$ given informally by the formula $(\mathcal{F}\otimes M)(X)=\mathcal{F}(X)\otimes M$.

If $\mathcal{M}=\mathcal{V}$, then this tensor product describes a right action of $\mathcal{V}$ on $\text{PSh}^\mathcal{V}(\mathcal{C})$. We begin in Section \ref{S51} by reviewing right modules. In the same way that left modules were generalized to enriched presheaves, right modules describe \emph{enriched copresheaves}. We introduce enriched copresheaves, which play an important role in Section \ref{S6}.

In Section \ref{S52}, we construct the right $\mathcal{V}$-action on presheaves (Definition \ref{DefRMPSh}, Theorem \ref{PropRMPSh2}). In Section \ref{S53}, we prove that it really is given by the formula above (Proposition \ref{PropEv1}).

In the event that $\mathcal{V}$ is presentable and closed monoidal, $\text{PSh}^\mathcal{V}(\mathcal{C})$ is even a presentable right $\mathcal{V}$-module (Theorem \ref{PropPrRight}) and assembles into a functor $$\text{PSh}^\mathcal{V}(-):\text{Cat}^\mathcal{V}_S\to\text{RMod}_\mathcal{V}(\text{Pr}^L)$$ (Corollary \ref{CorModFunc}). In other words, if $F:\mathcal{C}\to\mathcal{D}$ is an enriched functor, then $F_\ast:\text{PSh}^\mathcal{V}(\mathcal{C})\to\text{PSh}^\mathcal{V}(\mathcal{D})$ is compatible with the right $\mathcal{V}$-module structure, or $F_\ast(\mathcal{F}\otimes A)\cong F_\ast(\mathcal{F})\otimes A$.

\begin{remark}
Of course, we expect that $\text{PSh}^\mathcal{V}(-)$ extends to a functor $\text{Cat}^\mathcal{V}\to\text{RMod}_\mathcal{V}(\text{Pr}^L)$, where $\text{Cat}^\mathcal{V}$ is the $\infty$-category of \emph{all} $\mathcal{V}$-enriched categories (that is, without a fixed set of objects). As always, we are delaying such results until a future paper in this series.
\end{remark}

\noindent Our discussion so far all assumed $\mathcal{M}=\mathcal{V}$. In general, the construction $\mathcal{F}\otimes M$ describes a functor $\text{PSh}^\mathcal{V}(\mathcal{C})\otimes_\mathcal{V}\mathcal{M}\to\text{PSh}^\mathcal{V}(\mathcal{C};\mathcal{M})$. In Section \ref{S53}, we prove the main theorem of this section:

\begin{reptheorem}{ThmS5}
If $\mathcal{V}$ is presentable and closed monoidal, $\mathcal{C}$ is a $\mathcal{V}$-enriched category, and $\mathcal{M}$ is a presentable left $\mathcal{V}$-module, then the functor $$\Psi:\text{PSh}^\mathcal{V}(\mathcal{C})\otimes_\mathcal{V}\mathcal{M}\to\text{PSh}^\mathcal{V}(\mathcal{C};\mathcal{M})$$ is an equivalence, where $\otimes_\mathcal{V}$ is the relative tensor product in $\text{Pr}^L$.
\end{reptheorem}

\noindent The construction of the $\mathcal{V}$-module structure and this functor $\Psi$ are quite involved. However, most of our proofs rely only on the interaction of tensor presheaves with evaluation and representable presheaves. That is:
\begin{itemize}
\item $(\mathcal{F}\otimes M)(X)\cong\mathcal{F}(X)\otimes M$
\item $\text{rep}_X\otimes M$ is the free presheaf generated by $M$ at $X$.
\end{itemize}
\noindent In Section \ref{S33}, we used the notation $\text{rep}_X\otimes M$ for the free presheaf generated by $M$ at $X$. Now we see that this notation is justified: the free presheaf $\text{rep}_X\otimes M$ really is equivalent to the tensor presheaf $\text{rep}_X\otimes M$.

We prove these facts in Proposition \ref{PropEv1} in the case $\mathcal{M}=\mathcal{V}$, and Lemma \ref{LemComp} in general.

\subsection{Right modules}\label{S51}
\noindent First we introduce $\infty$-operads $\text{RM}_S$ for each set $S$, whose algebras are \emph{copresheaves}. In the case $|S|=1$, RM is the right module $\infty$-operad, and an algebra over RM is a pair $(A,M)$, where $A$ is an algebra and $M$ a right module.

The $\infty$-operad $\text{RM}_S$ is the \emph{reverse} of $\text{LM}_S$, in the sense of HA 4.1.1.7; that is, the underlying $\infty$-category is equivalent, but the operad structure is reversed. We begin by reviewing reverse $\infty$-operads.

If $\Gamma\in\text{Assoc}_S$ is a graph, let $\text{rev}(\Gamma)$ denote $\Gamma$ with the directions of all arrows reversed. For example, $\text{rev}(X,Y)=(Y,X)$. This construction describes a functor of categories $\text{rev}:\text{Assoc}_S\to\text{Assoc}_S$, which is an isomorphism. Recall:

\begin{definition}[HA 4.1.1.7]
If $\mathcal{O}$ is an $\infty$-operad, then the composite $\mathcal{O}\to\text{Assoc}\xrightarrow{\text{rev}}\text{Assoc}$ endows the $\infty$-category $\mathcal{O}$ with a distinct $\infty$-operad structure. We write $\mathcal{O}^\text{rev}$ for this $\infty$-operad, and call it the \emph{reverse $\infty$-operad} of $\mathcal{O}$.

If $\mathcal{V}:\text{Assoc}\to\text{Cat}$ is a monoidal $\infty$-category, then the composite $\mathcal{V}\circ\text{rev}$ is another monoidal $\infty$-category, which we denote $\mathcal{V}^\text{rev}$.
\end{definition}

\noindent Note that $\mathcal{V}$ and $\mathcal{V}^\text{rev}$ have the same underlying $\infty$-categories, but the monoidal operations are reversed; that is, $X\otimes^\text{rev}Y=Y\otimes X$.

If $\mathcal{O}$ is an $\infty$-operad and $\mathcal{V}$ is a monoidal $\infty$-category, then the reversal isomorphism induces an equivalence $\text{Alg}_\mathcal{O}(\mathcal{V})\cong\text{Alg}_{\mathcal{O}^\text{rev}}(\mathcal{V}^\text{rev})$.

\begin{example}\label{ExOp}
The isomorphism $\text{rev}:\text{Assoc}_S\to\text{Assoc}_S$ exhibits $\text{Assoc}_S$ as equivalent to its own reversal. Therefore, $\text{Cat}^\mathcal{V}_S\cong\text{Cat}^{\mathcal{V}^\text{rev}}_S$. If $\mathcal{C}$ is a $\mathcal{V}$-enriched category, then the corresponding $\mathcal{V}^\text{rev}$-enriched category is $\mathcal{C}^\text{op}$, which has mapping spaces $\mathcal{C}^\text{op}(X,Y)=\mathcal{C}(Y,X)$.
\end{example}

\begin{definition}
For any set $S$, $\text{RM}_S=\text{LM}_S^\text{rev}$.

When $|S|=1$, we also write $\text{RM}=\text{RM}_S$, which is the \emph{right module operad} of HA 4.2.1.36.
\end{definition}

\begin{remark}
We say that a graph $\Gamma\in\text{Assoc}_{S_{+}}$ is \emph{right modular} (compare \emph{left modular}, Definition \ref{DefLMS}) if for all edges $e\in\Gamma$, $t(e)\neq\ast$. Then $\text{RM}_S$ is equivalent to the full subcategory of $\text{Assoc}_{S_{+}}$ spanned by right modular graphs.

This is because $\text{rev}:\text{Assoc}_{S_{+}}\to\text{Assoc}_{S_{+}}$ restricts to an isomorphism between $\text{LM}_S$ and the full subcategory of right modular graphs.
\end{remark}

\noindent In the case $|S|=1$, we think of an RM-algebra as a pair $(A,M)$, where $A$ is an algebra and $M$ is a right $A$-module.

In general, we think of an $\text{RM}_S$-algebra as a pair $(\mathcal{C},\mathcal{F})$, where $\mathcal{C}$ is a $\mathcal{V}$-enriched category and $\mathcal{F}$ is an \emph{enriched copresheaf} on $\mathcal{C}$:

\begin{definition}\label{DefCopsh}
Suppose that $\mathcal{V}$ is a monoidal $\infty$-category and $\mathcal{N}$ a \emph{right} $\mathcal{V}$-module category, so that the pair $(\mathcal{V};\mathcal{N})$ is an RM-monoidal $\infty$-category.

If $\mathcal{C}$ is a $\mathcal{V}$-enriched category with set $S$ of objects, an enriched \emph{copresheaf on $\mathcal{C}$ with values in $\mathcal{M}$} is a lift of the $\text{Assoc}_S$-algebra $\mathcal{C}$ to an $\text{RM}_S$-algebra in $(\mathcal{V};\mathcal{N})$. We write $$\text{coPSh}^{\mathcal{V};\mathcal{N}}_S=\text{Alg}_{\text{RM}_S/\text{RM}}(\mathcal{V};\mathcal{N}),$$ $$\text{coPSh}^\mathcal{V}(\mathcal{C},\mathcal{N})=\text{coPSh}^{\mathcal{V};\mathcal{N}}_S\times_{\text{Cat}^\mathcal{V}_S}\{\mathcal{C}\}.$$
\end{definition}

\noindent We think of a copresheaf $\mathcal{F}\in\text{coPSh}^\mathcal{V}(\mathcal{C},\mathcal{N})$ informally as a $\mathcal{V}$-enriched functor $\mathcal{C}\to\mathcal{N}$.

\begin{remark}\label{RmkOp}
Because $\text{LM}_S$ and $\text{RM}_S$ are reverse $\infty$-operads,
\begin{itemize}
\item A right $\mathcal{V}$-action on $\mathcal{N}$ canonically induces a left $\mathcal{V}^\text{rev}$-action on $\mathcal{N}$, and vice versa;
\item $\text{coPSh}^{\mathcal{V};\mathcal{N}}_S\cong\text{PSh}^{\mathcal{V}^\text{rev};\mathcal{N}}_S$;
\item $\text{coPSh}^\mathcal{V}(\mathcal{C};\mathcal{N})\cong\text{PSh}^{\mathcal{V}^\text{rev}}(\mathcal{C}^\text{op};\mathcal{N})$.
\end{itemize}
Therefore, copresheaves may be regarded as examples of presheaves, so everything we have already proven about presheaves is also true of copresheaves.
\end{remark}

\subsection{The internal tensor product}\label{S52}
\noindent Now we are ready to describe the right $\mathcal{V}$-module structure on $\text{PSh}^\mathcal{V}(\mathcal{C})$, for any monoidal $\infty$-category $\mathcal{V}$ and $\mathcal{V}$-enriched category $\mathcal{C}$.

Given two graphs $\Gamma_0\in\text{LM}_S$ and $\Gamma_1\in\text{RM}_T$, define $\left\langle\Gamma_0,\Gamma_1\right\rangle\in\text{Assoc}_{S\amalg T}$ as follows:
\begin{itemize}
\item an edge of $\left\langle\Gamma_0,\Gamma_1\right\rangle$ is a pair $(e_0,e_1)$, where $e_i$ is an edge of $\Gamma_i$ and at least one of the edges $e_0,e_1$ touches the basepoint vertex $\ast$;
\item if one of the edges $e_i$ does not touch the vertex $\ast$, then $(e_0,e_1)$ has the same source and target as $e_i$;
\item if $e_0$ is an edge from some $X$ to $\ast$, and $e_1$ is an edge from $\ast$ to $Y$, then $(e_0,e_1)$ is an edge from $X$ to $Y$.
\end{itemize}

\begin{example}
Let $(X_0,\ldots,X_n)=(X_0,X_1)\otimes(X_1,X_2)\otimes\cdots\otimes(X_{n-1},X_n)$ denote the graph with a single path from $X_0$ to $X_n$. Then we have $$\left\langle(X_0,\ldots,X_m),(Y_0,\ldots,Y_n)\right\rangle=\emptyset$$ $$\left\langle(X_0,\ldots,X_m,\ast),(Y_0,\ldots,Y_n)\right\rangle=(Y_0,\ldots,Y_n)$$ $$\left\langle(X_0,\ldots,X_m,\ast),(\ast,Y_0,\ldots,Y_n)\right\rangle=(X_0,\ldots,X_m,Y_0,\ldots,Y_n).$$
\end{example}

\begin{remark}
The pairing $\left\langle -,-\right\rangle$ is a functor $\text{LM}_S\times\text{RM}_T\to\text{Assoc}_{S\amalg T}$. Moreover, if $\Gamma_0\to\Gamma_0^\prime$ is inert in $\text{LM}_S$ and $\Gamma_1\to\Gamma_1^\prime$ is inert in $\text{RM}_S$, then $\left\langle\Gamma_0,\Gamma_1\right\rangle\to\left\langle\Gamma_0^\prime,\Gamma_1^\prime\right\rangle$ is inert.

So it is even a marked functor $\left\langle -,-\right\rangle:\text{LM}_S^\mathsection\times\text{RM}_S^\mathsection\to\text{Assoc}_{S\amalg T}^\mathsection$.
\end{remark}

\begin{remark}
The bimodule operad BM of HA 4.3.1.5 can be identified with the full subcategory of $\text{Assoc}_{\{0,1\}}$ spanned by graphs that have no edges from $1$ to $0$. Then our pairing factors $\text{LM}_{\{0\}}\times\text{RM}_{\{1\}}\to\text{BM}\subseteq\text{Assoc}_{\{0,1\}}$. This recovers the functor of HA 4.3.2.1, $\textbf{Pr}:\text{LM}\times\text{RM}\to\text{BM}$.
\end{remark}

\noindent For now, we will just be interested in the case $|S|=|T|=1$ and the composite $\text{LM}\times\text{RM}\to\text{Assoc}_{\{0,1\}}\to\text{Assoc}$.

\begin{definition}\label{DefLeftInert}
If $\mathcal{V}$ is a monoidal $\infty$-category, define $\smallint^\prime\mathcal{V}$ via the pullback $$\xymatrix{
\smallint^\prime\mathcal{V}\ar[d]_-{p}\ar[r] &\smallint\mathcal{V}\ar[d] \\
\text{LM}\times\text{RM}\ar[r]^-{\left\langle -,-\right\rangle} &\text{Assoc},
}$$ and say that a morphism in $\smallint^\prime\mathcal{V}$ is \emph{left inert} if it is $p$-cocartesian, its projection to LM is inert, and its projection to RM is an equivalence. We will write $\smallint^\prime\mathcal{V}^{!}$ for the left inert marking.
\end{definition}

\begin{remark}\label{RmkCases}
Identify $\text{RM}=\text{RM}_{\{1\}}$, so that there are two kinds of edges in RM: $(\ast,1)$ and $(1,1)$.

Unpacking this construction, the fiber of $\smallint^\prime\mathcal{V}\to\text{LM}\times\text{RM}$ over any $\Gamma\in\text{RM}$ is an LM-monoidal $\infty$-category, which describes:
\begin{enumerate}
\item $\mathcal{V}$ acting on itself on the left if $\Gamma=(\ast,1)$;
\item $0$ acting on $\mathcal{V}$ on the left if $\Gamma=(1,1)$ (where $0$ is the trivial monoidal $\infty$-category).
\end{enumerate}
\end{remark}

\begin{definition}\label{DefWidebar}
If $\mathcal{V}$ is a monoidal $\infty$-category, $$\widebar{\text{Cat}}^\mathcal{V}_S=\text{Fun}^\dag_{/\text{LM}}(\text{Assoc}_S^\mathsection,\smallint^\prime\mathcal{V}^{!}),$$ $$\widebar{\text{PSh}}^\mathcal{V}_S=\text{Fun}^\dag_{/\text{LM}}(\text{LM}_S^\mathsection,\smallint^\prime\mathcal{V}^{!}).$$
\end{definition}

\noindent Consider the forgetful functors $$\widebar{\text{Cat}}^\mathcal{V}_S\to\text{Fun}^\dag_{/\text{LM}}(\text{Assoc}_S^\mathsection,\text{LM}^\mathsection\times\text{RM}^\flat)\cong\text{Fun}^\dag(\text{Assoc}_S^\mathsection,\text{RM}^\flat)\cong\text{RM},$$ and similarly $\widebar{\text{PSh}}^\mathcal{V}_S\to\text{RM}$. (The rightmost equivalence is by Lemma \ref{LemSpcf}.) These fit into a commutative triangle $$\xymatrix{
\widebar{\text{PSh}}^\mathcal{V}_S\ar[rr]^-{\bar{\theta}}\ar[rd]_-{p_0} &&\widebar{\text{Cat}}^\mathcal{V}_S\ar[ld]^-{p_1} \\
&\text{RM}. &
}$$ 

\begin{lemma}\label{LemBad}
In this triangle:
\begin{enumerate}
\item $p_0$ and $p_1$ are cocartesian fibrations of $\infty$-operads;
\item a morphism in $\widebar{\text{PSh}}^\mathcal{V}_S$ (respectively $\widebar{\text{Cat}}^\mathcal{V}_S$) is $p_0$-cocartesian ($p_1$-cocartesian) if and only if the evaluation at each $\Gamma\in\text{LM}_S$ (or $\text{Assoc}_S$) is a $p$-cocartesian morphism of $\smallint^\prime\mathcal{V}$, where $p$ is the map $\smallint^\prime\mathcal{V}\to\text{LM}\times\text{RM}$;
\item $\bar{\theta}$ sends $p_0$-cocartesian morphisms to $p_1$-cocartesian morphisms.
\end{enumerate}
\end{lemma}

\begin{proof}
(3) is a direct corollary of (2). The proof of (1) and (2) is completely identical to the proof of HA 4.3.2.5, which uses categorical patterns. Rather than introduce the categorical pattern terminology here (which would appear nowhere else in this paper), see the proof of HA 4.3.2.5, and replace the construction $\bar{X}\mapsto\text{LM}^\natural\times\bar{X}$ with $\text{LM}_S^\natural$ for $\widebar{\text{PSh}}^\mathcal{V}_S$ or $\text{Assoc}_S^\natural\times -$ for $\widebar{\text{Cat}}^\mathcal{V}_S$.

The proof carries over without any modification, except that we observe $\text{LM}_S^\natural$ and $\text{Assoc}_S^\natural\in(\text{Set}^{+}_\Delta)_{/\mathfrak{P}_{\text{LM}}}$ are cofibrant as well as $\text{LM}^\natural$.
\end{proof}

\noindent By the lemma, $\bar{\theta}$ is a functor of RM-monoidal $\infty$-categories (HA 2.1.3.7).

\begin{lemma}\label{LemLast}
The RM-monoidal $\infty$-category $\widebar{\text{PSh}}^\mathcal{V}_S\xrightarrow{p_0}\text{RM}$ describes a right $\mathcal{V}$-action on $\text{PSh}^\mathcal{V}_S$.

The RM-monoidal $\infty$-category $\widebar{\text{Cat}}^\mathcal{V}_S\xrightarrow{p_1}\text{RM}$ describes the (essentially unique) right action of the trivial monoidal $\infty$-category $0$ on $\text{Cat}^\mathcal{V}_S$.
\end{lemma}

\begin{proof}
If $\Gamma\in\text{RM}$, let $\smallint^\prime\mathcal{V}_\Gamma\to\text{LM}$ denote the fiber of $\smallint^\prime\mathcal{V}\to\text{LM}\times\text{RM}$ over $\Gamma$, and note that the fiber of $p_0$ over $\Gamma$ is $\text{Fun}^\dag_{/\text{LM}}(\text{LM}_S^\mathsection,\smallint^\prime\mathcal{V}_\Gamma^\mathsection)$.

By Remark \ref{RmkCases}, $\smallint^\prime\mathcal{V}_{(\ast,0)}=\smallint(\mathcal{V};\mathcal{V})$, so the fiber of $p_0$ over $(\ast,0)$ is $\text{PSh}^\mathcal{V}_S$. On the other hand, $\smallint^\prime\mathcal{V}_{(0,0)}=\smallint(0;\mathcal{V})$, so the fiber over $(0,0)$ is $\text{PSh}^{0;\mathcal{V}}_S\cong\mathcal{V}$; this equivalence is by Corollary \ref{CorTriv}. Therefore, $p_0$ exhibits a right $\mathcal{V}$-action on $\text{PSh}^\mathcal{V}_S$. Similarly, $p_1$ exhibits a right action of $\text{Cat}^0_S\cong 0$ (this equivalence by Example \ref{ExTriv1}) on $\text{Cat}^\mathcal{V}$.
\end{proof}

\noindent Summarizing our discussion so far, the forgetful functor $\theta:\text{PSh}^\mathcal{V}_S\to\text{Cat}^\mathcal{V}_S$ extends to an RM-monoidal functor, which is to say that it is compatible with the right $\mathcal{V}$-action on $\text{PSh}^\mathcal{V}_S$ and the trivial action on $\text{Cat}^\mathcal{V}_S$. It follows that each fiber $\text{PSh}^\mathcal{V}(\mathcal{C})=\text{PSh}^\mathcal{V}_S\times_{\text{Cat}^\mathcal{V}_S}\{\mathcal{C}\}$ inherits the right $\mathcal{V}$-action. To make this more clear, here is an explicit construction of the right $\mathcal{V}$-action on $\text{PSh}^\mathcal{V}(\mathcal{C})$:

\begin{definition}\label{DefRMPSh}
By Lemma \ref{LemLast}, $\widebar{\text{Cat}}^\mathcal{V}_S\cong\smallint(0;\text{Cat}^\mathcal{V}_S)$. By Corollary \ref{CorTriv}, $$\text{Cat}^\mathcal{V}_S\cong\text{RMod}(0;\text{Cat}^\mathcal{V}_S)=\text{Fun}^\dag_{/\text{RM}}(\text{RM}^\mathsection,\widebar{\text{Cat}}^{\mathcal{V}\mathsection}_S).$$ If $\mathcal{C}\in\text{Cat}^\mathcal{V}_S$, let $\mathcal{C}_\ast:\text{RM}\to\widebar{\text{Cat}}^\mathcal{V}_S$ be the associated functor, and define $\widebar{\text{PSh}}^\mathcal{V}(\mathcal{C})$ to be the pullback $$\xymatrix{
\widebar{\text{PSh}}^\mathcal{V}(\mathcal{C})\ar[r]\ar[d] &\widebar{\text{PSh}}^\mathcal{V}_S\ar[d] \\
\text{RM}\ar[r]_-{\mathcal{C}_\ast} &\widebar{\text{Cat}}^\mathcal{V}_S.
}$$
\end{definition}

\begin{theorem}\label{PropRMPSh2}
If $\mathcal{V}$ is a monoidal $\infty$-category and $\mathcal{C}\in\text{Cat}^\mathcal{V}_S$, then the functor $\widebar{\text{PSh}}^\mathcal{V}(\mathcal{C})\to\text{RM}$ is a cocartesian fibration of $\infty$-operads, and the associated RM-monoidal $\infty$-category describes a right $\mathcal{V}$-action on $\text{PSh}^\mathcal{V}(\mathcal{C})$.
\end{theorem}

\noindent In other words, we have constructed $\text{PSh}^\mathcal{V}(\mathcal{C})$ as a right $\mathcal{V}$-module.

\begin{proof}
The identity functor $\text{RM}\to\text{RM}$ is a cocartesian fibration of $\infty$-categories corresponding to the (unique) action of the monoidal $\infty$-category $0$ on itself. By Lemma \ref{LemBad}, $\mathcal{C}_\ast$ sends all morphisms to $p_1$-cocartesian morphisms, and the composite $p_1\mathcal{C}_\ast:\text{RM}\to\text{RM}$ is the identity functor.

Therefore, $\mathcal{C}_\ast$ is a functor of RM-monoidal $\infty$-categories, so the pullback $\widebar{\text{PSh}}^\mathcal{V}(\mathcal{C})\to\text{RM}$ is also an RM-monoidal $\infty$-category, exhibiting an action of $\mathcal{V}\times_00\cong\mathcal{V}$ on $\text{PSh}^\mathcal{V}_S\times_{\text{Cat}^\mathcal{V}_S}\{\mathcal{C}\}=\text{PSh}^\mathcal{V}(\mathcal{C})$.
\end{proof}

\subsection{Properties of the internal tensor product}
\begin{proposition}\label{PropEv1}
Suppose $\mathcal{C}\in\text{Cat}^\mathcal{V}_S$ and $X\in S$. Then
\begin{enumerate}
\item Evaluation $\text{ev}_X:\text{PSh}^\mathcal{V}_S\to\mathcal{V}$ promotes to a right $\mathcal{V}$-module functor;
\item If a morphism $A\to\mathcal{F}(X)$ exhibits $\mathcal{F}\in\text{PSh}^\mathcal{V}(\mathcal{C})$ as freely generated by $A$ at $X$ (Definition \ref{DefFree}), then $A\otimes B\to\mathcal{F}(X)\otimes B\cong(\mathcal{F}\otimes B)(X)$ exhibits $\mathcal{F}\otimes B$ as freely generated by $A\otimes B$ at $X$.
\end{enumerate}
\end{proposition}

\noindent (1) asserts that $\mathcal{F}\otimes A$ can be `computed', in that $(\mathcal{F}\otimes A)(X)\cong\mathcal{F}(X)\otimes A$. (2) asserts that tensor products of free presheaves are free. That is, the free presheaf $\text{rep}_X\otimes A$ (notation introduced in Section \ref{S33}) is equivalent to the tensor presheaf $\text{rep}_X\otimes A$ (in the sense of the right $\mathcal{V}$-action on presheaves).

In light of Proposition \ref{PropEv1}, Corollary \ref{CorFreeCo} can be restated:

\begin{repcorollary}{CorFreeCo2}
If $\mathcal{V}$ is presentable and closed monoidal and $\mathcal{C}$ is $\mathcal{V}$-enriched, then $\text{PSh}^\mathcal{V}(\mathcal{C})$ is generated under colimits and the right $\mathcal{V}$-action by the representable presheaves $\text{rep}_X$.
\end{repcorollary}

\begin{proof}[Proof of Proposition \ref{PropEv1}]
Consider the functor $$\widebar{\text{ev}}_X:\widebar{\text{PSh}}^\mathcal{V}_S=\text{Fun}^\dag_{/\text{LM}}(\text{LM}_S^\mathsection,\smallint^\prime\mathcal{V}^{!})\to\smallint\mathcal{V}$$ given by evaluation at $(X,\ast)\in\text{LM}_S$ composed with $\smallint^\prime\mathcal{V}\to\smallint\mathcal{V}$. By Lemma \ref{LemBad}, $\widebar{\text{ev}}_X$ sends $p_0$-cocartesian morphisms to $p$-cocartesian morphisms, where $p:\smallint\mathcal{V}\to\text{RM}$, so the induced functor $\widebar{\text{PSh}}^\mathcal{V}_S\to\smallint(\mathcal{V};\mathcal{V})=\smallint\mathcal{V}\times_{\text{Assoc}}\text{RM}$ is an RM-monoidal functor $\text{ev}_X:(\mathcal{V};\text{PSh}^\mathcal{V}_S)\to(\mathcal{V};\mathcal{V})$. Moreover, the restriction to $\text{Assoc}\subseteq\text{RM}$ is the identity on $\mathcal{V}$, so this describes a right $\mathcal{V}$-module functor $\text{ev}_X:\text{PSh}^\mathcal{V}_S\to\mathcal{V}$. It agrees with the functor we called $\text{ev}_X$ in Section \ref{S33} because they are both given by evaluation at $(X,\ast)\in\text{LM}_S$.

The second claim follows directly from Definition \ref{DefFree} and (1).
\end{proof}

\begin{corollary}\label{CorEv1}
Let $\widebar{\text{ev}}_X:\widebar{\text{PSh}}^\mathcal{V}(\mathcal{C})\to\smallint(\mathcal{V};\mathcal{V})$ be the RM-monoidal functor associated to $\text{ev}_X$, guaranteed by Proposition \ref{PropEv1}(1). Then $\widebar{\text{ev}}_X$ has a right adjoint $\widebar{\text{rep}}_X\otimes -$, which is also an RM-monoidal functor.
\end{corollary}

\begin{proof}
By HA 7.3.2.7, $\widebar{\text{ev}}_X$ has a right adjoint which is a map of $\infty$-operads over RM. (Intuitively, this is the statement that the right adjoint to an RM-monoidal functor is canonically a lax RM-monoidal functor.)

To prove that $\widebar{\text{rep}}_X\otimes -$ is not just lax but fully RM-monoidal, it suffices to prove that for each $A,B\in\mathcal{V}$, the map $(\text{rep}_X\otimes A)\otimes B\to\text{rep}_X\otimes(A\otimes B)$ is an equivalence of presheaves. But this is true by Proposition \ref{PropEv1}(2).
\end{proof}

\begin{remark}
That is, in the adjunction $\text{ev}_X:\text{PSh}^\mathcal{V}(\mathcal{C})\leftrightarrows\mathcal{V}:\text{rep}_X\otimes -$, both adjoints promote to right $\mathcal{V}$-module functors, and in a compatible way.
\end{remark}

\noindent For an arbitrary monoidal $\infty$-category $\mathcal{V}$, we can't do any better than construct the right $\mathcal{V}$-action on $\text{PSh}^\mathcal{V}(\mathcal{C})$. However, if $\mathcal{V}$ is presentable, we will prove that $\text{PSh}^\mathcal{V}(\mathcal{C})$ varies functorially in $\mathcal{C}$.

\begin{theorem}\label{PropPrRight}
If $\mathcal{V}$ is presentable and closed monoidal and $\mathcal{C}\in\text{Cat}^\mathcal{V}_S$, then $\text{PSh}^\mathcal{V}(\mathcal{C})$ is a presentable right $\mathcal{V}$-module.
\end{theorem}

\noindent Together with Proposition \ref{PropEv1}(1), this establishes Theorem \ref{PropRMPSh} from the introduction.

\begin{proof}
Since $\text{PSh}^\mathcal{V}(\mathcal{C})$ is presentable (Theorem \ref{PShPrL2}), we just need to prove that $\text{PSh}^\mathcal{V}(\mathcal{C})\times\mathcal{V}\to\text{PSh}^\mathcal{V}(\mathcal{C})$ preserves colimits independently in each variable. By Proposition \ref{PropEv1}(1), the following square commutes: $$\xymatrix{
\text{PSh}^\mathcal{V}(\mathcal{C})\times\mathcal{V}\ar[r]^-{\otimes}\ar[d]_-{\text{ev}_X} &\text{PSh}^\mathcal{V}(\mathcal{C})\ar[d]^-{\text{ev}_X} \\
\mathcal{V}\times\mathcal{V}\ar[r]_-{\otimes} &\mathcal{V}.
}$$ By Theorem \ref{PShPrL2} (and taking the upper composite around the square), it suffices to show that for each $X\in S$, the composite $\text{PSh}^\mathcal{V}(\mathcal{C})\times\mathcal{V}\to\mathcal{V}$ preserves colimits independently in each variable. However, we know $\text{ev}_X$ preserves colimits (Theorem \ref{PShPrL2}) and $\mathcal{V}\times\mathcal{V}\xrightarrow{\otimes}\mathcal{V}$ preserves colimits independently in each variable (since $\mathcal{V}$ is closed monoidal), so it follows that the lower composite preserves colimits in each variable separately. Since the square commutes, this completes the proof.
\end{proof}

\begin{proposition}\label{PropBarCo}
If $\mathcal{V}$ is presentable and closed monoidal, $\widebar{\text{PSh}}^\mathcal{V}_S\to\widebar{\text{Cat}}^\mathcal{V}_S$ is a cocartesian fibration.
\end{proposition}

\begin{lemma}
Suppose we have a commutative triangle of $\infty$-categories $$\xymatrix{
\mathcal{X}\ar[rr]^-{F}\ar[rd]_-{G} &&\mathcal{Y}\ar[ld]^-{H} \\
&\mathcal{Z} &
}$$ such that $G$ and $H$ are cocartesian fibrations, and $F$ sends $G$-cocartesian morphisms to $H$-cocartesian morphisms. Further assume that for any $a\in\mathcal{Z}$, $F_a:\mathcal{X}_a\to\mathcal{Y}_a$ is a cocartesian fibration, and for any $a\to b$ in $\mathcal{Z}$, the induced functor $\mathcal{X}_a\to\mathcal{X}_b$ sends $F_a$-cocartesian morphisms to $F_b$-cocartesian morphisms. Then $F$ is a cocartesian fibration.
\end{lemma}

\begin{proof}
This lemma is proven in a special case in HA 4.8.3.15, but the proof is general. We repeat it here. Let $f_0:A\to B$ be a morphism in $\mathcal{Y}$ lying over $\alpha:s\to t$ in $\mathcal{Z}$, and let $M\in\mathcal{X}$ be a lift of $A$. We are looking for an $F$-cocartesian lift $M\to N$ of $f_0$.

By assumption, there is a $G$-cocartesian morphism $f^\prime:M\to M^\prime$ lifting $\alpha$, which is therefore also $F$-cocartesian. Write $f_0^\prime:A\to A^\prime$ for $F(f^\prime)$. Also by assumption, $f_0^\prime$ is a cocartesian lift of $\alpha$, so $f_0$ factors $$\xymatrix{
&A^\prime\ar[rd]^-{f_0^{\prime\prime}} &\\
A\ar[rr]_-{f_0}\ar[ru]^-{f_0^\prime} &&B,
}$$ for some $f_0^{\prime\prime}$ which projects to the identity morphism $\text{id}_t$ in $\mathcal{Z}$. By assumption, there is an $F_t$-cocartesian lift $f^{\prime\prime}:M^\prime\to N$ of $f_0^{\prime\prime}$. Moreover, any $t\to t^\prime$ in $\mathcal{Z}$ induces a functor $\mathcal{X}_t\to\mathcal{X}_{t^\prime}$ which sends $f^{\prime\prime}$ to an $F_{t^\prime}$-cocartesian morphism; therefore, $f^{\prime\prime}$ is $F$-cocartesian.

So $f^{\prime\prime}f^\prime:M\to N$ is a composite of $F$-cocartesian morphisms, and therefore itself an $F$-cocartesian morphism (HTT 2.4.1.7) lifting $f_0$. Hence $F$ is a cocartesian fibration.
\end{proof}

\begin{proof}[Proof of Proposition \ref{PropBarCo}]
It suffices to prove that the commutative triangle $$\xymatrix{
\widebar{\text{PSh}}^\mathcal{V}_S\ar[rr]^-{\bar{\theta}}\ar[rd]_-{p_0} &&\widebar{\text{Cat}}^\mathcal{V}_S\ar[ld]^-{p_1} \\
&\text{RM} &
}$$ satisfies the hypotheses of the lemma. By Lemma \ref{LemBad}, $p_0$ and $p_1$ are cocartesian fibrations, and $\bar{\theta}$ sends $p_0$-cocartesian morphisms to $p_1$-cocartesian morphisms.

Next, we show that $\bar{\theta}_\Gamma:(\widebar{\text{PSh}}^\mathcal{V}_S)_\Gamma\to(\widebar{\text{Cat}}^\mathcal{V}_S)_\Gamma$ is a cocartesian fibration for any $\Gamma\in\text{RM}$. If $\Gamma=(\ast,0)$, then $\bar{\theta}_\Gamma$ is just $\theta:\text{PSh}^\mathcal{V}_S\to\text{Cat}^\mathcal{V}_S$, which is a cocartesian fibration by Corollary \ref{Cor1S4}. If $\Gamma=(0,0)$, then $\bar{\theta}_\Gamma$ is just $\mathcal{V}\to\ast$, which is a cocartesian fibration (as is any functor to $\ast$). For general $\Gamma$, $\bar{\theta}_\Gamma$ is a product of copies of $\bar{\theta}_{(\ast,0)}$ and $\bar{\theta}_{(0,0)}$, which are all cocartesian fibrations.

Now consider $f:\Gamma\to\Gamma^\prime$ in $\text{RM}$. We are reduced to showing that $f_\ast:(\widebar{\text{PSh}}^\mathcal{V}_S)_\Gamma\to(\widebar{\text{PSh}}^\mathcal{V}_S)_{\Gamma^\prime}$ sends $\bar{\theta}_\Gamma$-cocartesian morphisms to $\bar{\theta}_{\Gamma^\prime}$-cocartesian morphisms. Every morphism in $\text{RM}$ is a product of morphisms of these types:
\begin{enumerate}
\item inert morphisms $(0,0)\to\emptyset$ and $(\ast,0)\to\emptyset$;
\item $\emptyset\to(0,0)$;
\item $(\ast,0)\otimes(0,0)\to(\ast,0)$.
\end{enumerate}
\noindent Therefore it suffices to prove just in these three cases that $f_\ast$ sends cocartesian morphisms to cocartesian morphisms. For (1) and (2) this is clear because $(\text{PSh}^\mathcal{V}_S)_\emptyset=\ast$. For (3), the claim reduces to the following:

If $F:\mathcal{C}\to\mathcal{D}$ is a map in $\text{Cat}^\mathcal{V}_S$ and $A\in\mathcal{V}$, then for every $\mathcal{F}\in\text{PSh}^\mathcal{V}(\mathcal{C})$, the map $\eta_\mathcal{F}:F_\ast(\mathcal{F}\otimes A)\to F_\ast(\mathcal{F})\otimes A$ is an equivalence of presheaves on $\mathcal{D}$. (In other words, we are claiming $F_\ast:\text{PSh}^\mathcal{V}(\mathcal{C})\to\text{PSh}^\mathcal{V}(\mathcal{D})$ is compatible with the right $\mathcal{V}$-module structures; this is the main content of the proof.)

If $\mathcal{F}=\text{rep}_X\otimes B$ for some $B\in\mathcal{V}$, then the claim is true because free presheaves are sent to free presheaves by both the constructions $F_\ast$ (Remark \ref{RmkFuncRep}) and $-\otimes A$ (Proposition \ref{PropEv1}(2)). Hence we have a natural transformation $\eta:F_\ast(-\otimes A)\to F_\ast(-)\otimes A$ of functors $\text{PSh}^\mathcal{V}(\mathcal{C})\to\text{PSh}^\mathcal{V}(\mathcal{D})$, both functors preserve colimits, and $\eta$ is an equivalence at representables. Since $\text{PSh}^\mathcal{V}(\mathcal{C})$ is generated under colimits by representables (Corollary \ref{CorFreeCo2}), $\eta$ is an equivalence at all presheaves. Therefore $\bar{\theta}$ satisfies the hypotheses of the lemma, so it is a cocartesian fibration.
\end{proof}

\noindent Now we are ready to show that the right $\mathcal{V}$-modules $\text{PSh}^\mathcal{V}(\mathcal{C})$ are functorial in $\mathcal{C}$. Recall from Definition \ref{DefRMPSh} that $\text{Cat}^\mathcal{V}_S\cong\text{Fun}^\dag_{/\text{RM}}(\text{RM}^\mathsection,\widebar{\text{Cat}}^{\mathcal{V}\mathsection}_S)$. By adjunction, we have a functor $\text{RM}\times\text{Cat}^\mathcal{V}_S\to\widebar{\text{Cat}}^\mathcal{V}_S$. Form the pullback $$\xymatrix{
\widebar{\text{PSh}}^\mathcal{V}(\text{Cat}^\mathcal{V}_S)\ar[r]\ar[d]_p &\widebar{\text{PSh}}^\mathcal{V}_S\ar[d]^-{\bar{\theta}} \\
\text{RM}\times\text{Cat}^\mathcal{V}_S\ar[r] &\widebar{\text{Cat}}^\mathcal{V}_S.
}$$ By Proposition \ref{PropBarCo}, $\bar{\theta}$ is a cocartesian fibration, and therefore $p$ is also a cocartesian fibration. Moreover, each fiber $p_\mathcal{C}:\widebar{\text{PSh}}^\mathcal{V}(\mathcal{C})\to\text{RM}$ is a cocartesian fibration of $\infty$-operads by Theorem \ref{PropRMPSh2}. In other words, $p$ is a cocartesian $\text{Cat}^\mathcal{V}_S$-family of RM-monoidal $\infty$-categories in the sense of HA 4.8.3.1. By Theorem \ref{PropPrRight}, it is also compatible with small colimits in the sense of HA 4.8.3.4.

Therefore $p$ is classified by a functor $\text{Cat}^\mathcal{V}_S\to\text{Mod}_\text{RM}(\text{Pr}^L)$, which sends $\mathcal{C}$ to $\text{PSh}^\mathcal{V}(\mathcal{C})$ as a right $\mathcal{V}$-module (Ha 4.8.3.20). The composite with the forgetful functor $\text{Mod}_\text{RM}(\text{Pr}^L)\to\text{Mod}_\text{Assoc}(\text{Pr}^L)$ is constant with value $\mathcal{V}$, so $\text{PSh}^\mathcal{V}(-)$ factors through the fiber $\text{RMod}_\mathcal{V}(\text{Pr}^L)\subseteq\text{Mod}_\text{RM}(\text{Pr}^L)$ over $\mathcal{V}\in\text{Mod}_\text{Alg}(\text{Pr}^L)$. To summarize:

\begin{corollary}\label{CorModFunc}
If $\mathcal{V}$ is presentable and closed monoidal, then the cocartesian fibration $p:\widebar{\text{PSh}}^\mathcal{V}(\text{Cat}^\mathcal{V}_S)\to\text{RM}\times\text{Cat}^\mathcal{V}_S$ classifies a functor $$\text{PSh}^\mathcal{V}(-):\text{Cat}^\mathcal{V}_S\to\text{RMod}_\mathcal{V}(\text{Pr}^L).$$
\end{corollary}

\subsection{The external tensor product}\label{S53}
\noindent We will end this section by proving Theorem \ref{ThmS5}: that for any presentable pair $(\mathcal{V};\mathcal{M})$, there is an equivalence of $\infty$-categories $$\Psi:\text{PSh}^\mathcal{V}(\mathcal{C})\otimes_\mathcal{V}\mathcal{M}\to\text{PSh}^\mathcal{V}(\mathcal{C};\mathcal{M}).$$ Informally, $\mathcal{F}\otimes M$ corresponds to the presheaf which assigns $X\mapsto\mathcal{F}(X)\otimes M$. The construction of $\Psi$ is quite technical. However, the only properties of $\Psi$ that we will use are:

\begin{lemma}\label{LemComp}
Suppose $(\mathcal{V};\mathcal{M})$ is a presentable pair and $\mathcal{C}$ is $\mathcal{V}$-enriched.
\begin{enumerate}
\item For any $X\in\mathcal{C}$, the following diagram commutes: $$\xymatrix{
\text{PSh}^\mathcal{V}(\mathcal{C})\otimes_\mathcal{V}\mathcal{M}\ar[r]^-{\Psi}\ar[d]_-{\text{ev}_X} &\text{PSh}^\mathcal{V}(\mathcal{C};\mathcal{M})\ar[d]^-{\text{ev}_X} \\
\mathcal{V}\otimes_\mathcal{V}\mathcal{M}\ar@{=}[r] &\mathcal{M}.
}$$
\item For any $X\in\mathcal{C}$ and $M\in\mathcal{M}$, $M\to\mathcal{C}(X,X)\otimes M\cong\Psi(\text{rep}_X\otimes M)(X)$ exhibits $\Psi(\text{rep}_X\otimes M)$ as freely generated by $M$ at $X$, in the sense of Definition \ref{DefFree}.
\end{enumerate}
\end{lemma}

\noindent In other words, (1) asserts that $\Psi(\mathcal{F}\otimes M)(X)\cong\mathcal{F}(X)\otimes M$, and (2) asserts that $\Psi(\text{rep}_X\otimes M)$ is the free presheaf $\text{rep}_X\otimes M$.

We will proceed with the proof of Theorem \ref{ThmS5}, conditional on Lemma \ref{LemComp}. Then, at the end of the section, we will construct the functor $\Psi$ and prove Lemma \ref{LemComp}.

First, we need two lemmas concerning monadic functors which respect a module structure.

\begin{lemma}
Suppose $F:\mathcal{X}\to\mathcal{Y}$ is a map in $\text{RMod}_\mathcal{V}(\text{Pr}^L)$: that is, a $\mathcal{V}$-module functor with a right adjoint. Also suppose $F$ is monadic and its left adjoint is compatible with the $\mathcal{V}$-module structure (as in Corollary \ref{CorEv1}).

Then there is a monoidal $\infty$-category $\mathcal{E}$ such that $\mathcal{Y}$ is an $(\mathcal{E},\mathcal{V})$-bimodule, and an algebra $A\in\mathcal{E}$ such that there is an equivalence $f$ of right $\mathcal{V}$-module $\infty$-categories making the triangle commute: $$\xymatrix{
\mathcal{X}\ar[rr]^-{f}\ar[rd]_-{F} &&\text{LMod}_A(\mathcal{Y})\ar[ld] \\
&\mathcal{Y}.&
}$$
\end{lemma}

\begin{proof}
Let $\mathcal{E}=\text{End}_{\mathcal{V}\text{Mod}}(\mathcal{Y})$, the monoidal $\infty$-category of $\mathcal{V}$-module functors from $\mathcal{Y}$ to itself. Then $\mathcal{Y}\in\text{RMod}_\mathcal{V}(\text{Cat})$ is a left $\mathcal{E}$-module, so therefore $\mathcal{Y}$ is an $(\mathcal{E},\mathcal{V})$-bimodule by HA 4.3.3.8.

Moreover, $\mathcal{F}=\text{Fun}_{\text{RMod}_\mathcal{V}}(\mathcal{X},\mathcal{Y})$ is a left $\mathcal{E}$-module. If $G$ denotes the left adjoint to $F$, then $A=FG$ is an algebra in $\mathcal{E}$, and $F$ is a left $A$-module. Therefore, $F$ factors $\mathcal{X}\xrightarrow{f}\text{LMod}_A(\mathcal{Y})\to\mathcal{Y}$, where $f$ is a right $\mathcal{V}$-module functor. By Barr-Beck (HA 4.7.3.16), $f$ is an equivalence, because the induced morphism of monads $A\to A$ is an equivalence by construction.
\end{proof}

\begin{lemma}\label{LemS5Prep}
If $F:\mathcal{X}\to\mathcal{Y}$ is a monadic right $\mathcal{V}$-module functor as in the last lemma, and $\mathcal{M}$ is a left $\mathcal{V}$-module, $\mathcal{X}\otimes_\mathcal{V}\mathcal{M}\to\mathcal{Y}\otimes_\mathcal{V}\mathcal{M}$ is also monadic.
\end{lemma}

\begin{proof}
By the last lemma, we can assume without loss of generality that $\mathcal{X}=\text{LMod}_A(\mathcal{Y})$ and $F$ is the forgetful functor, where $A$ is an algebra in some $\mathcal{E}$, such that $\mathcal{Y}$ is an $(\mathcal{E},\mathcal{V})$-bimodule. By HA 4.8.4.6, $F$ factors $$\text{LMod}_A(\mathcal{Y})\cong\text{LMod}_A(\mathcal{E})\otimes_\mathcal{E}\mathcal{Y}\xrightarrow{F^\prime}\mathcal{E}\otimes_\mathcal{E}\mathcal{Y}\cong\mathcal{Y}.$$ Therefore, the functor that we wish to show is monadic is $$\text{LMod}_A(\mathcal{E})\otimes_\mathcal{E}\mathcal{Y}\otimes_\mathcal{V}\mathcal{M}\to\mathcal{Y}\otimes_\mathcal{V}\mathcal{M}.$$ This is equivalent (by HA 4.8.4.6 again) to $$\text{LMod}_A(\mathcal{Y}\otimes_\mathcal{V}\mathcal{M})\to\mathcal{Y}\otimes_\mathcal{V}\mathcal{M},$$ which is indeed monadic.
\end{proof}

\begin{reptheorem}{ThmS5}
If $\mathcal{M}$ is a presentable left $\mathcal{V}$-module, then $$\Psi:\text{PSh}^\mathcal{V}(\mathcal{C};\mathcal{M})\to\text{PSh}^\mathcal{V}(\mathcal{C})\otimes_\mathcal{V}\mathcal{M}$$ is an equivalence.
\end{reptheorem}

\begin{proof}[Proof (conditional on Lemma \ref{LemComp})]
Suppose $\mathcal{C}$ has set $S$ of objects. The proof is by Barr-Beck, following HA 4.8.4.6, which is the case $|S|=1$. By Lemma \ref{LemComp}(1), the following triangle commutes: $$\xymatrix{
\text{PSh}^\mathcal{V}(\mathcal{C})\otimes_\mathcal{V}\mathcal{M}\ar[rr]^-{\Phi}\ar[rd]_-{G} &&\text{PSh}^\mathcal{V}(\mathcal{C};\mathcal{M})\ar[ld]^-{G^\prime} \\
&\mathcal{M}^S, &
}$$ where $G$ and $G^\prime$ are given by evaluation at each object in $S$ (and $G$ factors through the equivalence $\mathcal{V}^S\otimes_\mathcal{V}\mathcal{M}\cong\mathcal{M}^S$). We know $G^\prime$ preserves small limits and colimits by Theorem \ref{PShPrL2}(2), so it has a left adjoint $F^\prime$. It is also conservative by Corollary \ref{CorPFib1b}. Therefore, it is monadic. By Lemma \ref{LemS5Prep}, $G$ is also monadic; call its left adjoint $F$.

According to the Barr-Beck Theorem (HA 4.7.3.16-17), $\Phi$ is an equivalence if $G,G^\prime$ are both monadic and the induced natural transformation $G^\prime F^\prime\to GF$ of functors $\mathcal{M}^S\to\mathcal{M}^S$ is an equivalence.

We will prove this is true. Since $G^\prime F^\prime$ and $GF$ preserve colimits, it suffices to check for each $X,Y\in S$ that $G^\prime_X F^\prime_Y\to G_XF_Y$ is a natural equivalence of functors $\mathcal{M}\to\mathcal{M}$, which is to say $\Phi(\text{rep}_Y\otimes M)(X)\to\text{rep}_Y(X)\otimes M$ is an equivalence for all $M\in\mathcal{M}$. This is true by Lemma \ref{LemComp}(2).
\end{proof}

\noindent Finally, we will construct the functor $\Psi$ and prove that it satisfies Lemma \ref{LemComp}. The construction follows HA 4.8.4.4, which is the case $|S|=1$.

\begin{remark}\label{RmkTerrible}
Let $\mathcal{X}=(\smallint\text{PSh}^\mathcal{V}(\mathcal{C})\rtimes\mathcal{BV})\times_{\smallint\mathcal{BV}}(\smallint\mathcal{BV}\ltimes\mathcal{M})$ as in Proposition \ref{PropLR}, which is equipped with a cocartesian fibration $r:\mathcal{X}\to\Delta^\text{op}$. We should think of an object of $\mathcal{X}$ as a tuple $(\mathcal{F},A_0,\ldots,A_n,M)\in\mathcal{X}_n$, where $\mathcal{F}\in\text{PSh}^\mathcal{V}(\mathcal{C})$, $A_i\in\mathcal{V}$, and $M\in\mathcal{M}$.

By Proposition \ref{PropLR}, constructing $\Psi:\text{PSh}^\mathcal{V}(\mathcal{C})\otimes_\mathcal{V}\mathcal{M}\to\text{PSh}^\mathcal{V}(\mathcal{C};\mathcal{M})$ is equivalent to constructing a functor $\mathcal{X}\to\text{PSh}^\mathcal{V}(\mathcal{C};\mathcal{M})$ which sends $r$-cocartesian morphisms to equivalences and (Remark \ref{PropLRRmk}) such that the restriction $\text{PSh}^\mathcal{V}(\mathcal{C})\times\mathcal{V}^{\times n}\times\mathcal{M}\to\text{PSh}^\mathcal{V}(\mathcal{C};\mathcal{M})$ preserves colimits for each $n$. By Proposition \ref{PropPShModel}, we have $\text{PSh}^\mathcal{V}(\mathcal{C};\mathcal{M})\cong\text{Fun}^\dag_{/\smallint\mathcal{BV}^\mathsection}(\Delta_{/S}^{\text{op}\ddagger},\smallint\mathcal{BV}\ltimes\mathcal{M}^\ddagger)$. Hence we need to construct a functor $\psi:\mathcal{X}\times\Delta_{/S}^\text{op}\to\smallint\mathcal{BV}\ltimes\mathcal{M}$ such that:
\begin{enumerate}
\item $\psi$ sends $r$-cocartesian morphisms in the first coordinate to equivalences;
\item $\psi$ sends totally inert morphisms in the second coordinate to totally inert morphisms;
\item the following square commutes $$\xymatrix{
\mathcal{X}\times\Delta_{/S}^\text{op}\ar[r]^-{\psi}\ar[d] &\smallint\mathcal{BV}\ltimes\mathcal{M}\ar[d] \\
\Delta_{/S}^\text{op}\ar[r]_-{\mathcal{C}} &\smallint\mathcal{BV};
}$$
\item for each $n$, $\text{PSh}^\mathcal{V}(\mathcal{C})\times\mathcal{V}^{\times n}\times\mathcal{M}\to\text{PSh}^\mathcal{V}(\mathcal{C};\mathcal{M})$ preserves colimits.
\end{enumerate}
\end{remark}

\noindent Let $\star:(\smallint\mathcal{BV}\ltimes\mathcal{V})\times(\smallint\mathcal{BV}\ltimes\mathcal{M})\to\smallint\mathcal{BV}\ltimes\mathcal{M}$ denote concatenation, defined by $(A_0,\ldots,A_m)\star(A_{m+1},\ldots,A_n,M)=(A_0,\ldots,A_n,M)$. Define $\phi_0$, $\phi_1$ to be the functors (respectively) $$\mathcal{X}\times\Delta_{/S}^\text{op}\to\text{PSh}^\mathcal{V}(\mathcal{C})\times\Delta_{/S}^\text{op}\to\text{Fun}(\Delta_{/S}^\text{op},\smallint\mathcal{BV}\ltimes\mathcal{V})\times\Delta_{/S}^\text{op}\to\smallint\mathcal{BV}\ltimes\mathcal{V},$$ $$\mathcal{X}\times\Delta_{/S}^\text{op}\to\mathcal{X}\to\smallint\mathcal{BV}\ltimes\mathcal{M},$$ and $\phi=\phi_0\star\phi_1$, which is also a functor $\mathcal{X}\times\Delta_{/S}^\text{op}\to\smallint\mathcal{BV}\ltimes\mathcal{M}$. Explicitly, if $T=((\mathcal{F},A_0,\ldots,A_n,M),(X_0<\cdots<X_m))\in\mathcal{X}\times\Delta^\text{op}_{/S}$, then $$\phi_0(T)=(\mathcal{C}(X_0,X_1),\ldots,\mathcal{C}(X_{m-1},X_m),\mathcal{F}(X_m))\in\smallint\mathcal{BV}\ltimes\mathcal{V},$$ $$\phi_1(T)=(A_1,\ldots,A_n,M)\in\smallint\mathcal{BV}\ltimes\mathcal{M},$$ $$\phi(T)=(\mathcal{C}(X_0,X_1),\ldots,\mathcal{C}(X_{m-1},X_m),\mathcal{F}(X_m),A_1,\ldots,A_n,M)\in\smallint\mathcal{BV}\ltimes\mathcal{M}.$$ By construction, the following diagram commutes: $$\xymatrix{
\mathcal{X}\times\Delta^\text{op}_{/S}\ar[r]^\phi\ar[d] &\smallint\mathcal{BV}\ltimes\mathcal{M}\ar[d]^-{q} \\
\Delta^\text{op}_{/S}\times\Delta^\text{op}_{/S}\ar[r]_-{\star} &\Delta^\text{op}_{/S},
}$$ where $(X_0<\cdots<X_m)\star(X_{m+1}<\cdots<X_n)=(X_0<\cdots<X_n)$. Call the composite $\bar{\phi}:\mathcal{X}\times\Delta^\text{op}_{/S}\to\Delta^\text{op}_{/S}$. On the other hand, call $\bar{\psi}:\mathcal{X}\times\Delta^\text{op}_{/S}\to\Delta^\text{op}_{/S}$ the projection onto the second coordinate. There is a canonical natural transformation $\bar{\phi}\to\bar{\psi}$ of inert morphisms, given essentially by inclusion $X\subseteq X\star Y$ in $\Delta_{/S}^\text{op}$.

We know $q:\text{Fun}(\mathcal{X}\times\Delta^\text{op}_{/S},\smallint\mathcal{BV}\ltimes\mathcal{M})\to\text{Fun}(\mathcal{X}\times\Delta^\text{op}_{/S},\Delta^\text{op}_{/S})$ is a cocartesian fibration since $\smallint\mathcal{BV}\ltimes\mathcal{M}\to\Delta^\text{op}_{/S}$ is a cocartesian fibration. Therefore, there is a $q$-cocartesian lift $\phi\to\psi$ of $\bar{\phi}\to\bar{\psi}$, essentially defined by $$\psi(T)=(\mathcal{C}(X_0,X_1),\ldots,\mathcal{C}(X_{m-1},X_m),\mathcal{F}(X_m)\otimes A_1\otimes\cdots\otimes A_n\otimes M).$$ Unpacking, $\psi$ satisfies conditions (1)-(3) of Remark \ref{RmkTerrible}, and it also satisfies condition (4) by Theorem \ref{PShPrL2}(4), so there is an induced functor $$\Psi:\text{PSh}^\mathcal{V}(\mathcal{C})\otimes_\mathcal{V}\mathcal{M}\to\text{PSh}^\mathcal{V}(\mathcal{C};\mathcal{M}).$$ We are finally ready to prove Lemma \ref{LemComp}, which we restate here for reference.

\begin{replemma}{LemComp}
Suppose $(\mathcal{V};\mathcal{M})$ is a presentable pair and $\mathcal{C}$ is $\mathcal{V}$-enriched.
\begin{enumerate}
\item For any $X\in\mathcal{C}$, the following diagram commutes: $$\xymatrix{
\text{PSh}^\mathcal{V}(\mathcal{C})\otimes_\mathcal{V}\mathcal{M}\ar[r]^-{\Psi}\ar[d]_-{\text{ev}_X} &\text{PSh}^\mathcal{V}(\mathcal{C};\mathcal{M})\ar[d]^-{\text{ev}_X} \\
\mathcal{V}\otimes_\mathcal{V}\mathcal{M}\ar@{=}[r] &\mathcal{M}.
}$$
\item For any $X\in\mathcal{C}$ and $M\in\mathcal{M}$, $M\to\mathcal{C}(X,X)\otimes M\cong\Psi(\text{rep}_X\otimes M)(X)$ exhibits $\Psi(\text{rep}_X\otimes M)$ as freely generated by $M$ at $X$, in the sense of Definition \ref{DefFree}.
\end{enumerate}
\end{replemma}

\begin{proof}[Proof of Lemma \ref{LemComp}]
By construction, $\text{ev}_X\circ\Psi$ corresponds to the functor $$\mathcal{X}\times\{(X_0)\}\subseteq\mathcal{X}\times\Delta_{/S}^\text{op}\xrightarrow{\psi}\smallint\mathcal{BV}\ltimes\mathcal{M}\to\mathcal{M};$$ this sends $(\mathcal{F},A_0,\ldots,A_n,M)$ to $\mathcal{F}(X)\otimes A_0\otimes\cdots\otimes A_n\otimes M$, which is the tensor product of $\mathcal{X}\to\text{PSh}^\mathcal{V}(\mathcal{C})\xrightarrow{\text{ev}_X}\mathcal{V}$ with $\mathcal{X}\to\smallint(\mathcal{V};\mathcal{M})\xrightarrow{\otimes}\mathcal{M}$. Unpacking definitions, this implies (1).

To prove (2), we just need to check Definition \ref{DefFree}; that is, we need to prove that $\mathcal{C}(Y,X)\otimes M\to\Psi(\text{rep}_X\otimes M)(Y)$ is an equivalence for all $Y\in\mathcal{C}$. But this is true by (1).
\end{proof}

\section{Duality for presheaves}\label{S6}
\noindent In ordinary category theory, we have a Yoneda embedding $\mathfrak{Y}:\mathcal{C}\to\text{PSh}(\mathcal{C})$, which exhibits $\text{PSh}(\mathcal{C})$ as freely generated by $\mathcal{C}$ under colimits.

In other words, if $\mathcal{D}$ is a presentable category, then restriction along $\mathfrak{Y}$ induces an equivalence of categories $$\text{Fun}^L(\text{PSh}(\mathcal{C}),\mathcal{D})\to\text{Fun}(\mathcal{C},\mathcal{D}).$$ In this section, we will prove the analogous statement for enriched $\infty$-categories. We will always assume the enriched $\infty$-category $\mathcal{V}$ is presentable and closed monoidal.

Because $\text{PSh}^\mathcal{V}(\mathcal{C})$ is a right $\mathcal{V}$-module, and not a priori a $\mathcal{V}$-enriched category, we will not want to speak of $\mathcal{V}$-enriched functors to $\mathcal{D}$, but rather $\mathcal{V}$-enriched copresheaves with values in $\mathcal{D}$. (We introduced enriched copresheaves in Section \ref{S51}.)

We will construct the Yoneda embedding in the guise of a copresheaf $\mathfrak{Y}\in\text{coPSh}^\mathcal{V}(\mathcal{C};\text{PSh}^\mathcal{V}(\mathcal{C}))$, and then we will prove:

\begin{theorem}\label{ThmS6b}
If $\mathcal{N}$ is a presentable right $\mathcal{V}$-module, and $F:\text{PSh}^\mathcal{V}(\mathcal{C})\to\mathcal{N}$ is a colimit-preserving right $\mathcal{V}$-module functor, let $\mathfrak{Y}_\ast(F)$ denote the pushforward of $\mathfrak{Y}$ along $F_\ast:\text{coPSh}^\mathcal{V}(\mathcal{V};\text{PSh}^\mathcal{V}(\mathcal{C}))\to\text{coPSh}^\mathcal{V}(\mathcal{V};\mathcal{N})$. Then $$\mathfrak{Y}_\ast:\text{Fun}^L_{\text{RMod}_\mathcal{V}}(\text{PSh}^\mathcal{V}(\mathcal{C}),\mathcal{N})\to\text{coPSh}^\mathcal{V}(\mathcal{V};\mathcal{N})$$ is an equivalence of $\infty$-categories.
\end{theorem}

\noindent This theorem is essentially equivalent to Theorem \ref{ThmS6} from the introduction, that $\text{PSh}^\mathcal{V}(\mathcal{C})$ and $\text{coPSh}^\mathcal{V}(\mathcal{C})$ are dual $\mathcal{V}$-modules. We will conclude Theorem \ref{ThmS6} at the end of this section.\\

\noindent We begin by constructing the Yoneda copresheaf.

Identify the $\mathcal{V}$-enriched category $\mathcal{C}$ with a marked functor $\text{Assoc}_S^\mathsection\to\smallint\mathcal{V}^\mathsection$, and consider the composite $$\text{LM}_S^\mathsection\times\text{RM}_S^\mathsection\xrightarrow{\left\langle -,-\right\rangle}\text{Assoc}_{S\amalg S}^\mathsection\xrightarrow{\nabla}\text{Assoc}_S^\mathsection\xrightarrow{\mathcal{C}}\smallint\mathcal{V}^\mathsection,$$ where $\nabla:S\amalg S\to S$ is the identity on each component (the codiagonal), and $\left\langle -,-\right\rangle$ is the pairing of Section \ref{S52}. Because $\text{LM}_S\times\text{RM}_S\to\text{Assoc}_S$ is natural in $S$, the following diagram commutes: $$\xymatrix{
\text{LM}_S\times\text{RM}_S\ar[rr]^-{\nabla\left\langle -,-\right\rangle}\ar[d] &&\text{Assoc}_S\ar[r]^-{\mathcal{C}}\ar[rd] &\smallint\mathcal{V}\ar[d] \\
\text{LM}\times\text{RM}\ar[rrr]_-{\nabla\left\langle -,-\right\rangle} &&&\text{Assoc}.
}$$ Therefore, there is an induced marked functor $$\text{LM}_S^\mathsection\times\text{RM}_S^\mathsection\to\smallint^\prime\mathcal{V}^\mathsection=\smallint\mathcal{V}\times_{\text{Assoc}}(\text{LM}\times\text{RM}).$$ By adjunction, we have $\mathfrak{Y}:\text{RM}_S\to\widebar{\text{PSh}}^\mathcal{V}_S=\text{Fun}^\dag_{/\text{LM}}(\text{LM}_S^\mathsection,\smallint^\prime\mathcal{V}^{!})$, which is compatible with the functors down to RM: $$\xymatrix{
\text{RM}_S\ar[r]^-{\mathfrak{Y}}\ar[rd] &\widebar{\text{PSh}}^\mathcal{V}_S\ar[d]^-{p_0} \\
&\text{RM}.
}$$ Moreover, $\mathfrak{Y}$ sends inert morphisms to $p_0$-cocartesian morphisms by Lemma \ref{LemBad}, so therefore it describes a copresheaf. Since $\widebar{\text{PSh}}^\mathcal{V}_S\cong\smallint(\mathcal{V};\text{PSh}^\mathcal{V}_S)$ by Lemma \ref{LemLast}, $\mathfrak{Y}\in\text{coPSh}^{\mathcal{V};\text{PSh}^\mathcal{V}_S}_S$. We will see shortly:
\begin{itemize}
\item that the underlying enriched category is $\mathcal{C}$, so that $\mathfrak{Y}\in\text{coPSh}^\mathcal{V}(\mathcal{C};\text{PSh}^\mathcal{V}_S)$;
\item that the copresheaf evaluated at any object of $\mathcal{C}$ is an enriched presheaf on $\mathcal{C}$, so that $\mathfrak{Y}\in\text{coPSh}^\mathcal{V}(\mathcal{C};\text{PSh}^\mathcal{V}(\mathcal{C}))$.
\end{itemize}
\noindent We will regard $\mathfrak{Y}$ as the enriched Yoneda embedding. Now we prove the two points in the next two lemmas; they amount to the observations that $$\{(X,\ast)\}\times\text{Assoc}_S\subseteq\text{LM}_S\times\text{RM}_S\xrightarrow{\nabla\left\langle -,-\right\rangle}\text{Assoc}_S,$$ $$\text{Assoc}_S\times\{(\ast,X)\}\subseteq\text{LM}_S\times\text{RM}_S\xrightarrow{\nabla\left\langle -,-\right\rangle}\text{Assoc}_S$$ are each the identity functor by construction of $\left\langle -,-\right\rangle$.

\begin{lemma}\label{Lem61}
$\mathfrak{Y}$ is a copresheaf on $\mathcal{C}$; that is, $\mathfrak{Y}\in\text{coPSh}^\mathcal{V}(\mathcal{C};\text{PSh}^\mathcal{V}_S)$.
\end{lemma}

\begin{proof}
The underlying enriched category of $\mathfrak{Y}$ is described by the restriction $\text{Assoc}_S\subseteq\text{RM}_S\xrightarrow{\mathfrak{Y}}\widebar{\text{PSh}}^\mathcal{V}_S$. Since $\widebar{\text{PSh}}^\mathcal{V}_S\cong\smallint(\mathcal{V};\text{PSh}^\mathcal{V}_S)$, this composite is just a $\mathcal{V}$-enriched category.

For any object $X\in S$, recall that $\text{ev}_X:\text{PSh}^\mathcal{V}_S\to\mathcal{V}$ is a right $\mathcal{V}$-module functor, described explicitly (as in the proof of Proposition \ref{PropEv1}) as $$\widebar{\text{ev}}_X:\widebar{\text{PSh}}^\mathcal{V}_S=\text{Fun}^\dag_{/\text{LM}}(\text{LM}_S^\mathsection,\smallint^\prime\mathcal{V}^{!})\to\smallint\mathcal{V},$$ which is evaluation at $(X,\ast)\in\text{LM}_S$. As a right $\mathcal{V}$-module functor, $\widebar{\text{ev}}_X$ restricts to an equivalence on $\widebar{\text{PSh}}^\mathcal{V}_S\times_{\text{RM}}\text{Assoc}\cong\smallint\mathcal{V}$.

In particular, this means that the composite $$\text{Assoc}_S\subseteq\text{RM}_S\xrightarrow{\mathfrak{Y}}\widebar{\text{PSh}}^\mathcal{V}_S\xrightarrow{\widebar{\text{ev}}_X}\smallint\mathcal{V},$$ which is a priori a $\mathcal{V}$-enriched category, actually recovers the underlying $\mathcal{V}$-enriched category of the copresheaf $\mathfrak{Y}$ (regardless of which $X\in S$ was chosen).

By construction of $\mathfrak{Y}$, this composite is also $$\{(X,\ast)\}\times\text{Assoc}_S\subseteq\text{LM}_S\times\text{RM}_S\xrightarrow{\left\langle -,-\right\rangle}\text{Assoc}_S\xrightarrow{\mathcal{C}}\smallint\mathcal{V},$$ but $\{(X,\ast)\}\times\text{Assoc}_S\to\text{Assoc}_S$ is the identity by construction of $\left\langle -,-\right\rangle$, so this composite recovers the enriched category $\mathcal{C}$, completing the proof.
\end{proof}

\begin{lemma}\label{Lem62}
The following square commutes: $$\xymatrix{
\text{RM}_S\ar[r]^-{\mathfrak{Y}}\ar[d] &\widebar{\text{PSh}}^\mathcal{V}_S\ar[d]^-{\bar{\theta}} \\
\text{RM}\ar[r]_-{\mathcal{C}_\ast} &\widebar{\text{Cat}}^\mathcal{V}_S.
}$$
\end{lemma}

\begin{proof}
Since $\widebar{\text{Cat}}^\mathcal{V}_S\cong\smallint(0;\text{Cat}^\mathcal{V}_S)$, each composite $\text{RM}_S\to\widebar{\text{Cat}}^\mathcal{V}_S$ is a copresheaf in $\text{coPSh}^{0;\text{Cat}^\mathcal{V}_S}_S$. Each evaluation map $\text{ev}_X:\text{coPSh}^{0;\text{Cat}^\mathcal{V}_S}_S\to\text{Cat}^\mathcal{V}_S$ is an equivalence by Corollary \ref{CorTriv}.

Therefore, it suffices to show (for a single chosen $X\in S$) that the two composites $\text{RM}_S\to\widebar{\text{Cat}}^\mathcal{V}_S$ are equivalent when evaluated at $(\ast,X)\in\text{RM}_S$. The lower composite evaluated at $(\ast,X)$ is $\mathcal{C}$, by construction, while the upper composite evaluated at $(\ast,X)$ is the enriched category $$\text{Assoc}_S\times\{(\ast,X)\}\subseteq\text{LM}_S\times\text{RM}_S\xrightarrow{\left\langle -,-\right\rangle}\text{Assoc}_S\xrightarrow{\mathcal{C}}\smallint\mathcal{V}.$$ As in the proof of the last lemma, the composite $\text{Assoc}_S\times\{(\ast,X)\}\to\text{Assoc}_S$ is the identity, so this enriched category is also $\mathcal{C}$, completing the proof.
\end{proof}

\begin{remark}
By Lemma \ref{Lem62}, $\mathfrak{Y}$ factors through a marked functor of the form $\text{RM}_S\to\widebar{\text{PSh}}^\mathcal{V}(\mathcal{C})=\widebar{\text{PSh}}^\mathcal{V}_S\times_{\widebar{\text{Cat}}^\mathcal{V}_S}\text{RM}$, which is to say a copresheaf with values in $\text{PSh}^\mathcal{V}(\mathcal{C})$. By Lemma \ref{Lem61}, the underlying enriched category of this copresheaf is $\mathcal{C}$, so that we have a \emph{Yoneda copresheaf} $$\mathfrak{Y}\in\text{coPSh}^\mathcal{V}(\mathcal{C};\text{PSh}^\mathcal{V}(\mathcal{C})).$$
\end{remark}

\noindent Now we will prove Theorem \ref{ThmS6b}, that the functor $$\mathfrak{Y}_\ast:\text{Fun}^L_{\text{RMod}_\mathcal{V}}(\text{PSh}^\mathcal{V}(\mathcal{C}),\mathcal{N})\to\text{coPSh}^\mathcal{V}(\mathcal{V};\mathcal{N})$$ is an equivalence of $\infty$-categories.

\begin{proof}[Proof of Theorem \ref{ThmS6b}]
The proof is by Barr-Beck, following HA 4.8.4.1 and very similar to Theorem \ref{ThmS5}. As $S\in X$ varies, the right $\mathcal{V}$-module functors $\text{rep}_X\otimes -:\mathcal{V}\to\text{PSh}^\mathcal{V}(\mathcal{C})$ assemble into a $\mathcal{V}$-module functor $\mathcal{V}^{\times S}\to\text{PSh}^\mathcal{V}(\mathcal{C})$. Precomposition with this functor induces $$\text{Fun}^L_{\text{RMod}_\mathcal{V}}(\text{PSh}^\mathcal{V}(\mathcal{C}),\mathcal{N})\to\text{Fun}^L_{\text{RMod}_\mathcal{V}}(\mathcal{V}^{\times S},\mathcal{N})\cong\mathcal{N}^S,$$ and the following triangle commutes (where $T^\prime$ denotes evaluation at each $X\in S$) $$\xymatrix{
\text{Fun}^L_{\text{RMod}_\mathcal{V}}(\text{PSh}^\mathcal{V}(\mathcal{C}),\mathcal{N})\ar[rr]^-{\mathfrak{Y}_\ast}\ar[rd]_-{T} &&\text{coPSh}^\mathcal{V}(\mathcal{C};\mathcal{N})\ar[ld]^-{T^\prime} \\
&\mathcal{N}^S,&
}$$ because $\mathfrak{Y}_\ast(F)(X)\cong F(\text{rep}_X)$. We proved $T^\prime$ is monadic in the proof of Theorem \ref{ThmS5} (end of Section \ref{S5}). We claim $T$ is also monadic.

Indeed, $T$ preserves colimits by construction (as it is a map of $\text{Pr}^L$), and it has a left adjoint $U$ by Corollary \ref{CorEv1}, which is given by precomposition with $\text{ev}:\text{PSh}^\mathcal{V}(\mathcal{C})\to\mathcal{V}^{\times S}$. To finish the proof that $T$ is monadic, we need only show $T$ is conservative. Let $F,G:\text{PSh}^\mathcal{V}(\mathcal{C})\to\mathcal{N}$ be two right $\mathcal{V}$-module functors and $\eta:F\to G$ a natural transformation such that $T(\eta)$ is an equivalence. In other words, $\eta$ is an equivalence at $\text{rep}_X$ for all $X\in S$. But since $\text{PSh}^\mathcal{V}(\mathcal{C})$ is generated by $\text{rep}_X$ as a right $\mathcal{V}$-module in $\text{Pr}^L$ (Corollary \ref{CorFreeCo2}), it follows that $\eta$ is an equivalence everywhere.

Therefore, $T$ and $T^\prime$ are monadic. By HA 4.7.3.16-17, to complete the proof that $\mathfrak{Y}_\ast$ is an equivalence, it suffices to show that the induced natural transformation $T^\prime U^\prime\to TU$ of functors $\mathcal{N}^S\to\mathcal{N}^S$ is an equivalence, where $U,U^\prime$ are the left adjoints to $T$, respectively $T^\prime$.

Since $T^\prime U^\prime$ and $TU$ preserve colimits, it suffices to check for each $X,Y\in S$ that $T^\prime_X U^\prime_Y\to T_XU_Y$ is a natural equivalence of functors $\mathcal{N}\to\mathcal{N}$. But both these functors are given by $-\otimes\mathcal{C}(Y,X)$, and unpacking, the map is an equivalence. This completes the proof.
\end{proof}

\noindent Finally, we will conclude Theorem \ref{ThmS6}, that $\text{PSh}^\mathcal{V}(\mathcal{C})\in\text{RMod}_\mathcal{V}(\text{Pr}^L)$ is left dual to $\text{coPSh}^\mathcal{V}(\mathcal{C})\in\text{LMod}_\mathcal{V}(\text{Pr}^L)$ (in the sense of HA 4.6.2.3).

\begin{proof}[Proof of Theorem \ref{ThmS6}]
We have $\text{coPSh}^\mathcal{V}(\mathcal{C};\text{PSh}^\mathcal{V}(\mathcal{C}))\cong\text{PSh}^\mathcal{V}(\mathcal{C})\otimes_\mathcal{V}\text{coPSh}^\mathcal{V}(\mathcal{C})$ by Theorem \ref{ThmS5}, so we may regard $\mathfrak{Y}$ as an object of this tensor product, or a colimit-preserving functor $$-\otimes\mathfrak{Y}:\text{Top}\to\text{PSh}^\mathcal{V}(\mathcal{C})\otimes_\mathcal{V}\text{coPSh}^\mathcal{V}(\mathcal{C}).$$ Here the $\infty$-category Top of spaces is the unit of the monoidal structure on $\text{Pr}^L$. By HA 4.6.2.18, we just need to prove that for each $\mathcal{D}\in\text{Pr}^L$ and $\mathcal{N}\in\text{RMod}_\mathcal{V}(\text{Pr}^L)$, the composite $$\xymatrix{
\text{iFun}_{\text{RMod}_\mathcal{V}}^L(\mathcal{D}\otimes\text{PSh}^\mathcal{V}(\mathcal{C}),\mathcal{N})\ar[d]^-{-\otimes_\mathcal{V}\text{coPSh}^\mathcal{V}(\mathcal{C})} \\
\text{iFun}^L(\mathcal{D}\otimes\text{PSh}^\mathcal{V}(\mathcal{C})\otimes_\mathcal{V}\text{coPSh}^\mathcal{V}(\mathcal{C}),\mathcal{N}\otimes_\mathcal{V}\text{coPSh}^\mathcal{V}(\mathcal{C}))\ar[d]^-{\mathfrak{Y}} \\
\text{iFun}^L(\mathcal{D},\mathcal{N}\otimes_\mathcal{V}\text{coPSh}^\mathcal{V}(\mathcal{C}))
}$$ is an equivalence of spaces ($\infty$-groupoids), where iFun denotes the maximal subgroupoid of the functor $\infty$-category. However, we have equivalences $$\text{Fun}^L_{\text{RMod}_\mathcal{V}}(\mathcal{D}\otimes\text{PSh}^\mathcal{V}(\mathcal{C}),\mathcal{N})\to\text{Fun}^L(\mathcal{D},\text{coPSh}^\mathcal{V}(\mathcal{C};\mathcal{N})),$$ $$\text{Fun}^L(\mathcal{D},\mathcal{N}\otimes_\mathcal{V}\text{coPSh}^\mathcal{V}(\mathcal{C}))\to\text{Fun}^L(\mathcal{D},\text{coPSh}^\mathcal{V}(\mathcal{C};\mathcal{N})),$$ the first by Theorem \ref{ThmS6b} and the second by Theorem \ref{ThmS5}, and (unpacking) they are compatible with the functor above. This establishes duality.
\end{proof}

\end{document}